\documentclass[11pt]{amsart}
\usepackage{amssymb}
\usepackage{euscript}
\usepackage[T1]{fontenc}
\usepackage[utf8]{inputenc}
\usepackage[french]{babel}
\usepackage{graphicx}

\normalbaselines
\font\frten=eufm10 at 10pt
\font\freight=eufm10
\font\frsix=eufm8
\newfam\frfam\textfont\frfam=\frten
\scriptfont\frfam=\freight
\scriptscriptfont\frfam=\frsix

\newcommand{\CC}{\mathbf{C}}

\newcommand{\RR}{\mathbf{R}}
\newcommand{\NN}{\mathbf{N}}

\theoremstyle{plain}
\newtheorem{thm}{Théorème}[section]
\newtheorem{prop}[thm]{Proposition}
\newtheorem{cor}[thm]{Corollaire}
\newtheorem{lem}[thm]{Lemme}

\theoremstyle{definition}
\newtheorem{dfn}[thm]{Définition}

\theoremstyle{remark}
\newtheorem{rema}[thm]{Remarque}
\newtheorem{rem}[thm]{Remarque}

\newtheorem{ton}[thm]{Notation}

\newtheorem{node}[thm]{Notations et Définitions}
\newtheorem{exem}[thm]{Exemples}
\newtheorem{ex}[thm]{Exemple}

\newtheorem*{ack}{Remerciements}{}      

\newcommand{\tensor}{\otimes}
\newcommand{\iso}{\cong}
\newcommand{\et}{\quad\textrm{et}\quad}
\newcommand{\suchthat}{\mid}
\newcommand{\lra}{\longrightarrow}
\newcommand{\ra}{{\to}}
\newcommand{\union}{\cup}
\newcommand{\inter}{\cap}

\newcommand{\Rad}{\operatorname{Rad}}

\newcommand{\codim}{\operatorname{codim}}

\newcommand{\R}{\mathbf{R}}

\renewcommand{\AA}{\mathbf{A}} % was: angstrom symbol
\newcommand{\AR}{\mathcal{A}\,\mathcal{R}}
\newcommand{\C}{\mathbf{C}}
\newcommand{\SC}{\mathcal{C}}
\newcommand{\D}{\mathcal{D}}
\newcommand{\SI}{\mathcal{I}}
\newcommand{\SO}{\mathcal{O}}
\renewcommand{\P}{\mathbf{P}}

\newcommand{\SQ}{\mathcal{Q}}

\newcommand{\N}{\mathbf{N}}
\newcommand{\SF}{\mathcal{F}}

\newcommand{\SRC}{\mathcal{R}^0}
\newcommand{\SR}{\mathcal{R}}
\newcommand{\U}{\mathcal{U}}
\newcommand{\V}{\mathcal{V}}

\let \Lbarre =\L
\renewcommand{\L}{\mathcal{O}}

\newcommand{\I}{\mathcal{I}}
\newcommand{\Z}{\mathcal{Z}}

\newcommand{\Spec}{\operatorname{Spec}}

\renewcommand{\pm}{\mathfrak{m}}
\newcommand{\M}{\mathfrak{M}}
\newcommand{\dom}{\mathrm{dom}}
\newcommand{\pol}{\mathrm{indet}}

\begin{document}

\title{Fonctions Régulues}

\author[G.~Fichou, J.~Huisman, F.~Mangolte, J.-P.~Monnier]{Goulwen Fichou, Johannes Huisman, Frédéric Mangolte  Jean-Philippe Monnier}

\address{Goulwen Fichou\\
IRMAR (UMR 6625), Universit\'e de
         Rennes 1\\
Campus de Beaulieu, 35042 Rennes Cedex, France}
\email{goulwen.fichou@univ-rennes1.fr}

\address{Johannes Huisman\\
LMBA (UMR 6205), Universit\'e de
Bretagne Occidentale\\
6, Av. Victor Le Gorgeu, CS 93837, 29238 Brest Cedex 3, France}
\email{Johannes.Huisman@univ-brest.fr}

\address{Frédéric Mangolte\\
   LUNAM Universit\'e, LAREMA, Universit\'e d'Angers} 
   \email{frederic.mangolte@univ-angers.fr}

\address{Jean-Philippe Monnier\\
   LUNAM Universit\'e, LAREMA, Universit\'e d'Angers}
\email{jean-philippe.monnier@univ-angers.fr}

\date{}

\maketitle

\begin{quote}\small
\end{quote}

\begin{abstract}
Nous étudions l'anneau des fonctions rationnelles qui se prolongent
par continuité sur $\RR^n$.  Nous établissons plusieurs propriétés algébriques de
cet anneau dont un Nullstellensatz fort.  Nous étudions les propriétés schématiques associées et montrons une version régulue des Théorèmes A et B de Cartan. Nous caractérisons
géométriquement les idéaux premiers de cet anneau à travers leurs
lieux d'annulation et montrons que les fermés régulus  coïncident avec les fermés algébriquement constructibles.
\begin{center}
\includegraphics[height =.5cm]{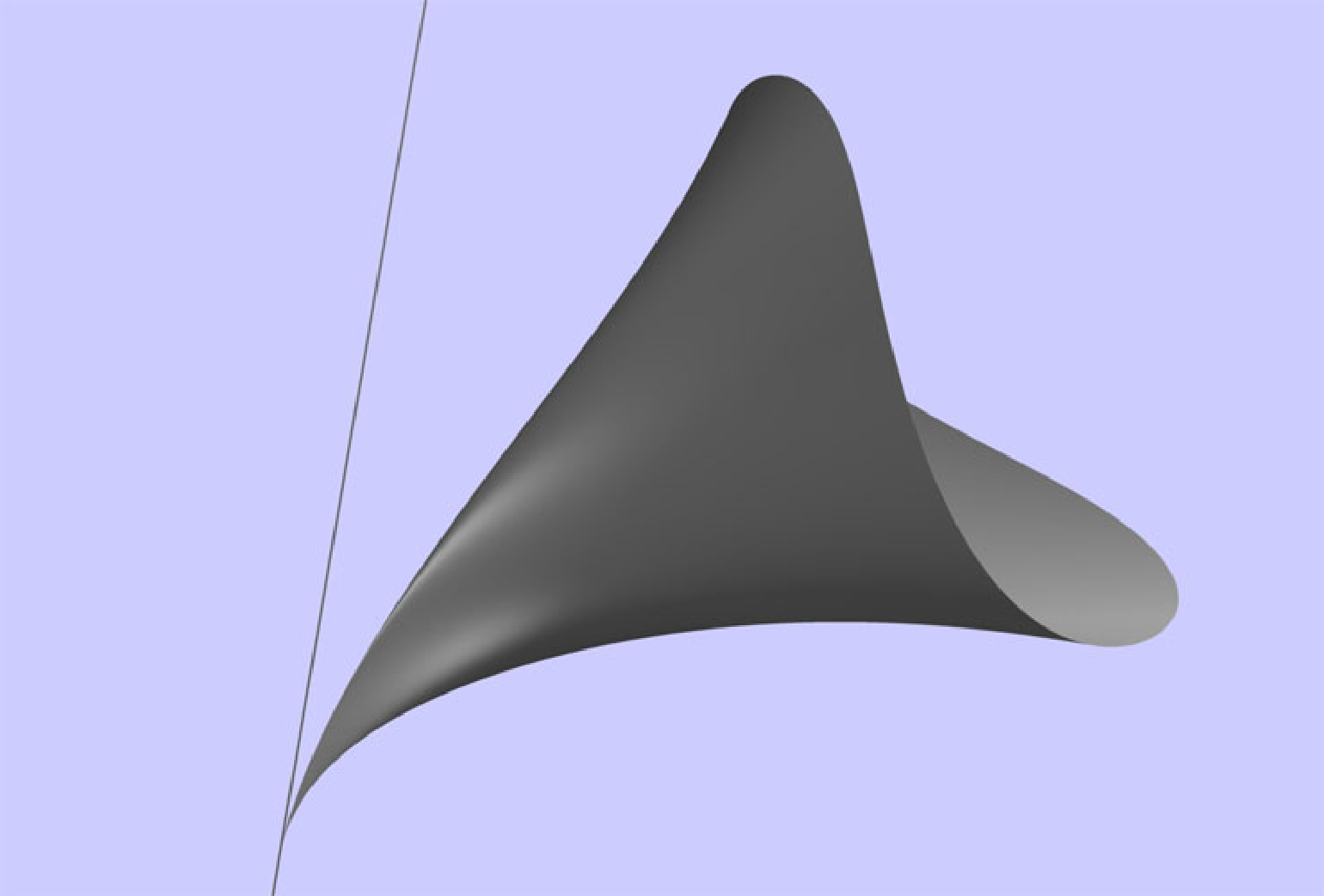}
\end{center}
\end{abstract}

\medskip
\renewcommand{\abstractname}{English title and abstract}
\begin{abstract} 

\textbf{Regulous functions}

We study the ring of rational functions admitting a continuous extension to the real affine space. We establish several properties of this ring. In particular, we prove a strong Nullstelensatz. We study the scheme theoretic properties and prove regulous versions of Theorems A and B of Cartan. We also give a geometrical characterization of prime ideals of this ring in terms of their zero-locus and relate them to euclidean closed Zariski-constructible sets.
 \end{abstract}
%%%%%%%%%%%%

\setcounter{tocdepth}{1}
\tableofcontents

\newpage
\section{Introduction}

Lorsqu'on développe la géométrie réelle algébrique\footnote{Nous ferons la
  distinction ici entre la «géométrie algébrique réelle» et la
  «géométrie réelle algébrique»; cette dernière étudie les variétés
  réelles possédant une structure algébrique, alors que la première
  étudie les variétés algébriques complexes munies d'une structure
  réelle.} selon le
modèle de la géométrie algébrique complexe, des
obstacles techniques apparaissent. L'un des exemples le plus
emblématique en étant que les versions algébriques complexes des
Théorèmes A~et B de Cartan (cf.~\cite[2.4]{FAC}) n'ont pas leurs
analogues en géométrie réelle algébrique. Dans la théorie régulière telle que développée
dans~\cite[Chapitre~3]{BCR}, aucun des Théorèmes A~et B de Cartan n'est valable\footnote{Voir~\cite[Exemple~12.1.5]{BCR} pour un
  contre-exemple au Théorème~A, et \cite[Theorem~1]{CosteDiop} pour un
  contre-exemple au Théorème~B en géométrie réelle régulière.}. De même
à travers l'approche Nash de la géométrie réelle algébrique, les Théorèmes
A~et B de Cartan ne sont pas vérifiés\footnote{Voir~\cite{Hubbard}
  pour des contre-exemples en géométrie réelle Nash.}. 
Dans cet article, nous proposons
une approche nouvelle de la géométrie réelle algébrique, basée sur les
fonction rationnelles continues, et plus généralement les fonctions
rationnelles de classe~$\SC^k$, qu'on appellera $k$-régulues. Pour ces
fonctions, qui constituent un assouplissement des fonctions régulières
lorsque $k<\infty$, les analogues des Théorèmes A~et B de Cartan
deviennent valables (cf. \ref{thAX} et \ref{thBX}). 

Avant de donner un
aperçu des principaux résultats de cet article, il convient de
rappeler certaines notions de géométrie réelle algébrique.
Suivant un abus courant en géométrie algébrique, nous dirons qu'une
\emph{fonction rationnelle} sur $\RR^n$ (cf.~\ref{dfn.ration} pour une définition formelle) est une fonction à valeurs
réelles~$f$ définie sur un ouvert de Zariski non vide~$U$ de~$\R^n$
telle qu'il existe des polynômes~$p,q\in \RR
[x_1,\ldots,x_n]$ pour lesquels
$$
f(x)=\frac{p(x)}{q(x)}
$$ pour chaque $x\in U$. Cette écriture fractionnaire sous-entend que
la fonction polynomiale~$q$ ne s'annule pas sur~$U$. Deux fonctions
rationnelles sur~$\R^n$ sont alors égales si elles coïncident sur un
ouvert de Zariski non vide contenu dans leurs domaines de définition.
Une fonction rationnelle~$f$ sur~$\R^n$ possède un plus grand ouvert de
Zariski sur lequel elle est définie
(cf.~\cite[Proposition~3.2.3]{BCR}). Cet ouvert est le
\emph{domaine} de~$f$, il est noté~$\dom(f)$. 
Nous noterons $\pol(f)$ le complémentaire du domaine de~$f$,
c'est le \emph{lieu d'indétermination} de~$f$.  
A titre d'exemple, le domaine de la fonction
rationnelle~$f$ sur~$\R$ définie par $f(x)=1/x$ est égal
à~$\R\setminus\{0\}$ et~$\pol(f)=\{0\}$.  L'ensemble des
fonctions rationnelles sur~$\R^n$ est un corps et s'identifie au corps
des fractions rationnelles~$\R(x_1,\ldots,x_n)$, il peut aussi être noté~$\R(\R^n)$.

Une fonction rationnelle~$f$ sur~$\R^n$ est une \emph{fonction
  régulière} sur~$\R^n$ si elle est définie sur~$\R^n$ tout entier,
i.e., si~$\dom(f)=\R^n$ \cite[Définition~3.2.1]{BCR}. A titre
d'exemple, la fonction rationnelle~$f(x)=1/(x^2+1)$ est régulière
sur~$\R$.  L'ensemble des fonctions régulières sur~$\R^n$ est un
sous-anneau du corps des fonctions rationnelles sur~$\R^n$,
que nous noterons~$\SR^\infty(\R^n)$. Nous justifierons plus loin ce choix de notation.

Dans ce travail, nous nous proposons donc d'étudier les fonctions rationnelles sur~$\R^n$
qui s'étendent par continuité à~$\R^n$ tout entier. Par
«continuité» nous entendons ici la continuité par rapport à la
topologie euclidienne. Remarquons tout de suite qu'une fonction rationnelle
sur~$\R^n$ étant continue sur son domaine de définition, il s'agit
de fonctions qui s'étendent par continuité à leur lieu
d'indétermination. Plus précisément, une
\emph{fonction régulue} sur $\RR^n$ est une fonction à valeurs réelles
définie en tout point de~$\R^n$, qui est continue pour la topologie
euclidienne et qui est rationnelle sur $\R^n$.
A titre d'exemple, la fonction
régulière
$$
f(x,y)=\frac{x^3}{x^2+y^2}
$$ 
sur~$\R^2\setminus\{0\}$ s'étend par
continuité en l'origine et définit donc une fonction régulue sur
$\RR^2$. Son graphe est la toile du célèbre parapluie de Cartan
(voir~\ref{exem.umbrellas}).  L'ensemble des fonctions régulues
sur~$\R^n$ est un sous-anneau du corps~$\R(\R^n)$ des fonctions
rationnelles sur~$\R^n$, que nous notons~$\SR^0(\RR^n)$. Une
fonction régulière sur~$\R^n$ étant évidemment  régulue, nous obtenons une chaîne de
sous-anneaux
$$
\SR^\infty(\R^n)\subseteq\SRC(\R^n)\subseteq\R(\R^n).
$$

Plus généralement, une fonction sur $\RR^n$ est
\emph{$k$-régulue}, si elle est à la fois régulière sur un ouvert de
Zariski non vide, et de classe~$\SC^k$ sur~$\R^n$. Ici, $k$ désigne un
\emph{entier surnaturel}, i.e., $k$ est ou bien un entier naturel, ou
bien $k$ est égal à~$\infty$.  A titre d'exemple, la fonction
régulière
$$f(x,y)=\frac{x^{3+k}}{x^2+y^2}$$ sur~$\R^2\setminus\{0\}$ s'étend
par continuité en l'origine et définit une fonction $k$-régulue
sur $\RR^2$, si $k$ est un entier naturel. Nous démontrons
(cf. Théorème~\ref{thinfregestreg}) qu'une fonction $\infty$-régulue
sur~$\R^n$ est nécessairement régulière. Pour $k$ un entier surnaturel, l'ensemble des fonctions
$k$-régulues est un sous-anneau du corps~$\R(\R^n)$ des fonctions
rationnelles sur~$\R^n$, qui sera noté~$\SR^k(\R^n)$. Remarquons qu'il n'y a pas de conflit de notation ni
avec l'anneau des fonctions régulues~$\SRC(\R^n)$, ni avec l'anneau de
fonctions régulières~$\SR^\infty(\R^n)$ introduits ci-dessus.  Nous obtenons finalement 
une chaîne de sous-anneaux
$$
\SR^\infty(\R^n)\subseteq\cdots\subseteq\SR^2(\R^n)\subseteq\SR^1(\R^n)
\subseteq\SRC(\R^n)\subseteq\R(\R^n).
$$ Le plus petit de ses sous-anneaux est égal à l'intersection de tous
les autres sous-anneaux de la chaîne, i.e.,
$$
\SR^\infty(\R^n)=\bigcap_{k\in\N}\SR^k(\R^n).
$$

\medskip
Revenons au contenu de cet article.
Dans un premier temps nous déterminons les propriétés algébriques de l'anneau~$\SR^k(\R^n)$
des fonctions $k$-régulues sur~$\R^n$, où $k$ est un entier
naturel. Ces anneaux ont été assez peu étudiés~; les seules références qui
nous sont connues étant~\cite{Ku,Ko}. L'anneau~$\SR^\infty(\R^n)$ des
fonctions régulières sur~$\R^n$, en revanche, a attiré beaucoup
d'attention~\cite{BCR}.
Nous montrons que~$\SR^k(\R^n)$ est un anneau
non-noethérien pour lequel le Nullstellensatz est
valable~(\ref{Nullstellensatz}). Cela est d'autant plus remarquable
que l'intersection de tous ces anneaux, à savoir~$\SR^\infty(\R^n)$,
est un anneau noethérien pour lequel le Nullstellensatz n'est pas
valable~!  Une version affaiblie du Nullstellensatz est néanmoins vraie
pour $\SR^\infty(\R^n)$, cette version fait intervenir le radical
réel d'un idéal de fonctions régulières~\cite[\S~4.4]{BCR}. Ce
Nullstellensatz réel, bien qu'intéressant en lui-même, ne répare en rien
le défaut accablant de l'anneau des fonctions régulières qui est de posséder
trop d'id\'eaux premiers, ce qui ne manque pas de poser des problèmes en
géométrie réelle régulière.  Nous verrons, en revanche, que la géométrie
réelle $k$-régulue, pour $k$ fini, ne pose pas ces problèmes, et,
de ce fait, se rapproche plus de la géométrie algébrique sur un corps
algébriquement clos que la géométrie réelle régulière.

Malgré le caractère non noethérien de l'anneau~$\SR^k(\R^n)$, son usage reste raisonnable en géométrie algébrique car nous démontrons que son spectre de
Zariski~$\Spec\SR^k(\R^n)$ est un espace topologique noethérien. De
manière équivalente (grâce à la validité du Nullstellensatz justement)
l'ensemble~$\R^n$ muni de la topologie $k$-régulue est un espace
topologique noethérien, lorsque $k$ est fini
(cf.~\ref{topregnoeth}). Ici et dans toute la suite de cet article, la «topologie $k$-régulue» est la
topologie dont une base de fermés est la collection des sous-ensembles
de la forme
$$
\Z(f)=\{x\in\R^n\suchthat f(x)=0\},
$$ 
où $f$ est $k$-régulue sur~$\R^n$. Cette topologie est strictement
plus fine que la topologie de Zariski sur~$\R^n$,
lorsque~$n\geq2$. Cette dernière pouvant également être dénommée
topologie régulière ou $\infty$-régulue.

Dans un deuxième temps, nous posons les bases de l'étude des
variétés régulues abstraites. Nous revenons sur un fibré en droites
régulier pathologique sur~$\RR^2$ bien connu qui n'est pas engendré
par ses sections régulières globales~\cite[Example~12.1.5]{BCR}, et
montrons que cette pathologie disparait si ce fibré est interprété comme fibré
en droites régulu. Nous en déduisons alors la validité des 
Théorèmes A~et B de Cartan dans le cadre régulu (\ref{thAX} et \ref{thBX}).

Nous concluons par un chapitre consacré à la caractérisation géométrique des fermés régulus de $\RR^n$ et montrons un résultat fondamental : les fermés régulument irréductibles de $\RR^n$    coïncident avec les sous-ensembles algébriquement constructibles fermés de $\RR^n$ (Théorème.~\ref{const}).
Dans le cas des courbes et des surfaces, nous poussons notre étude et proposons
en particulier une relecture \emph{régulue} des fameux parapluies de
Cartan, de Whitney, et d'un parapluie de Koll\'ar. Nous introduisons aussi un nouveau parapluie
\emph{cornu} (\ref{exem.umbrellas}).

\medskip
A notre connaissance, les fonctions régulues ont été étudiées de
façon systématique pour la première fois par Kucharz dans \cite{Ku}
(où elles sont appelées \emph{continuous rational}). Dans son article,
Kucharz montre que ce sont les bonnes fonctions pour approcher le plus
algébriquement possible les fonctions continues. Il démontre notamment
que toute classe d'homotopie d'une application continue entre deux
sphères de dimensions quelconques contient une application régulue, là
où les applications polynomiales et régulières font défaut~! 

Dans \cite{Ko}, Koll\'{a}r étudie les problèmes de restriction et d'extension
de fonctions continues rationnelles définies sur une variété réelle
algébrique affine.

Lorsque la variété est lisse, notre notion de
fonction régulue (cf. Définition~\ref{dfn.regulue}) co\"\i ncide avec
la notion de fonction continue rationnelle de Koll\'ar et de Kucharz.  
Lorsque la variété est singulière, notre notion de
fonction régulue coïncide avec ce que Koll\'ar appelle
``héréditairement rationnelle continue''
(cf. Remarque~\ref{rem.here}). Mentionnons qu'une nouvelle version \cite{KN} de l'article \cite{Ko}, en collaboration avec K.~J.~Nowak contient une version plus forte du théorème d'extension \cite[Proposition~10]{KN}, donnant un contr\^ole optimal sur la r\'egularit\'e de l'extension, ceci en utilisant des propri\'et\'es \'el\'ementaires des ensembles semi-alg\'ebriques.

Objets naturels, les fonctions régulues apparaissent aussi dans des
résultats antérieurs. En 1978, voir \cite[p.~369]{Del},  Kreisel remarque (sans employer le terme "fonction régulue" bien sûr) que
le Positivstellensatz de Stengle \cite{St} permet de représenter tout
polynôme $f\in \RR[x_1,\ldots,x_n]$ positif sur $\RR^n$ comme une
somme de carrés de fonctions régulues sur $\RR^n$.

Il est à noter que l'étude des isomorphismes $\infty$-régulus sur les surfaces a connu récemment des progrès importants, cf. e.g. \cite{bh,hm3,km1,bm1}.

Pour être complets, rappelons que sur $\CC$, toute fonction rationnelle continue sur une variété \emph{normale} est régulière. Dans le cas d'une variété singulière générale, l'étude des fonctions rationnelles continues amène aux concepts de semi-normalité et de semi-normalisation, cf. \cite{AN67,AB69}.

\begin{ack} Nous remercions J.~Koll\'ar pour nous avoir transmis une version préliminaire de son article, ainsi que S.~Cantat, M.~Coste, L.~Evain, W.~Kucharz, K.~Kurdyka, D.~Naie et A.~Parusi\'nski pour l'intérêt précoce qu'ils ont porté à nos travaux et pour leurs suggestions qui ont contribué à améliorer ce texte. Merci aussi à F.~Broglia et F.~Acquistapace pour nous avoir signalé les références \cite{AN67} et \cite{AB69}. La version finale de cet article doit beaucoup au referee dont la lecture attentive et les remarques ont été très constructives.
    
Ce travail a bénéficié d'un support partiel provenant du contrat ANR "BirPol"  ANR-11-JS01-004-01.
\end{ack}

\section{Variétés réelles algébriques et fonctions régulières}

\subsection*{Ensemble des zéros d'une fonction réelle}

Soit $n$ un entier naturel.  On définit, de manière générale,
l'ensemble des zéros et l'ensemble des non zéros d'une fonction réelle
sur~$\R^n$, ou d'un ensemble de fonctions réelles:

\begin{ton} Soit $f\colon \R^n\ra \R$ une fonction réelle sur~$\R^n$.
On note~$\Z(f)$ l'ensemble des z\'eros de~$f$ dans~$\RR^n$, i.e.,
$$ 
\Z(f)=\{x\in\RR^n|\,\,f(x)=0\}.
$$
L'ensemble des non-zéros de~$f$ est
$$
\D(f)=\{x\in\RR^n|\,\,f(x)\neq0\}.
$$ Soit $E$ un ensemble de fonctions réelles sur~$\R^n$. On
note~$\Z(E)$ l'ensemble des z\'eros communs des fonctions dans~$E$,
i.e.,
$$\Z(E)=\bigcap_{f\in
  E}\,\Z(f),
$$
\end{ton}

On a les propri\'et\'es habituelles suivantes~:

\begin{prop}
Soit $n$ un entier naturel.
\begin{enumerate}
\item $\Z(1)=\emptyset$ et $\Z(\emptyset)=\RR^n$.
\item Soit $E_\alpha$, $\alpha\in A$, une collection d'ensembles de
  fonctions réelles sur~$\R^n$. Alors
$$ 
\Z(\bigcup_{\alpha\in A} E_\alpha)= \bigcap_{\alpha\in
  A}\Z(E_\alpha).
$$
\item Soient $E_1,\ldots,E_m$ un nombre fini d'ensembles de fonctions
  réelles sur~$\R^n$, où $m$ est un entier naturel. Alors
$$ 
\Z(E_1\cdot\ldots\cdot E_m)=\Z(E_1)\union\cdots\union\Z(E_m).
$$
\item Soit $E$ un sous-ensemble d'un anneau~$A$ de fonctions réelles
  sur~$\R^n$. Alors $\Z(E)=\Z(I)$, o\`u $I$ est l'id\'eal de~$A$
  engendr\'e par~$E$.\qed
\end{enumerate}
\end{prop}

Un propriété moins habituelle, mais cruciale en géométrie réelle, est
la suivante:

\begin{prop}\label{prop.princip}
Soit $n$ un entier naturel. Soit $A$ un anneau de fonctions réelles
sur~$\R^n$, et soient $f_1,\ldots,f_m$ un nombre fini d'éléments
de~$A$.  Alors, il existe une fonction réelle~$f\in A$ telle que
$$
\Z(f)=\Z(f_1,\ldots,f_m).
$$
Plus précisément, la fonction~$f=f_1^2+\cdots+f_m^2$ convient.\qed
\end{prop}

Rappelons qu'un sous-ensemble $F$ de~$\R^n$ est un \emph{fermé de
  Zariski}, s'il existe un sous-ensemble~$E$ de l'anneau des fonctions
polynomiales réelles sur~$\R^n$ tel que
$$
\Z(E)=F.
$$ Comme ce dernier anneau est isomorphe à l'anneau noethérien des
polynômes réels~$\R[x_1,\ldots,x_n]$, un sous-ensemble~$F$ de~$\R^n$
est un fermé de Zariski si et seulement s'il existe un nombre fini de
fonctions polynomiales~$f_1,\ldots,f_m$ sur~$\R^n$ telles que
$$
\Z(f_1,\ldots,f_m)=F.
$$ En appliquant la proposition précédente \ref{prop.princip}, on pourra encore dire
qu'un sous-ensemble~$F$ de~$\R^n$ est un fermé de Zariski si et
seulement s'il existe une fonction polynomiale~$f$ sur~$\R^n$ telle
que
$$
\Z(f)=F.
$$

Les fermés de Zariski de~$\R^n$ constituent la collection des fermés
d'une topologie sur~$\R^n$, la \emph{topologie de Zariski} sur~$\R^n$.

\subsection*{Fonctions régulières sur~$\R^n$}

Soit~$n$ un entier naturel, et soit $U$ un ouvert de Zariski de~$\R^n$.
Soit~$f$ une fonction réelle sur~$U$, et $x\in U$. La fonction~$f$ est
\emph{régulière en~$x$} s'il existe un voisinage Zariski
ouvert~$V$ de~$x$ dans~$U$ et des fonctions polynomiales $p$~et $q$
sur~$\R^n$ tel que~$f(y)=p(y)/q(y)$ pour tout~$y\in V$, où il est
sous-entendu que~$q$ ne s'annule pas sur~$V$. La fonction~$f$ est
régulière sur~$U$ si elle est régulière en tout point de~$U$. On
note~$\SQ(U)$ l'ensemble des fonctions régulières sur~$U$. Cet
ensemble est une algèbre réelle de manière évidente. 

Soit $V$ un ouvert de Zariski contenu dans~$U$, la restriction à~$V$
d'une fonction régulière sur~$U$ est régulière sur~$V$. Il s'ensuit
que~$\SQ$ est un préfaisceau d'algèbres réelles sur~$\R^n$.  Compte
tenu de la définition locale d'une fonction régulière, il est évident
que~$\SQ$ est un faisceau sur~$\R^n$. 

Soit $x$ un point de~$\R^n$.  La fibre~$\SQ_x$ s'identifie avec
l'algèbre réelle locale des fonctions rationnelles~$p/q$ définies
en~$x$, i.e., $q(x)\neq 0$. Le faisceau~$\SQ$ sur~$\R^n$ est donc un
faisceau en algèbres réelles locales. 

Il est remarquable que les sections du faisceau~$\SQ$
au-dessus d'un ouvert de Zariski~$U$ de~$\R^n$ admettent une
description globale (cf.~\cite[Proposition~3.2.3]{BCR}):

\begin{prop}
\label{prq(U)}
Soit $n$ un entier naturel, et soit $U$ un ouvert de Zariski
de~$\R^n$. Soit~$f$ une fonction réelle sur~$U$. Alors~$f$ est
régulière sur~$U$ si et seulement s'il existe des fonctions
polynomiales $p$~et $q$ sur~$\R^n$ telles que~$f(x)=p(x)/q(x)$, avec $q(x) \neq 0$, pour
tout~$x\in U$.
\end{prop}

\begin{proof}
Voir Proposition~\ref{prqF(U)} ci-dessous où on démontre l'énoncé dans
un cadre plus général.
\end{proof}

\subsection*{Variétés réelles algébriques affines}

Soit~$n$ un entier naturel, et soit $F$ un fermé de Zariski de~$\R^n$. On
considère~$F$ muni de la topologie induite par la topologie de
Zariski, qu'on appelle \emph{topologie de Zariski} sur~$F$. Notons que
cette topologie est encore noethérienne.

Soit~$\SI$ le faisceau d'idéaux de~$\SQ$ des fonctions régulières
s'annulant sur~$F$. Le faisceau quotient $\SQ/\SI$ sera
noté~$\SQ_F$. Son support est égal à~$F$. De ce fait, on
considère~$\SQ_F$ comme un faisceau sur~$F$. C'est un faisceau en
algèbres réelles locales de corps résiduel~$\R$. On peut donc le
considérer comme un sous-faisceau des fonctions réelles sur~$F$, et on
l'appelle le faisceau des fonctions régulières sur~$F$.

On a à nouveau une description globale des sections de~$\SQ_F$
(cf.~\cite[Proposition~3.2.3]{BCR}):

\begin{prop}
\label{prqF(U)}
Soit $n$ un entier naturel, et soit $F$ un fermé de Zariski de~$\R^n$.
Soit $U$ un ouvert de Zariski de~$F$ et soit~$f$ une fonction réelle
sur~$U$.  Alors~$f$ est régulière sur~$U$ si et seulement s'il existe
des fonctions polynomiales $p$~et $q$ sur~$\R^n$ telles
que~$f(x)=p(x)/q(x)$, avec $q(x) \neq 0$, pour tout~$x\in U$.
\end{prop}

Pour la commodité du lecteur, on reproduit la démonstration
de~\cite[Proposition~3.2.3]{BCR}.

\begin{proof}
Il suffit de démontrer l'implication directe. Supposons donc que~$f$
est régulière sur~$U$. Comme la topologie de Zariski sur~$F$ est
noethérienne, il existe un recouvrement ouvert fini
$\{U_1,\ldots,U_m\}$ de~$U$ et des fonctions
polynomiales~$p_1,q_1,\ldots,p_m,q_m$ sur~$\R^n$ telles
que~$f_i(x)=p_i(x)/q_i(x)$ pour tout~$x\in U_i$,
pour~$i=1,\ldots,m$. Soit $s_i$ une fonction polynomiale sur~$\R^n$
avec $\Z(s_i)=F\setminus U_i$, pour tout~$i$. Montrons que
$$
f=\frac{s_1^2p_1q_1+\cdots+s_m^2p_mq_m}{s_1^2q_1^2+\cdots+s_m^2q_m^2}
$$ sur~$U$. 

Montrons d'abord que le second membre, qu'on notera~$g$ dans la suite,
est bien une fonction régulière sur~$U$. Soit~$x\in U$.  Il existe $i$
avec $x\in U_i$. On a donc $s_i^2(x)q_i^2(x_i)>0$. Comme
$s_j^2(x)q_j^2(x)\geq0$, pour $j\neq i$, le dénominateur de~$g$ ne
s'annule pas en $x$.  Le second membre est donc bien une fonction
régulière sur~$U$.

Il reste à montrer que~$f=g$ sur~$U$. Soit~$x\in U$. Soit $I$
l'ensembles des indices~$i$ pour lesquels~$x\in U_i$. Lorsque~$i\in
I$, on a~$f(x)=p_i(x)/q_i(x)$, et donc aussi
$$
s_i^2(x)q_i^2f(x)=s_i^2(x)p_i(x)q_i^2(x)\;.
$$ 

Observons que cette dernière formule est également valable
lorsque~$i\not\in I$ puisque~$s_i(x)=0$ dans ce cas. Il s'ensuit que
$$
(s_1^2(x)q_1^2(x)+\cdots+s_m^2(x)q_m^2(x))f(x)=
(s_1^2(x)p_1(x)q_1(x)+\cdots+s_m^2(x)p_m(x)q_m(x))
$$ 
et donc aussi que~$f(x)=g(x)$.
\end{proof}

\begin{dfn}
Une \emph{variété réelle algébrique affine} est un espace localement
annelé~$(X,\SO)$, où~$\SO$ est un faisceau en algèbres réelles
locales ayant la propriété suivante.  Il existe un entier naturel~$n$
et un fermé de Zariski~$F$ de~$\R^n$ tels que~$(F,\SQ_F)$ soit
isomorphe à~$(X,\SO)$. Si~$(X,\SO)$ est une variété réelle algébrique
affine, on appellera encore sa topologie la topologie de Zariski
sur~$X$, et une section de~$\SO$ sur un ouvert $U$ 
une fonction régulière sur~$U$.
\end{dfn}

\begin{rem}
On pourrait penser que cette définition induit une notion naturelle
de variété réelle algébrique abstraite, 
mais l'intérêt d'une telle notion est limité par le fait que tout sous-ensemble localement fermé de~$\R^n$ (voire de~$\P^n(\R)$) muni
de son faisceau naturel des fonctions régulières est une variété
réelle algébrique affine (voir \cite[Proposition~3.2.10]{BCR},
\cite[Théorème~3.4.4]{BCR} et la remarque qui suit).
\end{rem}

\begin{dfn}\label{dfn.ration}
Soit $X$ une variété réelle algébrique affine. Une \emph{fonction
  rationnelle} sur~$X$ est une classe d'équivalence de paires~$(f,U)$, où~$U$ est un ouvert
dense de~$X$ et $f$ est une fonction régulière sur~$U$. Deux telles
paires $(f,U)$~et $(g,V)$ étant équivalentes s'il existe un ouvert
dense~$W$ de~$X$ contenu dans~$U\inter V$ tel que $f_{|W}=g_{|W}$. 
\end{dfn}

Notons que cette dernière condition est équivalente à la condition 
$$
f_{|U\inter V}=g_{|U\inter V}\;.
$$

Soit $f$ une fonction rationnelle sur la variété réelle algébrique
affine~$X$. Considérons l'ensemble~$\SF$ de toutes les paires~$(g,V)$
qui sont équivalentes à~$(f,U)$.
On introduit une relation d'ordre partiel sur~$\SF$ en définissant
$(g,V)\leq(h,W)$ lorsque~$V\subseteq W$. Comme $\SO$ est un faisceau,
l'ensemble~$\SF$ contient un plus grand élément~$(g,V)$. Comme cet
élément est unique, on appelle~$V$ le domaine de définition de~$f$ et
on le note~$\dom(f)$. Du coup, on pourra également considérer~$f$
comme fonction régulière définie sur~$\dom(f)$, son extension
rationnelle étant
unique.

\begin{prop}\label{prop.pq}
Soit $X$ une variété réelle algébrique affine et soit~$f$ une fonction
rationnelle sur~$X$. Alors il existe des fonctions régulières $p$~et
$q$ sur~$X$, avec $q\neq0$ sur~$\dom(f)$, telles que~$f=p/q$
sur~$\dom(f)$.
\end{prop}

\begin{proof}
Conséquence immédiate de la Proposition~\ref{prqF(U)}.
\end{proof}

\begin{rem}
On aurait pu s'attendre à ce que les fonctions $p$~et $q$ ci-dessus
soient des fonctions polynomiales sur~$X$. 
La notion de «fonction polynomiale» sur une
variété réelle algébrique affine n'as pas de sens intrinsèque. Cette notion
dépend du choix d'un plongement de~$X$ dans un
espace affine~$\R^n$. Cela dit, si $X$ est un
fermé de Zariski de~$\R^n$, pour un certain entier naturel~$n$, on
peut supposer que $p$~et $q$ sont des fonctions polynomiales avec
les propriétés voulues, comme le montre la démonstration de~\ref{prqF(U)}.
\end{rem}

Soit~$X$ une variété réelle algébrique affine. L'ensemble des
fonctions rationnelles sur~$X$ est un anneau de manière évidente,
\emph{l'anneau total des fonctions rationnelles} sur~$X$. On le note~$\R(X)$.

\begin{prop}
Soit~$X$ une variété réelle algébrique affine irréductible. Alors
$\R(X)$ est un corps.
\end{prop}

\begin{prop}
Soit $X$ une variété réelle algébrique affine et notons $X_1,\ldots,X_m$ ses composantes irréductibles. Les morphismes de restriction $\R(X)\ra\R(X_i)$ induisent un isomorphisme d'algèbres réelles
$$
\R(X)\lra \prod_{i=1}^m\R(X_i).
$$ En particulier, l'anneau~$\R(X)$ est un produit de~$m$ corps.
\end{prop}

\subsection*{Variétés réelles algébriques affines lisses et fonctions
  $k$-r\'egulues}

Rappelons la définition d'une variété réelle algébrique lisse
(cf.~\cite[Definition~3.3.9]{BCR}):

\begin{dfn}
Soit~$X$ une variété réelle algébrique affine. La variété~$X$ est
\emph{lisse} en un point~$x$ de~$X$ si l'anneau local~$\SO_x$ est
régulier. La variété~$X$ est \emph{lisse} si elle l'est en chacun de
ses points.
\end{dfn}

On rappelle le résultat suivant (cf.~\cite[Proposition~3.3.10]{BCR}):

\begin{prop}
Soit $X$ une variété réelle algébrique affine et notons $X_1,\ldots,X_m$ ses
composantes irréductibles. Alors, $X$ est lisse si et seulement si
\begin{enumerate}
\item $X$ est réunion disjointe des $X_i$, et
\item chaque variété réelle algébrique affine~$X_i$ est lisse.
\end{enumerate}
\end{prop}

On va étendre la notion de fonction $k$-régulue de l'introduction aux
fonctions réelles définies sur une variété réelle algébrique affine
lisse.  Remarquons que si $X$ est une variété réelle algébrique affine
lisse, alors $X$ possède une structure sous-jacente de variété
différentiable de classe~$\SC^k$, pour tout entier surnaturel~$k$
(cf.~\cite[Proposition~3.3.6]{BCR}). La définition suivante est donc
naturelle.

\begin{dfn}\label{dfn.regulue}
Soit $k$ un entier surnaturel.  Soit $X$ une variété réelle algébrique
affine lisse, et soit~$f$ une fonction réelle sur~$X$. La fonction~$f$
est \emph{$k$-régulue} sur~$X$ si
\begin{enumerate}
\item $f$ est de classe $\SC^k$ sur~$X$, et
\item il existe un ouvert de Zariski~$U$ dense dans~$X$ tel que
  $f_{|U}$ est régulière.
\end{enumerate}
On dira qu'une fonction réelle sur~$X$ est \emph{régulue} lorsqu'elle
est $0$-régulue sur~$X$.  

On note~$\SR^k(X)$ l'ensemble des fonctions $k$-régulues sur~$X$. Cet
ensemble est de toute évidence un anneau.
\end{dfn}

\begin{rem}
Si $k=0$, la définition ci-dessus a même un sens lorsque la
variété réelle algébrique affine est singulière. Pourtant, elle ne
donne pas lieu à la bonne notion de fonction régulue sur une telle
variété. En effet, Koll\'ar montre qu'une fonction rationnelle sur une
variété réelle algébrique affine singulière qui s'étend par continuité
à toute la variété peut avoir des propriétés non
souhaitables~\cite[Ex.~2]{Ko}. Voilà pourquoi on introduit, pour
l'instant, uniquement la notion de fonction $k$-régulue sur une
variété réelle algébrique affine lisse. Cela nous suffira jusqu'en page~\pageref{p.dfn.regul.sing} où nous reviendrons sur cette définition.
\end{rem}

Pour une variété réelle algébrique affine lisse~$X$ donnée on dispose
d'une chaîne croissante
$$ \SO(X)\subseteq \SR^\infty(X)\subseteq \cdots\subseteq
\SR^1(X)\subseteq \SR^0(X).
$$ Le but du paragraphe suivant est de montrer que
$\SO(X)=\SR^\infty(X)$.

\section{Fonctions régulues sur les variétés affines}

\subsection*{Fonctions $\infty$-régulues}

Soit $f$ une
fonction réelle définie sur~$\R^n$. Le but de cette section est de
démontrer que si~$f$ est $\infty$-régulue sur~$\R^n$, alors elle est
régulière sur~$\R^n$. Cet énoncé est connu (voir par exemple
\cite[Proposition 2.1]{Ku}), mais nous allons en donner une
démonstration un peu plus détaillée.

On a bien l'inclusion~$\SQ(\R^n)\subseteq\SR^\infty(\R^n)$.  Le but de cette
section est de montrer que cette inclusion est une égalité. Dans
le reste de l'article, comme dans l'introduction, il sera ainsi
justifié d'utiliser la notation~$\SR^\infty(\R^n)$ pour l'anneau des
fonctions régulières sur~$\R^n$.

Soit $f$ une fonction réelle définie sur~$\R^n$.  Rappelons que~$f$
est \emph{semi-algébrique} si son graphe est un sous-ensemble
semi-algébrique de~$\R^n\times \R$~\cite[D\'efinition~2.2.5]{BCR}.
Rappelons également que~$f$ est \emph{Nash} si $f$ est à la fois
semi-algébrique et de
classe~$\SC^\infty$~\cite[Définition~2.9.3]{BCR}. De manière
équivalente, $f$ est Nash si $f$ est une fonction analytique réelle et
elle est algébrique sur l'anneau des fonctions régulières
sur~$\R^n$, i.e., il existe des fonctions polynomiales
$a_0,\ldots,a_d$ sur~$\R^n$ telles que
$$
a_d f^d+a_{d-1} f^{d-1}+\cdots+a_0=0
$$ sur~$\R^n$, où $a_d$ est non identiquement
nulle~\cite[Proposition~8.1.8]{BCR}.

\begin{prop}[Kucharz~\protect{\cite{Ku}}]
\label{regulsa}
Soient~$n$ un entier naturel et $k$ un entier surnaturel.  Une fonction
$k$-régulue sur~$\R^n$ est semi-algébrique. En particulier, une
fonction $\infty$-régulue sur~$\R^n$ est Nash.
\end{prop}

\begin{proof} 
Soit $f$ une fonction $k$-régulue sur~$\R^n$. Notons~$U$ le domaine
de~$f$ comme fonction rationnelle, et soient $p$~et $q$ deux fonctions
polynomiales sur~$\R^n$ telles que~$f(x)=p(x)/q(x)$ pour tout~$x\in
U$. Clairement, le graphe de la restriction de $f$ à $U$ est un
sous-ensemble semi-algébrique de $U\times\RR$. Comme~$f$ est continue,
le graphe de la fonction $f$ sur~$\R^n$ est l'adhérence de l'ensemble
précédent pour la topologie euclidienne dans $\RR^{n+1}$. Or,
l'adhérence d'un ensemble semi-algébrique est
semi-algébrique~\cite[Prop. 2.2.2]{BCR}. Il s'ensuit que~$f$ est une
fonction semi-algébrique.
\end{proof}

\begin{cor}
\label{zregulsa}
Soient~$n$ un entier naturel et $k$ un entier surnaturel. Un
sous-ensemble fermé $k$-régulu de~$\R^n$ est semi-algébriquement
fermé.\qed
\end{cor}

\begin{thm}
\label{thinfregestreg}
Soit~$n$ un entier naturel. 
On a
$$
\SR^\infty(\R^n)=\SQ(\R^n),
$$ i.e., une fonction réelle sur~$\R^n$ est $\infty$-régulue si et
seulement si elle est régulière.
\end{thm}

Cet énoncé se généralise certainement au cas où $X$ est une variété réelle algébrique affine lisse mais la preuve en devient plus technique et ne nous semble pas apporter d'idée nouvelle.

\begin{proof}
Comme une fonction régulière est trivialement $\infty$-régulue, il
suffit de démontrer la réciproque.

Soit $U$ le domaine
de la fonction rationnelle~$f$. Ecrivons $f=p/q$, où $p$~et $q$ sont
des fonctions polynomiales sur~$\R^n$, $q$ ne s'annulant pas sur~$U$.
On montre que la fraction rationnelle~$p/q$ est définie en tout point
de~$\R^n$.  Autrement dit, on montre que, pour tout~$x\in\R^n$, la
fraction rationnelle~$p/q$ appartient à l'anneau local~$\SQ_x$ des
germes des fonctions régulières en~$x$. Il suffit de le montrer pour
$x$ l'origine de~$\R^n$.

D'après la proposition précédente, la fonction~$f$ sur~$\R^n$ est
semi-algébrique. Comme $f$ est de classe $\SC^\infty$, elle est
Nash et donc analytique
réelle~\cite[Proposition~8.1.6]{BCR}. 
Ainsi $f$ définit un germe d'une fonction
analytique réelle en~$0$ ayant la propriété que~$qf=p$. La fonction
polynomiale~$q$ divise donc~$p$ dans l'anneau local des germes
des fonctions analytiques réelles en~$0$. Cela implique que~$q$
divise~$p$ dans l'anneau des séries
formelles ~$\R[[x_1,\ldots,x_n]]$. Ce dernier est aussi la
complétion de l'anneau local des fonctions rationnelles sur~$\R^n$ définies en~$0$. Comme ce dernier anneau est
noethérien, $q$ divise~$p$ dans~$\SQ_0$~\cite[Proposition~8.2.11]{BCR}. 

On a donc démontré que $f$ est une section globale du faisceau~$\SQ$,
i.e., $f$ est une fonction régulière sur~$\R^n$.
\end{proof}

\subsection*{La topologie $k$-régulue de $\R^n$.}

Soient $n$ un entier naturel et $k$ un entier surnaturel. Un
sous-ensemble~$F$ de~$\R^n$ est un \emph{fermé $k$-régulu} s'il
existe un sous-ensemble~$E$ de~$\SR^k(\RR^n)$ tel que
$$
\Z(E)=F.
$$ Remarquons que~$F$ n'est pas a priori le lieu des zéros communs d'un
nombre fini de fonctions $k$-régulues.  Un sous-ensemble~$U$ de~$\R^n$
est un \emph{ouvert $k$-régulu} si son complémentaire est un fermé
$k$-régulu. Les ouverts $k$-régulus de~$\R^n$ constituent une
topologie sur~$\R^n$, la \emph{topologie $k$-régulue}.

Soit $X$ un espace topologique. On rappelle que $X$ est
\emph{noethérien} si et seulement si toute suite décroissante
dénombrable $F_i\supseteq F_{i+1}$, $i\in\N$, de sous-ensembles fermés
de~$X$ est stationnaire. De manière équivalente, tout recouvrement
ouvert d'un ouvert de~$X$ contient un sous-recouvrement fini.

Compte tenu du Théorème~\ref{thinfregestreg}, un ouvert
$\infty$-régulu de~$\R^n$ est un ouvert de Zariski, et réciproquement.
Comme l'anneau des fonctions polynomiales sur~$\R^n$ est noethérien,
la topologie de Zariski sur~$\R^n$ est noethérienne. Il s'ensuit que
la topologie $\infty$-régulue sur~$\R^n$ est noethérienne.  On verra
ci-dessous qu'il en est de même pour la topologie $k$-régulue lorsque
$k$ est fini (voir Théorème~\ref{topregnoeth}).

Si $k\leq k'$, la topologie $k$-régulue est plus fine que la topologie
$k'$-régulue. En particulier, La topologie $k$-r\'egulue est plus fine
que la topologie de Zariski sur~$\RR^n$. L'exemple suivant montre que
la topologie $k$-régulue est strictement plus fine que la topologie de
Zariski, lorsque $k$ est un entier naturel et $n\geq2$.

\begin{ex}
\label{cubique}
Soit~$C$ la courbe cubique d'\'equation $y^2=x^2(x-1)$
dans~$\RR^2$. Comme le polyn\^ome d\'efinissant la courbe~$C$ est
irr\'eductible dans~$\RR[x,y]$, le sous-ensemble~$C$ de~$\RR^2$ est un
ferm\'e de Zariski irr\'eductible. En particulier, la clôture de
Zariski de~$C\setminus\{O\}$ est \'egale \`a~$C$.
\begin{figure}[ht]
\centering
\includegraphics[height =4cm]{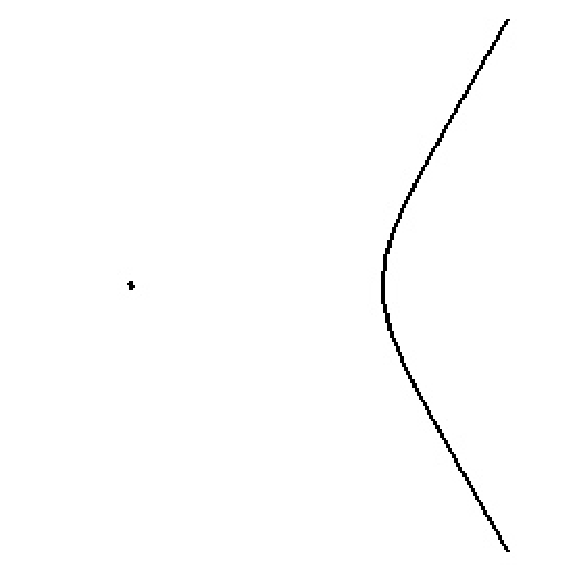}
\caption{Courbe cubique avec un point isolé.}
        \label{fig.cubic}
\end{figure}

Soit $f$ la fonction rationnelle sur~$\RR^2$ d\'efinie par
$$
f(x,y)=\frac{y^2-x^2(x-1)}{y^2+x^2}=1-\frac{x^3}{x^2+y^2}.
$$ Il est clair que $f$ s'\'etend par continuit\'e en une fonction
r\'egulue sur~$\RR^2$. On a $f(O)=1$ et
$\Z(f)=C\setminus\{O\}$. Remarquons que l'origine~$O$ est un point
isol\'e de la courbe~$C$. Par cons\'equent, $C\setminus\{O\}$ est
r\'egulument ferm\'e dans~$\RR^2$, mais non ferm\'e au sens de
Zariski. En particulier, le fermé $\infty$-régulu irréductible~$C$
n'est pas irréductible pour la topologie régulue sur~$\R^2$ et s'écrit
comme réunion de deux sous-fermés régulus stricts.
$$
C=(C\setminus\{O\})\union \{O\}.
$$

Pour montrer que $C\setminus\{O\}$ est un fermé $k$-régulu, pour tout
entier naturel~$k$, on peut supposer que $k+1$ est impair.  Posons
$$
f_k=1-(1-f)^{k+1}.
$$ La fonction~$f_k$ est bien $k$-régulue sur~$\R^2$, et, comme $k+1$
est impair, on a $\Z(f_k)=\Z(f)=C\setminus\{O\}$. Cela montre bien que
ce dernier ensemble est un fermé $k$-régulu, pour tout entier
naturel~$k$. 
\end{ex}

Cet exemple repr\'esente en fait le cas g\'en\'eral. En effet, on verra dans le Corollaire \ref{cor-const} que les topologies $k$-r\'egulue et $k'$-r\'egulue coincident pour $k$ et $k'$ des entiers naturels quelconques.

\subsection*[Fonctions $k$-r\'egulues sur~$\RR^n$]{Propriétés élémentaires
des fonctions $k$-régulues sur~$\R^n$}

Soit $n$ un entier naturel et $k$ un entier surnaturel.  Dans ce
paragraphe nous allons \'etudier les fonctions $k$-r\'egulues sur~$\RR^n$,
et les comparer avec d'autres classes de fonctions sur~$\RR^n$.

Dans l'énoncé suivant, et dans le reste de l'article d'ailleurs, on
utilise librement la notion de dimension, ou plutôt de codimension,
d'un fermé de Zariski de~$\R^n$, et plus généralement d'un ensemble
semi-algébrique de~$\R^n$~\cite[\S2.8]{BCR}.

\begin{prop}
\label{codim}
Soient $n$ un entier naturel et~$k$ un entier surnaturel.  Soit
$f\in\SR^k(\R^n)$. Soient $p$~et $q$ des fonctions polynomiales
sur~$\R^n$ telles que $f(x)=p(x)/q(x)$ pour chaque $x\in\dom(f)$.  Si
$p,q$ sont premiers entre eux, alors $\Z(q)\subseteq \Z(p)$ et
$\codim_{\RR^n}\, \Z(q)\geq 2$.
\end{prop}

\begin{proof}
Il suffit de montrer l'énoncé lorsque~$k=0$.  Montrons d'abord
l'inclusion~$\Z(q)\subseteq\Z(p)$. Soit $x\in \Z(q)$. Comme $q$ n'est
pas identiquement nulle, l'ensemble de ses z\'eros~$\Z(q)$ est nulle
part dense dans~$\RR^n$.  Il existe donc une suite $(x_m)$
dans~$\RR^n$ convergeant vers~$x$ telle que $q(x_m)\neq0$, pour
tout~$k$. On a alors
$$
p(x)=\lim p(x_m)=\lim q(x_m)f(x_m)=q(x)f(x)=0,
$$ i.e., $x\in \Z(p)$.

Montrons ensuite que $\codim\Z(q)\geq2$.  Par l'absurde, supposons que
$\codim \Z(q)\leq1$. Comme $q$ n'est pas identiquement nulle, on a
$\codim \Z(q)=1$. Il existe donc un diviseur irr\'eductible~$q'$ de~$q$
dans~$\RR[x_1,\ldots,x_n]$ avec $\codim \Z(q')=1$. D'apr\`es ce qui
pr\'ec\`ede, $\Z(q')\subseteq \Z(p)$. Comme $\codim\Z(q')=1$, on en
d\'eduit que~$q'$ divise~$p$~\cite[Th. 4.5.1, p.85]{BCR}.  Cela
contredit l'hypoth\`ese que $p$~et $q$ sont premiers entre eux.
\end{proof}

Comme on a vu, l'énoncé ci-dessus peut \^etre considérablement
renforcé lorsque~$k=\infty$. En effet, dans ce cas la fonction~$f$ est
régulière et est de la forme~$p/q$ où $q$ ne s'annule pas sur~$\R^n$,
i.e., $\Z(q)=\emptyset$.

\begin{cor}
Soient $n\leq1$ et~$k$ un entier surnaturel. Toute fonction $k$-r\'egulue
sur~$\RR^n$ est r\'eguli\`ere.\qed
\end{cor}

\begin{cor}
Soit~$k$ un entier surnaturel.  Une fonction $k$-r\'egulue sur~$\RR^2$
est r\'eguli\`ere en dehors d'un ensemble fini.\qed
\end{cor}

\begin{cor}
Soient $n$ un entier naturel et~$k$ un entier surnaturel.  Soit
$f\in\SR^k(\R^n)$. Le lieu~$\pol(f)$ où~$f$ n'est pas une fonction régulière est un
fermé de Zariski de~$\R^n$ de codimension~$\geq2$.\qed
\end{cor}

Une application de $\R^n$ dans $\R^m$ est \emph{$k$-régulue} lorsque
toutes ses fonctions coordonnées le sont. Il
n'est pas clair que la composition de deux applications $k$-régulues
soit encore $k$-régulue. Un cas où c'est évident est le suivant:

\begin{cor}
Soient $\ell,m,n$ des entiers naturels et~$k$ un entier surnaturel. 
Soient
$$
f\colon \R^n\ra \R^m
\et
g\colon \R^m\ra\R^\ell
$$ deux applications $k$-régulues. Si l'image de~$f$ est de
codimension~$\leq1$, alors la composition~$g\circ f\colon \R^n\ra
\R^\ell$ est $k$-r\'egulue.\qed
\end{cor}

On verra ci-dessous que la conclusion est valable même si la
codimension de l'image de~$f$ est plus grande que~$1$ (voir
Corollaire~\ref{cor.compositionregulues}).

\subsection*{Fonctions régulues et fonctions régulières après éclatement}

Soit $X$ une vari\'et\'e r\'eelle alg\'ebrique affine lisse.  Soit~$Y$
une sous-vari\'et\'e r\'eelle alg\'ebrique lisse de~$X$. Soient
$f_0,\ldots,f_n$ des fonctions r\'eguli\`eres sur~$X$ engendrant
l'id\'eal de~$Y$ dans~$X$. Notons~$f$ l'application r\'eguli\`ere de
$X\setminus Y$ dans~$\P^n(\R)$ de coordonn\'ees
homog\`enes~$f_0,\ldots,f_n$.  L'\emph{\'eclat\'e} $E_Y(X)$ de~$X$ le
long de~$Y$ est la cl\^oture de Zariski du graphe de~$f$ dans le
produit~$X\times\P^n(\R)$. L'\'eclat\'e~$E_Y(X)$ est encore une
vari\'et\'e r\'eelle alg\'ebrique affine lisse.  L'application de
projection
$$
\pi\colon E_Y(X)\lra X
$$
est une application r\'eguli\`ere, et bir\'eguli\`ere au-dessus
de~$X\setminus Y$. L'image r\'eciproque de~$Y$ par~$\pi$ peut
s'identifier avec le fibr\'e normal projectif de~$Y$ dans~$X$.

On aura besoin du résultat suivant:

\begin{thm}[Hironaka]
\label{thindeterminations}
Soit $X$ une variété réelle algébrique affine lisse et $f$ une
fonction rationnelle sur~$X$. Alors il existe une composition finie
d'éclatements de centres lisses
$$
\phi\colon\widetilde X\lra X
$$ telle que l'application rationnelle
$$
f\circ\phi\colon \widetilde X\dasharrow\P^1(\R)
$$
est régulière sur~$\widetilde X$ toute entière.\qed
\end{thm}

Pour une démonstration, on renvoi le lecteur vers le livre de
Kollár~\cite{Kollar-LRS} (notamment le Corollaire~3.18 amélioré en
utilisant Théorème~3.21 au lieu de Théorème~3.17, ceci appliqu\'e à
une complexification lisse~$X_\C$ de~$X$).

Soit $X$ une vari\'et\'e r\'eelle alg\'ebrique affine lisse.  Une
fonction réelle~$f$ sur~$X$ est \emph{r\'eguli\`ere apr\`es
  \'eclatements} (\emph{blow-regular}, en anglais) s'il existe une
composition d'\'eclatements \`a centres lisses $\phi\colon
\tilde{X}\ra X$ telle que la fonction réelle~$f\circ\phi$ est
r\'eguli\`ere.  Cette notion est l'analogue algébrique des fonctions
qui sont Nash apr\`es \'eclatements ou encore des fonctions analytiques
r\'eelles apr\`es \'eclatements \cite{Kuo}.

\begin{thm}
\label{threguliereaeclt}
Soit $X$ une vari\'et\'e r\'eelle alg\'ebrique affine lisse.  Soit $f$
une fonction r\'eelle sur~$X$. La fonction~$f$ est régulue sur~$X$ si
et seulement si $f$ est régulière apr\`es \'eclatements.  Plus
précisément, $f$ est régulue si et seulement s'il existe une
composition finie d'\'eclatements \`a centres lisses
$$
\phi\colon \widetilde{X}\lra X
$$ 
telle que la composition
$$
f\circ\phi\colon \widetilde{X}\lra \R
$$ 
est r\'eguli\`ere.
\end{thm}

\begin{proof} 
Supposons que~$f$ est régulière après éclatements, et montrons que~$f$
est régulue. Il existe une suite d'éclatements de centres lisses
$$
\phi\colon\widetilde X\lra X
$$ telle que~$f\circ\phi$ est régulière. Comme~$\phi$ est
birationnelle, il existe un ouvert de Zariski, qui est dense dans~$X$, sur
lequel~$f$ est régulière. Il nous reste à montrer que $f$ est
continue.  Or, la topologie euclidienne sur~$X$ est la topologie
quotient induite par la topologie euclidienne sur~$\tilde X$. Le fait
que~$f\circ\phi$ soit continue implique donc que~$f$ est continue.
Cela montre bien que~$f$ est régulue lorsque~$f$ est régulière après
éclatements.

Montrons la réciproque. Soit~$f$ une fonction régulue
sur~$X$. Il existe un ouvert de Zariski~$U$, dense dans~$X$, tel
que~$f_{|U}$ est régulière. En appliquant
Théorème~\ref{thindeterminations} à $f$, vue comme fonction
rationnelle sur~$X$, il existe une suite d'éclatements de centres
lisses
$$
\phi\colon\widetilde X\lra X
$$
telle que l'application rationnelle
$$
f\circ\phi\colon\widetilde X\dasharrow \P^1(\R)
$$ est régulière sur $\tilde X$ toute entière. Par définition de
composition d'applications rationnelles, cela veut dire que
l'application régulière~$f_{|U}\circ\phi_{|\phi^{-1}(U)}$ de
$\phi^{-1}(U)$~dans $\P^1(\R)$ s'étend à~$\tilde{X}$ toute entière
comme application régulière. Cette extension est nécessairement
continue. Or, la composition d'applications continues~$f\circ\phi$ est
également une extension continue
de~$f_{|U}\circ\phi_{|\phi^{-1}(U)}$. Comme~$\phi^{-1}(U)$ est dense
dans~$\tilde{X}$ pour la topologie euclidienne, ces deux extensions sont la
même ce qui veut dire que~$f\circ\phi$ est régulière, et $f$ est
régulière après éclatements.
\end{proof}

\section{La topologie $k$-régulue est noethérienne}

\subsection*{Fonctions régulues et fonctions régulières par strates}

Rappelons qu'un sous-ensemble localement ferm\'e de~$\R^n$, au sens de
Zariski, est un sous-ensemble de la forme~$U\inter F$, o\`u
$F\subseteq\R^n$ est un ferm\'e de Zariski et $U\subseteq \R^n$ un
ouvert de Zariski.

\begin{thm}\label{thm.stratif}
Soient $n$ un entier naturel et $k$ un entier surnaturel.  Soit $f$ une
fonction $k$-r\'egulue sur~$\R^n$. Alors, il existe une stratification
finie de~$\R^n$
$$
\R^n=\coprod_{i=1}^m S_i
$$
en sous-ensembles localement ferm\'es de~$\R^n$ au sens de Zariski
telle que la restriction~$f_{|S_i}$ est r\'eguli\`ere, pour tout~$i$.
\end{thm}

\begin{proof}
Il suffit de traiter le cas~$k=0$.  On peut, bien-s\^ur, supposer
que~$f$ est non identiquement nulle sur~$\R^n$.

Notons~$\AA^n$ l'espace affine~$\Spec\R[x_1,\ldots,x_n]$ sur~$\R$, de
sorte que~$\AA^n(\R)=\R^n$ et $\AA^n(\C)=\C^n$.  La conjugaison
complexe~$\gamma$ agit sur~$\AA^n(\C)$.  On identifie~$\AA^n(\R)$ avec
l'ensemble des points fixes de~$\gamma$ sur~$\AA^n(\C)$. 

Ecrivons $f=r/s$ avec $r$~et $s$ des polyn\^omes r\'eels
en~$x_1,\ldots,x_n$. Soit $Y$ le sous-sch\'ema r\'eduit de~$\AA^n$
d\'efini par le polyn\^ome~$rs$. Comme~$f$ est non nulle, $Y$ est un
sous-sch\'ema de~$\AA^n$ de codimension~$1$.

Appliquons le th\'eor\`eme de Hironaka\footnote{Nous utilisons ici la version \cite[Théorème~3.27]{Kollar-LRS}} au sous-sch\'ema~$Y$ de~$\AA^n$
en \'eclatant~$\AA^n$, de mani\`ere successive, mais seulement en des
centres dont l'ensemble des points r\'eels est dense. Plus
pr\'ecis\'ement, il existe une suite finie de morphismes
$$ X_\ell\lra X_{\ell-1}\lra\cdots\lra X_0=\AA^n
$$
o\`u chaque morphisme~$\pi_i\colon X_i\ra X_{i-1}$ est un \'eclatement
de centre lisse~$C_{i-1}\subseteq X_{i-1}$ de codimension~$\geq2$
ayant les propri\'et\'es suivantes. Pour tout~$i$, l'ensemble des
points r\'eels~$C_i(\R)$ est dense dans~$C_i$, et le transform\'e
strict~$Y_\ell$ de~$Y$ dans~$X_\ell$ n'a que des points r\'eels
lisses. Notons~$E_i\subseteq X_i$ le diviseur exceptionnel
de~$\pi_i$. On peut supposer, de plus, que $Y_\ell(\R)$
intersecte~$E_\ell(\R)$ transversalement en des points r\'eels lisses
de~$E_\ell$. On note encore $\phi$ la
composition~$\pi_1\circ\cdots\circ\pi_\ell$. Par construction,
$f\circ\phi$ est une fonction rationnelle sur~$X_\ell$ dont le domaine
de d\'efinition contient tous les points r\'eels de~$X_\ell$.

Soit $p$ un point du sch\'ema~$\AA^n$ dont le corps
r\'esiduel~$\kappa(p)$ est r\'eel. La fibre~$\phi^{-1}(p)$ est une
r\'eunion finie connexe d'espaces projectifs sur~$\kappa(p)$.
Soit~$P$ l'un de ces espaces projectifs. Comme le domaine de
d\'efinition de~$f\circ\phi$ contient tous les points r\'eels
de~$X_\ell$, la fonction rationnelle~$f\circ\phi$ se restreint \`a une
fonction rationnelle sur~$P$. Il existe un point
$\kappa(p)$-rationnel~$q$ de~$P$ en lequel~$f\circ\phi$ est
d\'efinie. Soit $F$ le sous-sch\'ema de~$X_0$ dont le point
g\'en\'erique est~$p$, et $G$ le sous-sch\'ema de~$X_\ell$ dont le
point g\'en\'erique est~$q$. La restriction de~$\phi$ \`a~$G$ est un
morphisme birationnel dans~$F$. Il existe donc des ouverts non
vides~$U$ de $F$ et $V$ de~$G$ tels que~$\phi_{|V}$ soit un
isomorphisme de~$V$ sur~$U$. Quitte \`a remplacer $U$~et $V$ par des
ouverts plus petits, on peut supposer que~$f\circ\phi$ est d\'efinie
sur~$V$.  Or $\phi$ induit, sur les points r\'eels, un isomorphisme
bir\'egulier de~$V(\R)$ sur~$U(\R)$. Il s'ensuit que la restriction
de~$f$ \`a~$U(\R)$ est r\'eguli\`ere.

La construction pr\'ec\'edente montre que pour tout id\'eal premier
r\'eel~$p$ de~$\R[x_1,\ldots,x_n]$, il existe un sous-ensemble
localement ferm\'e irr\'eductible~$R_p$ dense de~$\Z(p)$ tel
que~$f_{|R_p}$ est r\'eguli\`ere. Il existe alors un nombre fini
d'id\'eaux premiers r\'eels distincts~$p_1,\ldots,p_m$ tels que
$$
\R^n=\bigcup_{i=1}^m R_{p_i}.
$$
Pour simplifier la notation, on notera~$R_i$ au lieu de~$R_{p_i}$.

On construit ensuite une stratification selon le proc\'ed\'e
habituel~:  en agrandissant la famille~$\{R_i\}$ si n\'ecessaire, on
peut supposer que toutes les composantes irr\'eductibles des intersections
finies des sous-ensembles de la forme~$R_i$ appartiennent
encore \`a la famille~$\{R_i\}$. Introduisons un ordre partiel sur
la collection $\{R_i\}$. On pose $R_i\leq R_j$
si~$R_i$ est contenu dans la cl\^oture de
Zariski~$\overline{R}_j$ de~$R_j$. Soit
$$
S_i=R_i\setminus \bigcup_{R_j<R_i}\overline{R}_j.
$$
Le sous-ensemble~$S_i$ de~$\R^n$ est localement ferm\'e, et contenu
dans~$R_i$. La restriction de la fonction~$f$ \`a~$S_i$ est donc bien
r\'eguli\`ere. 

Montrons que~$S_i\inter S_j=\emptyset$ lorsque~$i\neq
j$. C'est clair lorsque~$R_j< R_i$ ou l'inverse. Supposons donc que
$R_i$~et $R_j$ ne sont pas comparables. Ecrivons l'intersection
$R_i\inter R_j$ comme r\'eunion des $R_k$ o\`u $k$ parcourt un
ensemble d'indice~$I$. Comme on a~$R_k\subseteq\overline{R}_i$ et
$R_k\subseteq \overline{R}_j$, pour tout~$k\in I$, on a
bien~$S_i\inter S_j=\emptyset$, lorsque $i\neq j$. 

Il nous reste \`a montrer que $\R^n=\bigcup S_i$. Soit $x\in\R^n$ et
soit $R_i$ le plus petit \'el\'ement de la collection~$\{R_i\}$ tel
que~$x\in R_i$. On a bien $x\in S_i$. 
\end{proof}

\begin{rem}
Il est possible de prouver le Théorème~\ref{thm.stratif}, et donc aussi le Théorème~\ref{topregnoeth} ci-dessous, sans utiliser le Théorème d'Hironaka ; et ce grâce à la Proposition~7 de \cite{KN}. Pour cela, on a besoin de la notion de restriction d'une fonction régulue à une sous-variété, cf. Proposition~\ref{prop.restric} et Remarque~\ref{rem.here}. 
\end{rem}

L'anneau $\SR^\infty(\R^n)$ est noethérien, la topologie induite est
donc noethérienne. En fait c'est encore le cas pour la topologie
$k$-r\'egulue, comme on peut le voir en utilisant la noeth\'erianit\'e
de la topologie alg\'ebriquement constructible
\cite{KurPar,Paru-const} combin\'ee avec le Théorème~\ref{thm.stratif}
(qui implique que les ferm\'es $k$-r\'egulus sont alg\'ebriquement
constructibles, cf. Corollaire \ref{coregfermeestconstructible}). Nous
allons en donner une d\'emonstration directe, en nous inspirant de
l'idée donnée dans \cite{N08} pour \'etablir la noethérianité de la
topologie quasi-analytique. Soit $F=\Z(f)$ avec
$f\in\SR^0(\R^n)$. D'après le Théorème~\ref{thm.stratif}, il existe
une stratification finie $\R^n=\coprod_{i=1}^m S_i$ en sous-ensembles
localement fermés telle que la restriction de $f$ à $S_i$ est
régulière. On peut même exiger que les strates soient lisses et
irr\'eductibles (au sens Zariski, ou encore alg\'ebriquement
constructible, cf. \cite{Paru-const}), et que $F$ soit une réunion de
strates. À une telle stratification de $F$, on associe le multi-indice $\mu=(\mu_l,\mu_{l-1},\dots,\mu_0)\in\N^{l+1}$ où $\mu_j$ est le nombre de strates de dimension $j$. On note $\mu_{\Z(f)}$ le plus petit (pour l'ordre lexicographique) des multi-indices associés aux stratifications de $\Z(f)$.

\begin{thm}
\label{topregnoeth}
Soient~$n$ un entier naturel et $k$ un entier surnaturel.  La topologie
$k$-régulue sur~$\R^n$ est noethérienne.
\end{thm}

\begin{proof}
L'anneau $\SR^\infty(\R^n)$ est noethérien, on peut donc supposer que
$k$ est un entier naturel. 
Il suffit de prouver que la topologie régulue (c'est-à-dire pour
$k=0$) sur~$\R^n$ est noethérienne. 
Tout fermé régulu est une intersection de fermés principaux (un fermé
est principal s'il est de la forme $\Z(f)$ pour $f\in\SR^0(\R^n)$). 
Il suffit donc de montrer que toute suite décroissante de fermés principaux est stationnaire. 

On considère une suite décroissante de fermés régulus principaux $\Z(f_\alpha)$, $f_\alpha\in\SR^0(\R^n)$. Il suffit alors de remarquer que $\mu_{\Z(g)}<\mu_{\Z(f)}$ si $f,g\in\SR^0(\R^n)$ sont telles que $\Z(g)$ est inclus strictement dans $\Z(f)$. 

Pour le montrer, prenons une strate $S$ de dimension $s$ pour $\Z(f)$, et supposons que $g$ s'annule sur un sous-ensemble semi-alg\'ebrique de dimension $s$ de $S$. Pour une stratification en irr\'eductibles et lisses associée \`a $\Z(g)$ par le Th\'eor\`eme~\ref{thm.stratif}, il existe une strate dont l'intersection $U$ avec $S$ est dense dans $S$ au sens de la topologie de Zariski puisque $S$ est irr\'eductible. En particulier $g_{|U}$ est r\'eguli\`ere et s'annule sur un sous-ensemble semi-alg\'ebrique de $U$ de dimension maximale, donc $g_{|U}$ est identiquement nulle. Par continuité, $g$ s'annule sur l'adh\'erence de $U$ pour la topologie euclidienne, donc $g$ est nulle sur $S$ par lissit\'e de $S$.
\end{proof}

\begin{cor}
\label{cofkresthyp}
Soient~$n$ un entier naturel et $k$ un entier surnaturel.  Tout fermé
$k$-régulu de~$\R^n$ est de la forme~$\Z(f)$, pour une certaine
fonction $k$-régulue~$f$ sur~$\R^n$.
\end{cor}

\begin{proof}
Ce corollaire se déduit de \ref{topregnoeth} en utilisant la Proposition~\ref{prop.princip}.
\end{proof}

\begin{cor}\label{cor-eclt}
Soient~$n$ un entier naturel et $k$ un entier surnaturel.  Tout fermé
$k$-régulu~$F$ de~$\R^n$ est un fermé de Zariski après éclatements,
i.e., il existe une composition d'éclatements \`a centres lisses
$$
\phi\colon \widetilde{\R}^n\lra \R^n
$$ 
telle que l'image réciproque~$\phi^{-1}(F)$ est Zariski fermée
dans~$\widetilde{\R}^n$.
\end{cor}

\begin{proof}
D'après Corollaire~\ref{cofkresthyp}, il existe une fonction
$k$-régulue~$f$ sur~$\R^n$ telle que~$\Z(f)=F$. D'après
Théorème~\ref{threguliereaeclt}, il existe une composition
d'\'eclatements de~$\R^n$ à centres lisses~$\phi$ telle que~$f\circ\phi$
est régulière. En particulier,
$$
\Z(f\circ\phi)=\phi^{-1}(\Z(f))=\phi^{-1}(F)
$$
est un fermé de Zariski.
\end{proof}

\begin{rem}
La réciproque du
corollaire ci-dessus est fausse comme le prouve l'exemple suivant. Soit $F$ la clôture euclidienne de la partie de dimension $2$ du parapluie de Whitney (cf. Exemple~\ref{exem.umbrellas}). Après résolution des singularités du parapluie, la préimage de $F$ est la résolution du parapluie complet, qui est un fermé de Zariski.
\end{rem}

Même si la topologie $k$-régulue sur~$\R^n$ est plus fine, $\R^n$
reste irréductible pour cette topologie:

\begin{prop}
Soient $n$ un entier naturel et $k$ un entier surnaturel.  Alors $\R^n$
muni de la topologie $k$-régulue est irréductible.
\end{prop}

\begin{proof}
Supposons que $\R^n=F\union G$, où $F$~et $G$ sont des fermés
$k$-régulus de~$\R^n$.  Il existe des fonctions $k$-régulues $f$~et
$g$ sur~$\R^n$ telles que $\Z(f)=F$ et $\Z(g)=G$. On a
$$
\Z(fg)=\Z(f)\union\Z(g)=F\union G=\R^n.
$$ Par conséquent, $fg=0$. L'anneau~$\SR^k(\R^n)$ étant un sous-anneau
du corps~$\R(\R^n)$, il est intègre. On a donc $f=0$~ou $g=0$, i.e.,
$F=\R^n$~ou $G=\R^n$.
\end{proof}

\subsection*{Ensembles symétriques par arcs}
Dans ce paragraphe, nous faisons le lien avec la théorie des ensembles symétriques par arcs introduite par K.~Kurdyka.
Une fonction semi-alg\'ebrique $f\colon \RR^n\rightarrow \RR$ est
appelée \emph{analytique par arcs}~\cite{Kur} si $f\circ\gamma$ est
analytique pour tout arc analytique $\gamma\colon I\rightarrow \R^n$,
où $I$ est un intervalle ouvert de $\R$. 
Il est immédiat que si $f\colon \RR^n\rightarrow \RR$ devient analytique après une suite finie d'éclatements à centres algébriques lisses bien choisis, alors $f$ est analytique par arcs
(Dans \cite[Thm. 1.1]{BM},
Bierstone et Milman montrent en fait l'équivalence). Comme une fonction régulière est analytique, on obtient
la conséquence suivante du Théorème~\ref{threguliereaeclt}, gr\^ace à
la Proposition~\ref{regulsa}:

\begin{cor}
\label{reguluestarcanalytique}
Soient $n$ un entier naturel et $k$ un entier surnaturel.  Une fonction
$k$-régulue sur $\R^n$ est analytique par arcs.\qed
\end{cor}

\begin{rema} Il existe bien évidemment des fonctions analytiques par arcs
  qui ne sont pas régulues. Par exemple $f(x,y)=\sqrt{x^4+y^4}$ est
  analytique par arcs \cite{BM} et n'est clairement pas régulue sur $\R^2$.
\end{rema}

On rappelle qu'un ensemble $E$ semi-algébrique dans $\R^n$ est dit
\emph{symétrique par arcs} si et seulement si pour tout arc analytique
$\gamma\colon ]-1,1[\rightarrow \R^n$, si $\gamma(]-1,0[)\subseteq E$
    alors $\gamma(]-1,1[)\subseteq E$, voir~\cite{Kur}.  Le lien avec
    les fonctions analytiques par arcs est le suivant. Si $f$ est une
    fonction analytique par arcs sur~$\R^n$, alors son lieu de
    zéros~$\Z(f)$ est symétrique par arcs.

Suivant les notations de \cite{Kur}, on appelle $\AR$ la topologie
dont les ensembles fermés sont les ensembles symétriques par arcs.
On verra ci-dessous que la
topologie $\AR$ est plus fine que la topologie régulue. 

\begin{prop} 
Soient~$n$ un entier naturel et $k$ un entier surnaturel. Soit $F$ un
fermé $k$-régulu de $\R^n$. Alors $F$ est un ensemble symétrique par
arcs.
\end{prop}

\begin{proof}
Soit $F=\Z(E)$ un fermé $k$-régulu de~$\R^n$.  D'après le
Corollaire~\ref{reguluestarcanalytique}, $\Z(f)$ est symétrique par
arcs, pour tout~$f\in E$. Comme les sous-ensembles symétriques par
arcs sont les fermés d'une topologie,
$$
\Z(E)=\bigcap_{f\in E}\Z(f)
$$ 
est encore symétrique par arcs.
\end{proof}

\begin{rem}
Le théorème de noethérianité~\ref{topregnoeth} est alors un corollaire immédiat du théorème de noethérianité de Kurdyka~\cite[Th. 1.4.]{Kur} grâce à la proposition ci-dessus.
\end{rem}

\subsection*{Ensembles algébriquement constructibles}
Rappelons qu'un sous-ensemble $E$ de~$\R^n$ est \emph{Zariski
  constructible} ou \emph{algébriquement
  constructible} s'il est réunion finie de sous-ensembles localement
fermés au sens de Zariski.

\begin{cor}
\label{coregfermeestconstructible}
Soient~$n$ un entier naturel et $k$ un entier surnaturel. Un ensemble
fermé $k$-r\'egulu de~$\R^n$ possède une stratification finie en
sous-ensembles localement ferm\'es de~$\R^n$. En particulier, tout
ensemble $k$-régulument fermé de~$\R^n$ est Zariski
constructible.\qed
\end{cor}

Notons qu'un sous-ensemble $k$-régulument fermé de~$\R^n$ n'est pas
forcément localement fermé, comme le montre l'exemple suivant.

\begin{ex}
\label{exkregfermeconstructible}
Soit $C\subseteq\R^2$ la courbe cubique ayant la singularité isolée
en~$O$ de l'Exemple~\ref{cubique}. Soit~$C'=C\setminus\{O\}$.
Le sous-ensemble
$$
F=(C'\times\R)\union \{O\}
$$ est un sous-ensemble $k$-régulument fermé de~$\R^3$ qui n'est pas
localement fermé au sens de Zariski. En effet, la clôture de Zariski
de~$F$ est égale à~$C\times\R$. Le sous-ensemble~$F$ de~$C\times\R$
n'est pas Zariski ouvert pour la topologie induite 
car son intersection avec~$\{O\}\times\R$ est
réduite à l'origine de~$\R^3$ et n'est pas ouvert
dans~$\{O\}\times\R$.
\end{ex}

Dans le cas~$k=\infty$, l'énoncé du
Corollaire~\ref{coregfermeestconstructible} peut être considérablement
renforcé car les fermés sont alors les fermés de
Zariski.

Le Th\'eor\`eme~\ref{thm.stratif} a encore l'\'enoncé suivant comme
conséquence~:

\begin{cor}\label{cor.compositionregulues}
Soient~$\ell,m,n$ des entiers naturels et $k$ un entier surnaturel. Soient
$$
f\colon \R^n\ra \R^m
\et
g\colon \R^m\ra\R^\ell
$$ deux applications $k$-régulues. Alors la composition~$g\circ f\colon
\R^n\ra \R^\ell$ est $k$-r\'egulue.
\end{cor}

\begin{proof}
L'application $g\circ f$ est bien-sûr de classe~$\SC^k$.  Soit
$U$ l'intersection des domaines des fonctions coordonnées de~$f$.  La
restriction de~$f$ à~$U$ est donc une application régulière
dans~$\R^m$.  D'après le Théorème~\ref{thm.stratif}, il existe une
stratification
$$
\R^m=\coprod_{i=1}^p S_i
$$ en sous-ensembles localement ferm\'es de~$\R^m$ au sens de Zariski
telle que la restriction~$g_{|S_i}$ est r\'eguli\`ere, pour tout~$i$.
Soit $i$ tel que $U\inter f^{-1}(S_i)$ est Zariski dense
dans~$U$. Soit~$U'$ un ouvert Zariski dense dans~$U$
avec~$f(U')\subseteq S_i$. La restriction à~$U'$ de~$g\circ f$ est
alors régulière.
\end{proof}

\begin{cor}
\label{cokregcont}
Soient~$m,n$ des entiers naturels et $k$ un entier surnaturel.  Une
application $k$-régulue de~$\R^n$ dans~$\R^m$ est continue pour la
topologie $k$-régulue.
\end{cor}

\begin{proof}
Soit $f\colon \R^n\ra\R^m$ une application $k$-régulue et $F$ un fermé
$k$-régulu de~$\R^m$. Il existe une fonction $k$-régulue~$g$
sur~$\R^m$ dont l'ensemble des zéros est égal à~$F$. D'après le
Corollaire~\ref{cor.compositionregulues}, $g\circ f$ est
$k$-régulue. L'image réciproque
$$
f^{-1}(F)=\Z(g\circ f)
$$ est donc un fermé $k$-régulu de~$\R^n$.
\end{proof}

Grâce au Corollaire~\ref{cor.compositionregulues}, on est en mesure de
montrer que l'anneau~$\SR^k(\RR^n)$ n'est pas noethérien, lorsque~$k$
est fini, même si la topologie $k$-régulue est noethérienne:

\begin{prop}\label{prop.nonoeth}
Soient $n$~et $k$ des entiers naturels. L'anneau~$\SR^k(\R^n)$
n'est pas noethérien lorsque~$n\geq2$.
\end{prop}

\begin{proof}
Pour $m\in \N$, soit $f_m$ la fonction $k$-régulue sur~$\R^n$ définie par
$$
f_m=\frac{x_2^{3+k}}{x_2^2+(x_1-m)^2}.
$$ Soit $I_m$ l'idéal de~$\SR^k(\R^n)$ engendré par les
fonctions~$f_0,\ldots,f_m$. On montre comme dans~\cite[Ex. 6.11]{Kur}
que la suite d'idéaux croissante~$I_0,I_1,\ldots$ n'est pas
stationnaire. En fait, on montre par l'absurde que~$I_{m+1}\neq
I_m$. En effet, supposons que $f_{m+1}$ appartient à~$I_m$. On peut
donc écrire
$$
f_{m+1}=\sum_{i=1}^m g_if_i,
$$ où $g_1,\ldots,g_m$ sont des fonctions $k$-régulues sur~$\R^n$.  En
restreignant à la courbe réelle
algébrique~$C=\{m+1\}\times\R\times\{0\}^{n-2}$, qu'on identifie
  avec~$\R$, on obtient
$$ x_2^{1+k}=\sum_{i=1}^m (g_i)_{|C}(f_i)_{|C}
$$ sur~$C$. D'après Corollaire~\ref{cor.compositionregulues}, les
  fonctions $(g_i)_{|C}$ et $(f_i)_{|C}$ sont $k$-régulues et donc
  régulières. 
De plus, la fonction
$$
(f_i)_{|C}=\frac{x_2^{3+k}}{x_2^2+(m+1-i)^2}
$$ s'annule au point $x_2=0$ avec multiplicité~$3+k$,
  pour~$i=1,\ldots,m$. Il s'ensuit que $x_2^{1+k}$ s'annule en~$x_2=0$ avec
  multiplicité au moins~$3+k$. Contradiction.
\end{proof}

\subsection*{Jets de fonctions $k$-régulues}

Soient $k,\ell$ et $n$ des entiers naturels, avec $\ell\leq k$.
Soit~$f$ une fonction de classe~$\SC^k$ sur un ouvert~$U$ de~$\R^n$.
Habituellement, on écrit le $\ell$-ième polynôme de Taylor de~$f$ en
un point~$a$ sous la forme
$$
\sum_{|I|\leq \ell} \frac{\partial^{|I|}f}{\partial x^I}(a)\frac{(x-a)^I}{I!}.
$$ Lorsqu'on s'intéresse à celui-ci comme fonction de~$a$, on est
amené à le noter
$$
\sum_{|I|\leq k} \frac{\partial^{|I|}f}{\partial x^I}(x)\frac{(y-x)^I}{I!},
$$ où $y$ désigne un système de coordonnées affine arbitraire
sur~$\R^n$.  On appelle ce dernier le \emph{$\ell$-jet} de~$f$, et on
le note~$j_\ell(f)$. 

Le fibré vectoriel trivial~$J_\ell$ sur~$\R^n$ de base
$$
\frac{(y-x)^I}{I!}, \quad |I|\leq \ell
$$ est le \emph{fibré des $\ell$-jets} sur~$\R^n$.  Le $\ell$-jet
$j_\ell(f)$ de~$f$ est une section de~$J_\ell$ de
classe~$\SC^{k-\ell}$ au-dessus~$U$.

Le jonglage notationnel ci-dessus a l'avantage que le $0$-jet~$j_0(f)$
de~$f$ coïncide avec~$f$. Plus généralement, le coefficient de
$(y-x)^I/I!$ du $\ell$-jet de $f$ est égal à la dérivée partielle
d'ordre~$I$ de~$f$.

Le fibré vectoriel des $\ell$-jets $J_\ell$ possède une structure de
fibré en algèbres si on définit
$$
\frac{(y-x)^I}{I!}\cdot\frac{(y-x)^I}{J!}=
\begin{cases}
\binom{I+J}{I}
\frac{(y-x)^{I+J}}{(I+J)!}&\text{si $|I+J|\leq\ell$, et}\\
0&\text{sinon.}
\end{cases}
$$ Si $f$~et $g$ sont des fonctions de classe~$\SC^k$ sur un
ouvert~$U$ de~$\R^n$, on a
$$
j_\ell(fg)=j_\ell(f) j_\ell(g)
$$
sur~$U$.

Comme le fibré vectoriel~$J_\ell$ est libre d'une base explicitée, on
peut identifier ses sections au-dessus d'un ouvert~$U$ de~$\R^n$ avec
des applications de~$U$ dans~$\R^N$, où $N=N_{n,k,\ell}$ est le rang
de~$J_\ell$. Du coup, cela a un sens de parler de \emph{section
  rationnelle} ou encore de \emph{section $m$-régulue} de~$J_\ell$, où
$m$ est un entier surnaturel quelconque.

\begin{lem}\label{lejellrat}
Soit~$\ell$ un entier naturel.  Soit $f$ une fonction rationnelle
sur~$\R^n$, et soit~$U$ son domaine. Le $\ell$-jet $j_\ell(f)$ est une
section rationnelle de~$J_\ell$ de domaine~$U$.
\end{lem}

\begin{proof}
Ecrivons $f=p/q$ sur~$U$, où $p$~et $q$ sont des fonctions
polynomiales, $q$ ne s'annulant pas sur~$U$. On sait que la dérivée
partielle d'ordre~$I$ de $f$ est de la forme~$r/q^{|I|+1}$, où $r$ est
une fonction polynomiale sur~$\R^n$. Il s'ensuit que~$j_\ell(f)$ est
une section rationnelle de~$J_\ell$ de domaine~$V\supseteq U$. Comme
le coefficient de $(y-x)^0/0!$ est $p/q=f$, le domaine~$V$ est égal
à~$U$.
\end{proof}

\begin{prop}\label{prj1c0}
Soit $f$ une fonction réelle sur~$\R^n$. La fonction~$f$ est
$1$-régulue si et seulement si~$f$ est rationnelle et la section
rationnelle $j_1(f)$ de~$J_1$ au-dessus de~$\dom(f)$ s'étend de
manière continue à~$\R^n$.
\end{prop}

\begin{proof}
L'implication directe est évidente compte tenu du
Lemme~\ref{lejellrat}. Montrons donc l'implication réciproque, et
soit~$f$ une fonction rationnelle telle que~$j_1(f)$ s'\'etend de
manière continue à~$\R^n$. Notons~$U$ le domaine de~$f$. D'après
l'hypothèse, $f$ est de classe~$\SC^1$ sur~$U$.  Soit $Z$ le
complémentaire de~$U$, i.e., $Z=\pol(f)$. Comme $Z$ est un
sous-ensemble Zariski fermé strictement contenu dans~$\R^n$, il existe
un système de coordonnées affine~$x'$ sur~$\R^n$ avec la propriété
suivante. Aucune des fibres des projections associées
$$
p_i\colon\R^n\lra\R^{n-1}
$$ omettant la $i$-ième coordonnée~$x_i'$ n'est contenue entièrement
dans~$Z$. Quitte à remplacer $x$~par $x'$, on peut supposer que la
propriété est satisfaite par le système des coordonnées~$x$, i.e.,
pour tout~$i=1,\ldots,n$ et pour tout~$a\in\R^{n-1}$,
l'intersection~$Z\inter p_i^{-1}(a)$ est un ensemble fini.

Comme~$f$ est de classe~$\SC^1$ sur~$U$, les dérivées
partielles~$f_i=\partial f/\partial x_i$, $i=1,\ldots,n$, de $f$
existent et sont continues sur~$U$. De plus, d'après l'hypothèse,
elles s'étendent par continuité sur~$\R^n$ tout entier. En
particulier, $f_i$ s'étend par continuité sur~$p_i^{-1}(a)$ pour
tout~$a\in\R^{n-1}$ et pour tout~$i$. D'après un résultat classique en
analyse~\cite[Theorem~11.7]{Spivak}, cela implique que la restriction
de~$f_i$ à~$p_i^{-1}(a)$ est de classe~$\SC^1$ sur~$p_i^{-1}(a)$, pour
tout $a$~et $i$. (Nous utilisons ici la finitude de~$Z\inter
p_i^{-1}(a)$.)  Autrement dit, $f$
est partiellement différentiable sur~$\R^n$ tout entier, et ses
dérivées partielles d'ordre~$1$ sont continues sur~$\R^n$. D'après un
autre résultat classique d'analyse~\cite[Thm~9.21]{Rudin-PMA}, $f$ est
alors de classe~$\SC^1$ sur~$\R^n$.
\end{proof}

\begin{thm}\label{thjkc0}
Soit $k$ un entier naturel.  Soit $f$ une fonction réelle
sur~$\R^n$. La fonction~$f$ est $k$-régulue si et seulement si~$f$ est
rationnelle et la section rationnelle $j_k(f)$ de~$J_k$ au-dessus
de~$\dom(f)$ s'étend de manière continue à~$\R^n$.
\end{thm}

\begin{proof}
L'implication directe étant encore évidente, on démontre l'implication
réciproque. Le cas~$k=0$ est trivial. Supposons que $k>0$ et soit~$f$
une fonction rationnelle de domaine~$U$ telle que la section~$j_k(f)$
au-dessus de~$U$ s'étend de manière continue à~$\R^n$. Les dérivées
partielles~$f_I$ d'ordre~$k-1$ de~$f$ ont des $1$-jets qui s'étendent
de manière continue à~$\R^n$. D'après la Proposition~\ref{prj1c0}, ces
dérivées partielles sont toutes $1$-régulues. Il s'ensuit que~$f$ est
$k$-régulue.
\end{proof}

\begin{thm}\label{thjlck-l}
Soient $k$ un entier surnaturel et $\ell$ un entier naturel~$\leq k$.
Soit $f$ une fonction réelle sur~$\R^n$. La fonction~$f$ est
$k$-régulue si et seulement si~$f$ est rationnelle et la section
rationnelle $j_\ell(f)$ de~$J_\ell$ au-dessus de~$\dom(f)$ s'étend de
manière $\SC^{k-\ell}$ à~$\R^n$.
\end{thm}

\begin{proof}
On peut supposer que $k$ est fini. Comme précédemment, il suffit de
démontrer l'implication indirecte.  Soit $f$ une fonction rationnelle dont
toutes les dérivées d'ordre~$\leq\ell$ sont $(k-\ell)$-régulues.
D'après l'implication directe du Théorème~\ref{thjkc0}, le $k$-jet de~$f$
s'étend de manière continue à~$\R^n$.  D'apr\`es l'implication
indirecte du même théorème, $f$ est $k$-régulue.
\end{proof}

\section{Nullstellenzatz et Théorèmes A et B}

\subsection*{L'inégalité de {\Lbarre}ojasiewicz}

Un rôle clé dans la suite est joué par une version $k$-régulue d'un corollaire de
l'inégalité de {\Lbarre}ojasiewicz, cf.~\cite[Prop. 2.6.4]{BCR}.

\begin{lem} 
\label{LojaregulRn}
Soient $n$~et $k$ des entiers naturels.  Soit $f$ une fonction
$k$-régulue sur~$\R^n$.  Si
$$
g\colon \D(f)\rightarrow \RR
$$ est $k$-régulue, alors il existe un entier naturel~$N$ tel que
l'extension à~$\R^n$ par~$0$ de~$f^Ng$ est $k$-régulue.
\end{lem}

\begin{proof}
Si $k=0$, le résultat est classique, cf.~\cite[Prop. 2.6.4]{BCR}. Par conséquent, pour tout entier naturel $k$, il existe un entier
naturel~$M$ tel que pour tout~$I$ avec $|I|\leq k$, l'extension de
$$
f^M\cdot\frac{\partial^{|I|} g}{\partial x^I}
$$ par~$0$ sur~$\R^n$ est continue. Soit
$N=M+k$. La fonction $f^Ng$ étendue par~$0$ sur~$\R^n$ est de classe~$\SC^k$. En effet, toutes les dérivées partielles d'ordre inférieur ou égal à $k$ se prolongent par continuité par~$0$ sur~$\R^n$. La preuve pour tout entier naturel $k$ se déduit donc du cas $k=0$. 

On peut préciser cela en termes de $k$-jets.
Comme~$f$ est $k$-régulue sur~$\R^n$, tout coefficient du
$k$-jet~$j_k(f^N)$ de~$f^N$ est de la forme~$f^M\cdot h$, où $h$ est
une fonction régulue sur~$\R^n$. Du coup, $j_k(f^N)$ lui-même est de
la forme~$f^M\cdot s$, où $s$ est une section régulue de~$J_k$ au
dessus de $\R^n$. Il
s'ensuit que l'extension par~$0$ de
$$
j_k(f^Ng)=j_k(f^N)j_k(g)=f^Msj_k(g)
$$ est une section régulue de~$J_k$ au-dessus~$\R^n$ tout
entier. D'après le Théorème~\ref{thjkc0}, la fonction~$f^Ng$ est
$k$-régulue sur~$\R^n$.
\end{proof}

Le lemme ci-dessus nous sera utile sous la forme suivante:

\begin{lem}
\label{LojaregulU}
Soient $n$~et $k$ des entiers naturels. Soit~$U$ un ouvert $k$-régulu
de~$\R^n$, et soit $f$ une fonction $k$-régulue sur~$U$. Si
$$
g\colon \D(f)\rightarrow \RR
$$ est une fonction $k$-régulue, il existe un entier
naturel~$N$ tel que l'extension à~$U$ par~$0$ de~$f^Ng$ est
$k$-régulue.
\end{lem}

\begin{proof}
Soit $h$ une fonction $k$-régulue sur~$\R^n$ telle
que~$\D(h)=U$. D'après le Lemme~\ref{LojaregulRn}, il existe un entier
naturel~$M$ tel que l'extension par~$0$ de~$h^Mf$ est $k$-régulue
sur~$\R^n$. Comme~$g$ est $k$-régulue sur~$\D(h^Mf)$, il existe encore
d'après le Lemme~\ref{LojaregulRn}, un entier naturel~$N$ tel que
l'extension par~$0$ de~$(h^Mf)^Ng$ est $k$-régulue sur~$\R^n$. Il
s'ensuit que l'extension par~$0$ de~$f^Ng$ à~$U=\D(h)$ est
$k$-régulue.
\end{proof}

Remarquons que les deux énoncés ci-dessus sont faux lorsque~$k=\infty$. Un
contre-exemple est le suivant. Soit $f$ la fonction sur~$\R^2$ définie
par~$f(x)=x^2+y^2$, et soit $g$ la fonction sur $\R^2\setminus\{0\}$
définie par~$g(x,y)=1/(x^2+2y^2)$. Les fonctions $f$~et $g$ sont bien de
classe~$\SC^\infty$, mais l'extension à~$\R^2$ par~$0$ de $f^Ng$ n'est
de classe~$\SC^\infty$ pour aucun entier naturel~$N$.

\subsection*{Le faisceau des fonctions $k$-régulues sur~$\R^n$}

Soient $n$ un entier naturel et $k$ un entier surnaturel.  Soit $U$ un
ouvert $k$-régulu de~$\R^n$. La topologie de Zariski sur~$U$ est la
topologie induite par la topologie de Zariski sur~$\R^n$.  Une
fonction réelle~$f$ définie sur~$U$ est $k$-régulue sur~$U$ si $f$ est
de classe~$\SC^k$ sur~$U$, et il existe des fonctions polynomiales $p$
et~$q$ sur~$\R^n$ telles que~$f=p/q$ sur un ouvert de Zariski dense
de~$U$. On note~$\SR^k(U)$ l'ensemble des fonctions $k$-régulues
sur~$U$. C'est un anneau sous les opérations habituelles. Si $V$ est
un ouvert $k$-régulu de~$\R^n$ contenu dans~$U$, l'application de
restriction de $\SR^k(U)$~dans $\SR^k(V)$ est un morphisme
d'anneaux. Il est clair que~$\SR^k$ est donc un préfaisceau.

\begin{prop}
Soient $n$ un entier naturel et $k$ un entier surnaturel.  Le
préfaisceau~$\SR^k$ est un faisceau sur~$\R^n$ pour la topologie
$k$-régulue.
\end{prop}

\begin{proof}
Il suffit de montrer l'énoncé suivant. Soit $\{U_i\}_{i\in I}$ un
recouvrement ouvert d'un ouvert $k$-régulu~$U$ de~$\R^n$. Supposons
que $f_i$ est une fonction $k$-régulue sur
$U_i$, pour tout~$i$, telles que
$$
(f_i)_{|U_i\inter U_j}=(f_j)_{|U_i\inter U_j}
$$ pour tout $i,j$. Alors il existe une fonction $k$-régulue $f$
sur~$U$ dont la restriction à~$U_i$ est égale à~$f_i$. En effet, il
existe une fonction~$f$ sur~$U$ de classe~$\SC^k$ dont la restriction
à~$U_i$ est égale à~$f_i$ pour tout~$i$. Pour montrer que $f$ est
$k$-régulue, on peut supposer que~$U$ est non vide. Du coup, il
existe~$i$ tel que~$U_i$ est non vide. Comme~$\R^n$ est irréductible
pour la topologie $k$-régulue, tout ouvert $k$-régulu non vide
de~$\R^n$ est dense. En particulier, un ouvert de Zariski dense
dans~$U_i$ est aussi dense dans~$U$. Comme~$f_i$ est régulière sur un
ouvert de Zariski dense dans~$U_i$, la fonction~$f$ est régulière sur un
ouvert de Zariski dense dans~$U$.
\end{proof}

\begin{cor}
Soient $n$ un entier naturel et $k$ un entier surnaturel. Le
faisceau~$\SR^k$ sur~$\R^n$ est un faisceau en $\R$-algèbres
locales. La paire~$(\R^n,\SR^k)$ est un espace localement annelé en
$\R$-algèbres.\qed
\end{cor}

\begin{cor}
Soient $m,n$ des entiers naturels et $k$ un entier surnaturel. Une
application $k$-régulue de~$\R^n$ dans $\R^m$ induit naturellement un
morphisme d'espaces localement annelés en $\R$-algèbres
de~$(\R^n,\SR^k)$ dans~$(\R^m,\SR^k)$. Cela établit une bijection
entre l'ensemble des applications $k$-régulues de~$\R^n$ dans~$\R^m$
et l'ensemble des morphismes d'espaces localement annelés en $\R$-algèbres
de~$(\R^n,\SR^k)$ dans~$(\R^m,\SR^k)$.
\end{cor}

\begin{proof}
Soit~$f$ une application $k$-régulue de~$\R^n$ dans~$\R^m$.  D'après le
Corollaire~\ref{cokregcont}, $f$ est une application continue de~$\R^n$
dans~$\R^m$, par rapport aux topologies $k$-régulues. De plus, si $g$
est une fonction $k$-régulue sur un ouvert $k$-régulu $U$  de~$\R^m$, la
composition~$f^\#(g)=g\circ f$ est une fonction $k$-régulue sur l'ouvert
$k$-régulu~$f^{-1}(U)$ de~$\R^n$, d'après le
Corollaire~\ref{cor.compositionregulues}. On en déduit donc un morphisme d'espaces localement annelés en $\R$-algèbres
$$
(f,f^\#)\colon (\R^n,\SR^k)\lra(\R^m,\SR^k).
$$

Réciproquement, étant donné un tel morphisme~$(f,f^\#)$,
l'application~$f$ est une application $k$-régulue de~$\R^n$
dans~$\R^m$ car la $i$-ième fonction coordonnée~$f_i$ de~$f$ est
$f^\#(x_i)$ et est une fonction $k$-régulue sur~$\R^n$,
pour~$i=1,\ldots,m$.
\end{proof}

Grâce à la version $k$-régulue de l'inégalité de {\Lbarre}ojasiewicz, on a une
description de l'anneau des sections du faisceau~$\SR^k$ au-dessus
d'un ouvert comme une localisation de~$\SR^k(\R^n)$, lorsque~$k$ est
fini:

\begin{prop}
\label{restriction}
Soient $n$~et $k$ des entiers naturels.  Soit $U$ un ouvert $k$-régulu
de~$\R^n$, et soit~$f$ une fonction $k$-régulue sur~$\R^n$ telle
que~$\D(f)=U$. Alors, le morphisme de restriction de~$\SR^k(\R^n)$
dans~$\SR^k(U)$ induit un isomorphisme
$$
\SR^k(\R^n)_f\iso\SR^k(U).
$$
\end{prop}

\begin{proof}
Soit $\varphi\colon \SR^k(\R^n)\rightarrow
\SR^k(U)$ le morphisme de restriction. Comme la restriction de~$f$
à~$U$ ne s'annule pas, $\varphi$ induit un morphisme
$$
\varphi_f\colon \SR^k(\R^n)_f\lra \SR^k(U).
$$ On montre que~$\varphi_f$ est un isomorphisme.

Soit $g\in \SR^k(U)$. D'après le Lemme \ref{LojaregulU}, il existe un entier
naturel $N$ tel que $f^Ng$ se prolonge en une fonction $k$-régulue
sur~$\R^n$. Par conséquent $\varphi_f$ est surjectif.

Pour montrer que~$\varphi_f$ est injectif, il suffit de montrer
l'énoncé suivant. Pour toute fonction $k$-régulue~$h$ sur~$\R^n$ dont
la restriction à~$U$ est identiquement nulle, la fonction~$fh$ est
identiquement~$0$ sur~$\R^n$. Cet énoncé est clair
lorsque~$U=\emptyset$, car~$f=0$ dans ce cas. Si~$U\neq\emptyset$,
alors~$U$ est dense dans~$\R^n$ pour la topologie euclidienne. Il
s'ensuit que~$h=0$ sur~$\R^n$, et en particulier, $fh=0$.
\end{proof}

\subsection*{Idéaux radicaux de fonctions régulues}

On rappelle qu'un idéal $I$ d'un anneau commutatif $A$ est \emph{réel}
s'il satisfait la propriété suivante~: Si $f_1^2+\ldots +f_m^2\in I$,
avec $f_1,\ldots, f_m\in A$, alors $f_i\in I$ pour chaque
$i=1,\ldots,m$.

\begin{prop}
\label{radreel} 
Soient $n$ et $k$ des entiers naturels.  Soit $I\subseteq \SR^k(\R^n)$
un idéal radical. Alors $I$ est un idéal réel.
\end{prop}

\begin{proof} On suppose $f_1^2+\ldots +f_m^2\in I$ avec
$f_1,\ldots,f_m\in \SR^k(\R^n)$. 
Pour $i=1,\ldots,k$ on a
  $\frac{f_i^{3+k}}{\sum_{i=1}^m f_i^2}\in \SR^k(\R^n)$.  En effet c'est
  la composée des applications $k$-régulues $\R^n\rightarrow
  \RR^m,\,\,x\mapsto \,(f_1(x),\ldots,f_m(x))$ et de $\R^m\rightarrow
  \RR,\,\, (x_1,\ldots,x_m)\mapsto\, \frac{x_i^{3+k}}{\sum_{i=1}^m
    x_i^2}$, qui est bien $k$-régulue d'après le
  Corollaire~\ref{cor.compositionregulues}.  Donc
  $f_i^{3+k}=\frac{f_i^{3+k}}{\sum_{i=1}^m f_i^2}(\sum_{i=1}^m f_i^2)\in I$.
  L'idéal $I$ étant radical on obtient $f_i\in I$.
\end{proof}

\begin{rema} La proposition précédente montre que l'idéal engendré par 
$x^2+y^2$ dans $\SR^k(\RR^2 )$ n'est pas radical. En effet, s'il est
  radical alors il est réel et par conséquent $x\in
  \SR^k(\R^2).(x^2+y^2)$ 
  et donc $x=f.(x^2+y^2)$
  avec $f\in \SR^k(\RR^2)$.  Mais $f=\frac{x}{x^2+y^2}$ n'est
  clairement pas continue en l'origine.
\end{rema}

Soit $I$ un idéal radical dans $\SR^k(\R^n)$. On sait que $I$ est un
idéal réel d'après la Proposition \ref{radreel}.  On montre que cette
propriété est conservée par intersection avec les polynômes.

\begin{lem} 
\label{reelintersection}
Soit $I$ un idéal radical dans $\SR^k(\R^n)$. Alors $J=I\cap
\RR[x_1,\ldots,x_n]$ est un idéal réel et on a
$\I(\Z(J))=J$.
\end{lem}

\begin{proof}
L'idéal $J$ est réel car $I$ est réel (Proposition \ref{radreel}).
On obtient $\I(\Z(J))=J$ par le Nullstellensatz réel
\cite[Thm. 4.1.4]{BCR}
\end{proof}

Soit $J$ un idéal de $\RR[x_1,\ldots,x_n]$.
On regarde le comportement de $J$ quand on l'étend dans les fonctions
régulues. 

On fixe d'abord quelques notations.
\begin{node}
Soit $x\in\RR^n$. Dans la suite, on note 
$$
\pm_x=\{p\in \RR [x_1,\ldots,x_n]|\,\,p(x)=0\}
$$
$$
\M_x=\{f\in\SR^k(\R^n)|\,\,f(x)=0\}
$$
les idéaux maximaux de $\RR [x_1,\ldots,x_n]$ et $\SR^k(\R^n)$ respectivement, 
des polynômes et des fonctions régulues qui s'annulent en $x$.
Plus généralement, si $A\subseteq \R^n$, on note dans la suite
$$\I (A)=\{p\in \RR [x_1,\ldots,x_n]|\,\,p(x)=0\,\,\forall x\in A\}$$
et 
$$\I_{\SR^k}(A)=\{f\in\SR^k(\R^n)|\,\,f(x)=0\,\,\forall x\in A\}.$$
\end{node}

Examinons l'exemple suivant.
Soit $(x,y)=\pm_O=\{p\in \RR [x,y]|\,\,p(O)=0\}$ l'idéal maximal de
$\RR[x,y]$ où $O$ est l'origine de
$\RR^2$. On montre que $\SR^0 (\RR^2).\pm_O$ n'est pas maximal dans $\SR^0
(\RR^2)$.\\
On a clairement $\SR^0(\RR^2).\pm_O\subseteq \M_O$.
On a aussi
$f=\frac{xy^2}{x^2+y^2}\in\M_O$. On va montrer par contre que
$f\not\in\SR^0(\RR^2).\pm_O$. 
Supposons que $f$ s'\'ecrive $f=x.g+y.h$ avec $g,h\in \SR^0(\R^2)$. Alors on peut écrire
$$f=g(O).x+h(O).y+(g-g(O)).x+(h-h(O)).y=a.x+b.y+o (\sqrt{x^2+y^2}).$$
Mais alors $f$ serait différentiable à l'origine et on obtient une
contradiction. Cet exemple montre que l'on peut s'attendre à quelques
surprises. 

\begin{thm}
\label{radideal}
Soit $J$ un idéal réel de $\RR[x_1,\ldots,x_n]$. Alors 
$$\Rad(\SR^k(\R^n).J)=
\I_{\SR^k} (\Z (J)).$$
\end{thm}

\begin{proof}
On note $I=\SR^k(\R^n).J$. Soit $f\in \Rad(I)$, il existe $r\in \NN^*$ tel que 
$f^r\in I$. Comme $J=\I (\Z (J))$ \cite[Thm. 4.1.4]{BCR} dans $\RR
[x_1,\ldots,x_n]$ car $J$ est réel, on en déduit que $f^r$ s'annule
identiquement sur $\Z
(J)$. Par conséquent $f$ s'annule aussi identiquement sur $\Z(J)$
i.e. $f\in\I_{\SR^k} (\Z (J))$. On a montré une inclusion.

Pour l'inclusion réciproque, on suppose que $J$ est l'id\'eal $J=(p_1,\ldots,p_l)$ avec $p_i\in\RR
[x_1,\ldots,x_n]$. Alors on a  $V=\Z (J)=\Z(s)$ avec $s=p_1^2+\ldots +p_l^2$.
Soit $f\in\SR^k(\R^n)$ tel que $f\in\I_{\SR^k(\R^n)} (V)$. 
Soit $g=\frac{1}{s}$, c'est une
fonction $k$-régulue sur $\RR^n\setminus V$ donc sur $\D
(f)$. 
Par le Lemme
\ref{LojaregulRn}, il existe un entier positif $N$
tel que $h=f^N.g$ étendue par $0$ sur $\Z(f)$ est $k$-régulue sur
$\RR^n$. On a clairement  $h\in \I_{\SR^k}
(\Z(f))\subseteq \I_{\SR^k}
(V)$ et de plus $f^N=h.s$: en effet, sur $\D (f)$ c'est évident
et si $x\in Z(f)$
on a aussi $f^N(x)=h(x)s(x)=0$
(on utilisera plusieurs fois dans la suite ce même argument).
Par conséquent
$$f^N=h.s\in \SR^k(\R^n).J=I$$ 
car $s\in J$.
\end{proof}

Dans le cas où $J$ est maximal, on obtient
\begin{cor}
\label{radmax}
On a 
$$
\Rad(\SR^k(\R^n).\pm_x)=\M_x\;.
$$
\end{cor}

\begin{cor}
\label{diff} On suppose $k=0$ et 
soit $f\in \M_x$. Il existe un entier strictement positif $N$ tel que $f^N$
soit différentiable en $x$ avec $D(f)_x=0$.
\end{cor}

\begin{proof}
On suppose $x=O$. En regardant la preuve du théorème précédent, il
existe un entier strictement positif $N$ tel que
$f^N=(x_1^2+\ldots+x_n^2).h$ avec $h\in\SR^0(\R^n)$. On a donc
$f^N=x_1.h_1+\ldots +x_n.h_n$
avec $h_i\in \SR^0(\R^n)$ vérifiant $h_i(O)=0$. On peut donc écrire
$f^N=o(\sqrt{x_1^2+\ldots
+x_n^2})$ ce qui termine la preuve.
\end{proof}

\subsection*{L'anneau~$\SR^k(\R)$ est radicalement principal}

Soit $A$ un anneau.  On dira qu'un id\'eal $I$ de~$A$ est
\emph{radicalement principal} s'il existe $f\in I$ tel que
$$
\Rad(I)=\Rad(f).
$$ Dans ce cas, on dira aussi que $f$ est un \emph{générateur radical}
de~$I$.  On dit que~$A$ est \emph{radicalement principal} si tous ses
id\'eaux le sont.

\begin{prop}
\label{prgenrad}
Soient $n$ et $k$ des entiers naturels.  Soit~$I$ un idéal
de~$\SR^k(\R^n)$, et supposons que $f\in I$ soit telle
que~$\Z(f)=\Z(I)$. Alors, $f$ est un générateur radical de~$I$, i.e.,
$$ \Rad(f)=\Rad(I).
$$
\end{prop}

\begin{proof}
Comme $f\in I$, on a bien-sûr
l'inclusion~$\Rad(f)\subseteq\Rad(I)$. Pour montrer l'inclusion
réciproque, il suffit de montrer que~$\Rad(f)\supseteq I$. Soit $g\in
I$. La fonction~$1/f$ est bien définie et $k$-régulue
sur~$\D(g)$. D'après le Lemme~\ref{LojaregulRn}, il existe un entier
naturel~$N$ tel que l'extension $h$ par~$0$ de~$g^N/f$ sur $\Z(g)$ est encore
$k$-régulue. On a $g^N=fh$ sur $\R^n$. Cela implique que~$g\in\Rad(f)$.
\end{proof}

\begin{prop}
\label{radprinc}
Soient $n$ et $k$ des entiers naturels.  L'anneau~$\SR^k(\R^n)$ est
radicalement principal. Plus précisément, soit~$I$ un idéal
de~$\SR^k(\R^n)$. Alors il existe une fonction $k$-régulue~$f\in I$
telle que
$$ \Rad(f)=\Rad(I).
$$
\end{prop}

\begin{proof}
D'après le Théorème~\ref{topregnoeth}, la topologie $k$-régulue
sur~$\R^n$ est noethérienne. Il existe donc un nombre fini de fonctions
$k$-régulues $f_1,\ldots,f_m$ dans~$I$ telles que
$$
\Z(f_1,\ldots,f_m)=\Z(I).
$$ Soit 
$$
f=f_1^2+\cdots+f_m^2.
$$ Comme $f\in I$ et $\Z(f)=\Z(I)$, la fonction~$f$ est un générateur
radical de~$I$ d'après la Proposition~\ref{prgenrad}.
\end{proof}

\begin{prop}
\label{prfinradI}
Soient $n$ et $k$ des entiers naturels. Soit $f$ une fonction
$k$-régulue sur~$\R^n$, et soit~$I$ un idéal de~$\SR^k(\R^n)$. Alors,
$f\in\Rad(I)$ si et seulement si~$\Z(f)\supseteq\Z(I)$.
\end{prop}

\begin{proof}
D'après la Proposition~\ref{radprinc}, il existe~$g\in I$ telle
que~$\Rad(g)=\Rad(I)$. En particulier, on a
$$
\Z(g)=\Z(\Rad(g))=\Z(\Rad(I))=\Z(I).
$$
L'énoncé est donc une conséquence de la proposition suivante.
\end{proof}

\begin{prop}
\label{prfinradg}
Soient $n$ et $k$ des entiers naturels. Soient $f$~et $g$ des fonctions
$k$-régulues sur~$\R^n$. Alors, $f\in\Rad(g)$ si et seulement
si~$\Z(f)\supseteq\Z(g)$.
\end{prop}

\begin{proof}
Supposons que~$f\in\Rad(g)$, et soit~$x\in\Z(g)$. Il existe un entier
naturel~$N$ tel que~$f^N\in (g)$, i.e., $f^N=gh$ pour une certaine
fonction $k$-régulue~$h$ sur~$\R^n$. Comme~$g(x)=0$, on a~$f^N(x)=0$
et donc~$f(x)=0$, i.e., $x\in\Z(f)$.

Réciproquement, supposons que~$\Z(f)\supseteq\Z(g)$. Cela veut dire
que~$g$ ne s'annule pas sur~$\D(f)$. En particulier, la fonction~$1/g$
existe sur~$\D(f)$ et y est $k$-régulue. D'après la version
$k$-régulue de l'inégalité de {\Lbarre}ojasiewicz, il existe un entier
naturel~$N$ tel que l'extension par~$0$ de~$f^N/g$ sur $\Z(f)$ est $k$-régulue
sur~$\R^n$. Notons cette extension par~$h$. On a donc~$f^N=gh$, 
i.e.,
$f\in\Rad(g)$.
\end{proof}

\begin{cor}
Soient $n$ et $k$ des entiers naturels. Soient $f$~et $g$ des fonctions
$k$-régulues sur~$\R^n$. Alors, $\Rad(f)=\Rad(g)$ si et seulement
si~$\Z(f)=\Z(g)$.\qed
\end{cor}

En utilisant Proposition~\ref{radprinc} encore, on obtient la caractérisation
géométrique suivante des générateurs radicaux d'un idéal de fonctions
$k$-régulues:

\begin{cor}
Soient $n$ et $k$ des entiers naturels.  Soit $f$ une fonction
$k$-régulue sur~$\R^n$, et soit~$I$ un idéal de~$\SR^k(\R^n)$. Alors,
$\Rad(f)=\Rad(I)$ si et seulement si~$\Z(f)=\Z(I)$.\qed
\end{cor}

Ou encore plus généralement:

\begin{cor}
Soient $n$ et $k$ des entiers naturels.  Soient $I$~et $J$ des idéaux
de~$\SR^k(\R^n)$. Alors, $\Rad(I)=\Rad(J)$ si et seulement
si~$\Z(I)=\Z(J)$. En particulier, deux idéaux radicaux
de~$\SR^k(\R^n)$ sont égaux si et seulement s'ils ont le même ensemble
des zéros.\qed
\end{cor}

En résumé:
\begin{thm}
\label{thmradprinc}
Soient $n$ et $k$ des entiers naturels.  L'anneau~$\SR^k(\R^n)$ est
radicalement principal. Plus précisément, soit~$I$ un idéal
de~$\SR^k(\R^n)$. Alors il existe une fonction $k$-régulue~$f\in I$
telle que
$$ \Rad(f)=\Rad(I).
$$
i.e. telle que $$\Z(f)=\Z (I).$$
\end{thm}

\subsection*{Le Nullstellensatz $k$-régulu}

\begin{ton}
Soient $n$ un entier naturel et $k$ un entier surnaturel.  Si $V$ est un
sous-ensemble de $\RR^n$, on note $\I_{\SR_k}(V)$, ou maintenant 
simplement~$\I(V)$
s'il n'y a pas de confusion possible, l'ensemble des fonctions
$k$-régulues sur~$\R^n$ s'annulant sur~$V$, i.e.,
$$
\I(V)=\{f\in \SR^k(\R^n)\suchthat f(x)=0\,\,\forall x\in V\}.
$$
\end{ton}

Il est clair que~$\I(V)$ est un idéal radical dans l'anneau~$\SR^k(\R^n)$.

\begin{prop}[Nullstellensatz faible]
Soient $n$ un entier naturel et $k$ un entier surnaturel.  Soit $I$ un
id\'eal de l'anneau~$\SR^k(\R^n)$ des fonctions
$k$-r\'egulues sur~$\RR^n$.  Si $\Z(I)=\emptyset$ dans~$\RR^n$, alors
$I=\SR^k(\R^n)$.
\end{prop}

\begin{proof}
Il existe $f\in I$ tel que $\Z(f)=\Z(I)$, en effet, si $k$ est un entier naturel c'est le Théorème~\ref{thmradprinc}, et si $k=\infty$, c'est la noethérianité de l'anneau des fonctions régulières. Supposons que~$\Z(I)=\emptyset$.  
Cela veut dire que la fonction $k$-r\'egulue $f$ n'a pas de z\'eros
dans~$\RR^n$. Elle est donc inversible dans~$\SR^k(\R^n)$. Par
conséquent $1=f.\frac{1}{f}\in I$ i.e. $I=\SR^k(\R^n)$.
\end{proof}

Nous venons de voir que le Nullstellensatz faible est valable sur $\SR^k(\R^n)$ pour tout entier surnaturel $k$. Par contre, le Nullstellensatz fort ($\I(\Z(I))=\Rad(I)$ cf. e.g. \cite[Theorem~5.1, p. 49]{Har}) n'est pas valable pour les fonctions
$\infty$-régulues. Un contre-exemple est le suivant. Soit~$I$ l'idéal
de~$\SR^\infty(\R^2)$ engendré par~$x^2+y^2$. L'ensemble des
zéros~$\Z(I)$ de~$I$ est égal à l'origine de~$\R^2$. La fonction~$x$
s'annule bien sur~$\Z(I)$, mais n'appartient pas à l'idéal
radical~$\Rad(I)$ de~$I$ dans~$\SR^\infty(\R^2)$. En effet, 
l'extension à~$\R^2$ de la fonction~$x^N/(x^2+y^2)$ n'est de
classe~$\SC^\infty$ pour aucun entier naturel~$N$.

\begin{thm}[Nullstellensatz]
\label{Nullstellensatz}
Soient $n$ et $k$ des entiers naturels. Soit $I$ un id\'eal (non-nécessairement de type fini) 
de~$\SR^k(\R^n)$. Alors, $$\I(\Z(I))=\Rad(I).$$
\end{thm}

\begin{proof}
Soit~$I$ un idéal quelconque
de~$\SR^k(\R^n)$. L'inclusion $\Rad(I)\subset \I(\Z(I))$ est
évidente. Pour l'inclusion inverse on considère $g\in \I(\Z(I))$. 
D'après le Théorème~\ref{thmradprinc}, il existe~$f\in
I$ telle que~$\Rad(f)=\Rad(I)$ et $\Z(f)=\Z(I)$.
Par conséquent $g\in \I(\Z(f))$ i.e. $\Z(f)\subseteq
\Z(g)$. D'apr\`es le Lemme~\ref{LojaregulRn}, il existe un entier
naturel~$N$ tel que l'extension par~$0$ de $\frac{g^N}{f}$ sur $\Z(g)$
soit $k$-r\'egulue sur~$\RR^n$. Notons $h$ cette extension. 
On obtient finalement  $g^N=fh\in I$ i.e. $g\in\Rad (I)$.
\end{proof}

\begin{rem} Soient $n$ et $k$ des entiers naturels.   Il est facile de voir que si $F$ est un ferm\'e
  $k$-r\'egulu de~$\R^n$ alors $\Z(\I(F))=F$. En effet, $F\subset \Z(\I(F))$ pour tout sous-ensemble de $\R^n$ et en écrivant~$F=\Z(I)$ pour un
idéal~$I$ de~$\SR^k(\R^n)$, on a $I \subset \I(\Z(I))$ et 
$$
\Z(\I(F))=\Z(\I(\Z(I)))\subset\Z(I)=F.
$$
Le Théorème \ref{Nullstellensatz} établit donc une correspondance entre les ferm\'es $k$-r\'egulus
de~$\R^n$ et les id\'eaux radicaux de~$\SR^k(\R^n)$.
\end{rem}

Remarquons que l'inégalité de {\Lbarre}ojasiewicz \ref{LojaregulRn} pour les
fonctions $k$-régulues est l'outil principal pour la démonstration du
Nullstellensatz \ref{Nullstellensatz}. Ce même type d'idée est utilisé dans \cite{CarralCoste}
pour montrer que l'anneau des fonctions semi-algébriques continues est
un anneau de Gelfand.

\subsection*{Le spectre de l'anneau des fonctions $k$-r\'egulues sur $\R^n$}
\label{sespec}

Soit $n$ un entier naturel et $k$ un entier surnaturel.  Soit
$f\in\SR^k(\R^n)$. L'ensemble des zéros de~$f$ dans le spectre
$\Spec\SR^k(\R^n)$ est noté~$\V(f)$, i.e.,
$$
\V(f)=\{p\in\Spec\SR^k(\R^n)\suchthat f\in p\}\;.
$$ 
Son complémentaire est noté~$\U(f)$. Plus gén\'eralement, si $E$ est
un sous-ensemble de~$\SR^k(\R^n)$, l'ensemble des zéros communs
dans~$\Spec\SR^k(\R^n)$ des éléments de~$E$ est noté~$\V(E)$, et son
complémentaire est~$\U(E)$. Plus précisément,
$$
\V(E)=\{p\in\Spec\SR^k(\R^n)\suchthat E\subseteq p\}
$$
et
$$
\U(E)=\{p\in\Spec\SR^k(\R^n)\suchthat E\not\subseteq p\}.
$$ Il est bien connu que les sous-ensembles de la forme~$\U(E)$
constituent une topologie
sur~$\Spec\SR^k(\R^n)$~\cite[Chapitre~1]{EGA1}. L'espace
topologique~$\Spec\SR^k(\R^n)$ est quasi-compact. Plus généralement,
les ouverts de la forme~$\U(f_1,\ldots,f_m)$ sont quasi-compacts.

Soient $n$ un entier naturel et $k$ un entier surnaturel. Soit
$x\in\R^n$. On note encore~$\M_x$ l'ensemble des fonctions $k$-régulues
sur~$\R^n$ s'annulant en~$x$, i.e.,
$$
\M_x=\{f\in\SR^k(\R^n)\suchthat f(x)=0\}.
$$ Il est clair que~$\M_x$ est un idéal maximal de~$\SR^k(\R^n)$. Une
conséquence du Nullstellensatz (\ref{Nullstellensatz}) est que tout idéal maximal
de~$\SR^k(\R^n)$ est de cette forme.

\begin{prop}
\label{maximal}
Soient $n$ un entier naturel et $k$ un entier surnaturel.  Soit $\M$
un idéal maximal de $\SR^k(\R^n)$. Il existe un et un seul~$x\in\R^n$
tel que~$\M_x=\M$.
\end{prop}

\begin{proof} 
Regardons d'abord le cas~$k=\infty$.  L'anneau des fonctions
$\infty$-régulues est égal à l'anneau des fonctions régulières
sur~$\R^n$. Ce dernier est la localisation de l'anneau des
polynômes~$\R[x_1,\ldots,x_n]$ par rapport à la partie multiplicative
des polynômes ne s'annulant pas sur~$\R^n$. 
Si $\M$ est un idéal maximal de $\SR^\infty(\R^n)$, c'est par
conséquent un ideal réel et le Nullstellensatz classique fonctionnne
i.e. $\I (\Z (\M))=\M$ \cite[Thm. 4.1.4]{BCR}.

Soit $k$ un entier surnaturel. D'après le
Théorème~\ref{Nullstellensatz} si $k$ est fini et ce qui précède si
$k$ est infini, $\I (\Z (\M))=\M$. En particulier,
$\Z (\M)$ est non-vide. Soit~$x\in \Z (\M)$. On a
donc~$\M\subseteq\M_x$.  Par maximalité, on obtient $\M=\M_x$. Cela
montre l'existence de~$x$. L'unicité est claire car~$\M_x\neq\M_y$
lorsque $x$~et $y$ sont des éléments distincts de~$\R^n$.
\end{proof}
 
Soit 
$$
\iota\colon \RR^n\lra\Spec\SR^k(\R^n)
$$ l'application définie par~$\iota(x)=\M_x$, où~$\M_x$ est l'idéal
maximal de~$\SR^k(\R^n)$ des fonctions $k$-régulues s'annulant en~$x$.
L'application~$\iota$ est bien continue lorsqu'on considère~$\R^n$
avec la topologie $k$-régulue. D'après la version faible du
Nullstellensatz régulu, $\iota$ est une bijection sur l'ensemble des idéaux
maximaux de~$\SR^k(\R^n)$. La version forte du Nullstellensatz
(Théorème~\ref{Nullstellensatz}) peut se reformuler ainsi~:

\begin{thm}
\label{thbijferirrprem}
Soient $n$~et $k$ des entiers naturels.  L'application~$\iota$ induit
une bijection entre~$\Spec\SR^k(\R^n)$ et l'ensemble des
sous-ensembles $k$-régulument fermés et irréductibles
de~$\RR^n$. Plus précisément, pour tout sous-ensemble fermé
$k$-régulument irréductible~$X$ de~$\RR^n$ il existe un et un seul
idéal premier~$p$ de~$\SR^k(\R^n)$ tel
que~$X=\iota^{-1}(\V(p))$.\qed
\end{thm}

Comme précédemment, cet énoncé est faux si~$k=\infty$. Par exemple,
l'idéal de~$\R[x,y]$ engendré par le polynôme irréductible~$x^2+y^2$
est un idéal premier~$p$. Or,
$$
\iota^{-1}(\V (p))=O=\iota^{-1}(\V(\M_O)).
$$

Soit $F\subseteq\RR^n$ un fermé $k$-régulu. Notons par~$\tilde F$ le plus
petit sous-ensemble ferm\'e de~$\Spec\SR^k(\R^n)$ tel que
$\iota(F)\subseteq\tilde F$, i.e., $\tilde F$ est l'adh\'erence
de~$\iota(F)$ dans $\Spec\SR^k(\R^n)$ i.e. $\tilde F=\V (\iota(F))$. 

\begin{lem}
\label{letF=VIF}
Soient $n$ un entier naturel et $k$ un entier surnaturel.
\begin{enumerate}
\item Soit $F_\alpha$, $\alpha\in A$ une collection de
  sous-ensembles $k$-régulument fermés de~$\RR^n$. Alors
$$
\widetilde{\bigcap_{\alpha\in A}F_\alpha}=
\bigcap_{\alpha\in A} \widetilde{F}_\alpha.
$$
\item Soient $F_1,\ldots,F_m$ des sous-ensembles $k$-régulument fermés
  de~$\RR^n$. Alors 
$$
\widetilde{(F_1\union\cdots\union F_m)}=
\widetilde{F}_1\union\cdots\union \widetilde{F}_m.
$$
\end{enumerate}
\end{lem}

\begin{proof}
La deuxième propriété est évidente. La première est valable
car~$\iota$ est injectif.
\end{proof}

\begin{thm}
Soient $n$~et $k$ des entiers naturels.  L'application~$\iota$ induit
une bijection entre l'ensemble des sous-ensembles fermés
de~$\Spec\SR^k(\R^n)$ et l'ensemble des sous-ensembles $k$-régulument
fermés de~$\RR^n$. Plus précisément, pour tout sous-ensemble
$k$-régulument fermé~$X$ de~$\RR^n$ il existe un et un seul ensemble
fermé~$Y$ de~$\Spec\SR^k(\R^n)$ tel que~$X=\iota^{-1}(Y)$, à
savoir~$Y=\tilde{X}$.
\end{thm}

\begin{proof}
Soit $X$ un fermé $k$-régulu dans~$\RR^n$. Posons~$Y=\tilde{X}$ et montrons
que~$\iota^{-1}(Y)=X$. Comme~$\RR^n$ est noethérien, $X$ possède un
nombre fini de composantes irréductibles $k$-régulues~$F_1,\ldots,F_m$.
D'après le lemme précédent, $Y=\tilde{F}_1\union\cdots\union
\tilde{F}_m$.  D'après le Théorème~\ref{thbijferirrprem}, on a alors
$$
\iota^{-1}(Y)=\iota^{-1}(\widetilde{F}_1)\union\cdots\union\iota^{-1}(\widetilde{F}_m)=
F_1\union\cdots\union F_m=X \;.
$$

Montrons l'unicité de~$Y$. Soit $Z$ un sous-ensemble fermé
de~$\Spec\SR^k(\R^n)$ tel que $\iota^{-1}(Z)=X$. Comme $Y$ est le plus
petit fermé de~$\Spec\SR^k(\R^n)$ contenant~$\iota(X)$, on
a~$Y\subseteq Z$. Montrons l'inclusion inverse. Soit $p\in Z$, et soit
$F=\iota^{-1}(\V (p))$. Montrons d'abord
que~$\iota(F)\subseteq Y$.  Soit~$x\in F$. On
a~$\iota(x)\in \V(p)$. Comme $\V (p)$ est contenu
dans~$Z$, et comme~$\iota^{-1}(Z)=X$, on a $x\in X$. On a
également~$\iota^{-1}(Y)=X$, donc $\iota(x)\in Y$. Cela montre bien
que~$\iota(F)\subseteq Y$. Or, d'après le Théorème~\ref{thbijferirrprem},
$F$ est un fermé $k$-régulument irréductible dans~$\RR^n$, et $p\in
\tilde{F}$. Il s'ensuit que~$p\in Y$. Cela montre
l'inclusion~$Z\subseteq Y$.
\end{proof}

\begin{cor}
Soient $n$ un entier naturel et $k$ un entier surnaturel.
L'espace topologique~$\Spec\SR^k(\R^n)$ est noethérien.
\end{cor}

\begin{proof}
Si $k$ est fini, c'est une conséquence du théorème précédent. Pour le
cas $k=\infty$, il suffit de rappeler que l'anneau $\SR^\infty(\R^n)$
est noethérien.
\end{proof}

Pour $U$ un ouvert r\'egulu de~$\RR^n$, soit $\tilde U$ le
compl\'ementaire de $\tilde{F}$ dans $\Spec\SR^k(\R^n)$, o\`u $F$ est
le compl\'ementaire de~$U$ dans~$\RR^n$.

Dans l'anneau $\SR^\infty$ des fonctions régulières sur $\R^2$, on
peut remarquer que
$O=\iota^{-1}(\V (x,y))=\iota^{-1}(\V (x^2+y^2))$ où
$(x,y)$ et $(x^2+y^2)$ sont deux idéaux premiers distincts de
$\SR^\infty$.  Voici un corollaire qui montre bien la différence entre
l'anneau des fonctions régulues et l'anneau des fonctions régulières
sur~$\RR^n$~:

\begin{cor}
\label{cobijouv}
Soient $n$~et $k$ des entiers naturels.  Soit $U\subseteq \RR^n$ un
ouvert $k$-r\'egulu. Il existe un et un seul ouvert~$V$
de $\Spec\SR^k(\R^n)$ tel que $\iota^{-1}(V)=U$, à
savoir~$V=\widetilde{U}$.\qed
\end{cor}

Là encore, cet énoncé est faux lorsque~$k=\infty$, et ce pour
presque tous les ouverts lorsque~$n\geq2$~!

Comme conséquence du corollaire précédent, on donne l'énoncé suivant:

\begin{cor}
Soient $n$~et $k$ des entiers naturels. Tout ouvert
de~$\Spec\SR^k(\R^n)$ est de la forme~$\U(f)$.
\end{cor}

Dans la suite de ce paragraphe, on identifiera~$\RR^n$ avec son image
dans $\Spec\SR^k(\R^n)$ via l'application~$\iota$, pour simplifier.

Notons~$\widetilde\SR^k$ le faisceau structural sur le schéma
affine~$\Spec\SR^k(\R^n)$. Il induit par restriction un faisceau en
anneaux locaux sur~$\RR^n$ que l'on notera
encore~$\widetilde\SR^k$. Soit~$U$ un ouvert $k$-régulu de~$\R^n$,
et~$f$ une fonction $k$-régulue sur~$\R^n$ telle que~$\D(f)=U$.
D'après le Corollaire~\ref{cobijouv}, on a
$$
\widetilde\SR^k(U)=\SR^k(\R^n)_f.
$$ 
D'après la Proposition~\ref{restriction},
$$
\SR^k(\R^n)_f=\SR^k(U).
$$
Il s'ensuit l'énoncé suivant:

\begin{cor}
Soient $n$ un entier naturel et $k$ un entier surnaturel.  La
restriction du faisceau structural sur~$\Spec\SR^k(\R^n)$ à~$\R^n$ est
le faisceau~$\SR^k$ des fonctions $k$-régulues sur~$\R^n$.\qed
\end{cor}

Le cas~$k=\infty$ vient du fait suivant. Soit~$U$ un ouvert de Zariski
de~$\R^n$. L'anneau des fonctions régulières~$\SR^\infty(U)$ sur~$U$
peut s'identifier avec la limite inductive des
localisations~$\SR^\infty(\R^n)_f$, où~$f$ parcourt l'ensemble des
fonctions régulières sur~$\R^n$ qui ne s'annule pas sur~$U$.

\subsection*{Faisceaux $k$-régulus quasi-cohérents sur~$\R^n$}

Soient $n$ un entier naturel et $k$ un entier surnaturel.  Soit $M$ un
$\SR^k(\R^n)$-module. Le faisceau~$\tilde{M}$ sur le 
schéma~$\Spec\SR^k(\R^n)$ induit encore par restriction un
faisceau~$\tilde{M}$ sur~$\RR^n$. Là aussi, pour un ouvert
$k$-régulu~$U$ de~$\R^n$, on a
$$
\tilde{M}(U)=M_f
$$ lorsque~$k$ est fini, où~$f$ est une fonction $k$-régulue
sur~$\RR^n$ avec~$\D(f)=U$.  Dans le cas~$k=\infty$, le
module~$\tilde{M}(U)$ peut s'identifier avec la limite inductive des
localisations~$M_f$, où~$f$ parcourt l'ensemble des fonctions
régulières sur~$\R^n$ qui ne s'annulent pas sur~$U$.

Plus généralement, soit~$V$ un ouvert $k$-régulu de~$\R^n$, et $g$ une
fonction $k$-régulue sur~$\R^n$ telle que~$\D(g)=V$.  Soit $M$ un
$\SR^k(V)$-module. La restriction à~$V$ du faisceau~$\tilde{M}$
sur~$\Spec\SR^k(V)$ est un faisceau en $\SR_{|V}^k$-modules, encore
noté~$\tilde{M}$, déterminé lorsque $k$ est fini par
$$
\tilde{M}(U)=M_f
$$ pour tout ouvert $k$-régulu~$U$ de~$V$, et pour toute fonction
$k$-régulue~$f$ sur~$V$ avec~$\D(f)=U$. Dans le cas~$k=\infty$, le
module~$\tilde{M}(U)$ peut s'identifier avec la limite inductive des
localisations~$M_f$, où~$f$ parcourt l'ensemble des fonctions
régulières sur~$V$ qui ne s'annulent pas sur~$U$.

Un faisceau \emph{$k$-régulu} sur~$\R^n$ est un faisceau en
$\SR^k$-modules sur $\R^n$.  Soit~$\SF$ un faisceau $k$-régulu
sur~$\R^n$.  On dira que le faisceau~$\SF$ est \emph{quasi-cohérent}
s'il existe un recouvrement de~$\RR^n$ par des ouverts
$k$-régulus~$U_i$, $i\in I$, tel que, pour tout~$i$, il existe un
$\SR^k(U_i)$-module $M_i$ avec~$\tilde M_i$ isomorphe à~$\SF_{|U_i}$.

Nous obtenons alors une version $k$-régulue du
Théorème~A de Cartan, cf.~\cite[Chap. II, Corollary 5.5]{Ha77}.

\begin{thm}
\label{thA}
Soient $n$~et $k$ des entiers naturels.  Le foncteur qui associe à un
$\SR^k(\R^n)$-module $M$ le faisceau~$\tilde M$ est une équivalence de
catégories sur la catégorie des faisceaux $k$-régulus quasi-cohérents
sur~$\RR^n$. Le foncteur réciproque est le foncteur ``sections
globales'' $H^0$. En particulier, le faisceau~$\tilde{M}$ est engendré
par ses sections globales.\qed
\end{thm}

Comme signalé dans l'introduction, ce résultat est faux
pour~$k=\infty$. En effet, il existe des faisceaux réguliers
quasi-cohérents sur~$\R^n$ qui ne sont pas engendrés par leurs
sections globales:

\begin{ex}[cf.\protect{\cite[Example~12.1.5]{BCR}}]
Soit $p\in\RR[x,y]$ le polynôme défini par
$$
p=x^2(x-1)^2+ y^2.
$$ Remarquons que $p$ possède exactement deux racines réelles à savoir
$c_0=(0,0)$ et $c_1=(1,0)$. Soit $U_i=\RR^2\setminus\{c_i\}$ pour
$i=0,1$. Les ouverts réguliers $U_0$~et $U_1$ constituent un
recouvrement de~$\RR^2$. On définit un faisceau régulier
quasi-cohérent~$\SF$ par rapport à ce recouvrement. En fait, $\SF$ va
être localement libre de rang~$1$. Explicitement, on construit~$\SF$ à
partir des faisceaux $\SR_{|U_0}^\infty$~et $\SR_{|U_1}^\infty$ en les
recollant au-dessus de~$U_0\inter U_1$ à l'aide de la fonction de
transition $g_{01}=p$ sur $U_0\inter U_1$. Plus précisément, deux
sections $s_0$~et $s_1$ de $\SR_{|U_0}^\infty$~et $\SR_{|U_1}^\infty$
au-dessus d'ouverts de Zariski $V_0$~et $V_1$, respectivement, se
recollent si~$g_{01}s_1=s_0$ sur~$V_0\inter V_1$.

Montrons que toute section régulière globale~$s$ de~$\SF$ s'annule
en~$c_1$. La restriction~$s_i$ de $s$ à~$U_i$ est une fonction
régulière sur~$U_i$, pour~$i=0,1$. La condition de recollement
est~$g_{01}s_1=s_0$ sur~$U_0\inter U_1$. Ecrivons $s_i=p_i/q_i$, où
$p_i,q_i\in\RR[x,y]$, avec $q_i\neq0$ sur~$U_i$ et $p_i,q_i$ premiers
entre eux, pour $i=0,1$.  La condition de recollement implique que
$pq_0p_1=p_0q_1$ sur~$\RR^2$. Comme~$p$ est irréductible, et comme
$q_1(c_0)\neq0$, le polynôme~$p$ divise~$p_0$. En particulier,
$p_0(c_1)=0$ et donc aussi~$s(c_1)=0$. 

Il s'ensuit que le faisceau régulier quasi-cohérent~$\SF$ sur~$\R^2$
n'est pas engendré par ses sections globales.

Pour compléter, montrons que le faisceau $k$-régulu
quasi-cohérent~$\SF^k=\SR^k\tensor\SF$ induit est bien engendré par
ses sections globales, pour~$k$ un entier naturel. On montre même
plus: le faisceau~$\SF^k$ est isomorphe à~$\SR^k$ en exhibant une section
globale de~$\SF^k$ qui ne s'annule jamais. Soit $s$ la section régulue
globale de~$\SF^k$ définie par
$$
s_0=\frac{1}{(x^2+y^2)^\ell}
\quad\text{et}\quad
s_1=\frac{(x^2+y^2)^\ell}{p},
$$ où~$\ell$ est le plus petit entier $\geq(k+1)/2$.  La fonction
$s_i$ est bien $k$-régulue sur~$U_i$ et non nulle, pour $i=0,1$. Le
faisceau $k$-régulu $\SF^k$ est donc isomorphe à~$\SR^k$. En particulier,
il est engendré par ses sections globales.
\end{ex}

Il suit du Théorème~A de Cartan version $k$-régulue, que le foncteur
``sections globales'' est exact sur les faisceaux $k$-régulus quasi-cohérents
sur~$\RR^n$. D'où la version $k$-régulue du Théorème~B de Cartan, cf. \cite[Théorème 1.3.1]{EGA3}. (On se réfère ici à EGAIII car le schéma $\Spec\SR^k(\R^n)$ n'est pas noethérien d'après la Proposition~\ref{prop.nonoeth} or la preuve donnée de \cite[Chap. III, Theorem 3.5]{Ha77} n'est valable que pour les schémas noethériens).

\begin{thm}
\label{thB}
Soient $n$~et $k$ des entiers naturels.  Soit~$\SF$ un faisceau
$k$-régulu quasi-cohérent sur~$\RR^n$. Alors,
$$
H^i(\RR^n,\SF)=0
$$
pour tout~$i\neq0$.\qed
\end{thm}

\subsection*{Variétés $k$-régulues affines}\label{p.dfn.regul.sing}

Soit $F$ un sous-ensemble $k$-régulument fermé de~$\R^n$.  La
\emph{topologie $k$-régulue} sur~$F$ est la topologie induite par la
topologie $k$-régulue de~$\R^n$.  Soit~$\SI$ le faisceau d'idéaux
de~$\SR^k$ des fonctions s'annulant sur~$F$. Le faisceau
quotient~$\SR^k/\SI$ possède un support égal à~$F$, et est donc un
faisceau sur~$F$, le \emph{faisceau des fonctions $k$-régulues}
sur~$F$, noté~$\SR_F^k$. Soit~$U$ un ouvert $k$-régulu de~$F$. Une
\emph{fonction $k$-régulue} sur~$F$ est une section du
faisceau~$\SR_F^k$. En particulier, une fonction $k$-régulue sur~$F$
est une section globale du faisceau~$\SR_F^k$ sur~$F$.

Lorsque $F\subset \RR^n$ est une variété réelle algébrique affine lisse, cette définition est a priori plus forte que la Définition~\ref{dfn.regulue}. En effet, tout élément de $\SR_F^k(F)$ est une fonction $k$-régulue au sens de \ref{dfn.regulue}. Pour $k=0$, la réciproque est vraie d'après \cite{Ko}. Pour un entier naturel $k\ne 0$, la question est ouverte.

\begin{prop}\label{prop.restric}
Soient $n$ un entier naturel et $k$ un entier surnaturel.  Soit~$F$ un
fermé $k$-régulu de~$\R^n$, et $U$ un ouvert $k$-régulu de~$F$. Soit
$V$ un ouvert $k$-régulu de~$\R^n$ tel que~$V\inter F=U$.  Si $f$ est
une fonction $k$-régulue sur~$U$, alors il existe une fonction
$k$-régulue~$g$ sur~$V$ dont la restriction à~$U$ est égal à~$f$.
\end{prop}

\begin{proof}
Montrons l'énoncé d'abord dans le cas où~$k$ est fini.  On a une suite
exacte de faisceaux quasi-cohérents
$$
0\ra \SI_{|V} \ra \SR_{|V}^k\ra (\SR_F^k)_{|V}\ra 0
$$
sur~$\R^n$. D'après la version $k$-régulue du Théorème B de Cartan,
on en déduit une suite exacte
$$
0\lra \SI(V)\ra \SR^k(V)\ra \SR_F^k(V)\ra0.
$$ Comme~$\SR_F^k(V)=\SR_F^k(U)$, il existe bien une fonction
$k$-régulue~$g$ sur~$V$ dont la restriction à~$U$ est égale à~$f$, pour
toute fonction $f$ $k$-régulue sur~$U$, lorsque $k$ est fini.

Il nous reste à montrer l'énoncé lorsque~$k=\infty$. Soit~$f$ une
fonction régulière sur~$U$. Il existe des fonctions polynomiales
$p$~et $q$ sur~$\R^n$ telles que~$f(x)=p(x)/q(x)$ pour tout~$x\in
U$. Soit~$s$ une fonction polynomiale sur~$\R^n$ dont l'ensemble des
zéros est égal à~$F$. L'ensemble des zéros de la fonction polynomiale
$q^2+s^2$ est contenu dans le fermé de Zariski~$G=F\setminus U$.  La
fonction définie sur~$\R^n\setminus G$ par
$$
g(x)=\frac{p(x)q(x)}{q^2(x)+s^2(x)}
$$ est bien régulière, et sa restriction à~$U$ est égale à~$f$.
\end{proof}

\begin{cor}
Soient $n$ un entier naturel et $k$ un entier surnaturel.  Soit~$F$ un
fermé $k$-régulu de~$\R^n$.  Si $f$ est une fonction $k$-régulue
sur~$F$, alors il existe une fonction $k$-régulue~$g$ sur~$\R^n$ dont
la restriction à~$F$ est égale à~$f$. Plus précisément, l'application
de restriction de~$\SR^k(\R^n)$ dans~$\SR^k(F)$ induit un isomorphisme
$$
\SR^k(\R^n)/\I(F)\iso \SR^k(F).\qed
$$
\end{cor}

\begin{rema}\label{rem.here}
D'après le corollaire précédent, les fonctions $k$-régulues telles que
nous venons de les définir sur des sous-ensembles fermés $k$-régulus
coïncident avec les fonctions «hereditarily rational» de
Koll\'ar~\cite{Ko}, lorsque~$k=0$.  L'exemple \cite[Ex.~2]{Ko} d'une
fonction rationnelle continue sur une surface singulière dont la
restriction à l'axe des $z$ n'est pas rationnelle n'est pas une fonction
régulue.
\end{rema}

Soit $F$ un sous-ensemble $k$-régulument fermé de~$\R^n$.  La
paire~$(F,\SR_F^k)$ est un espace localement annelé en
$\R$-algèbres. Une \emph{variété $k$-régulue affine} est un espace
localement annelé en $\R$-algèbres isomorphe à~$(F,\SR_F^k)$ pour un
certain fermé $k$-régulu $F$ de~$\R^n$, où~$n$ est un entier naturel.  Un
\emph{morphisme} entre variétés $k$-régulues affines est un morphisme
d'espaces localement annelés en $\R$-algèbres.

\begin{cor}
Soient $m$~et $n$ des entiers naturels et $k$ un entier surnaturel.
Soient $F\subseteq \R^n$ et $G\subseteq \R^m$ des fermés
$k$-régulus. Une application~$f\colon G\ra F$ est un morphisme de
variétés affines $k$-régulues si et seulement s'il existe des
fonctions $k$-régulues $f_1,\ldots,f_n$ sur~$\R^m$ telles que
$$
f(x)=(f_1(x),\ldots,f_n(x))
$$ pour tout~$x\in G$.\qed
\end{cor}

Soit $k$ un entier surnaturel. Soit~$(X,\SO)$ une variété $k$-régulue
affine.  Soit $f$ une fonction $k$-régulue sur~$X$, i.e., $f$ est une
section globale de~$\SO$. L'ensemble des zéros de~$f$ dans le spectre
$\Spec\SO(X)$ est noté~$\V(f)$, i.e.,
$$
\V(f)=\{p\in\Spec\SO(X)\suchthat f\in p\}\;.
$$ Son complémentaire est noté~$\U(f)$. Plus gén\'eralement, si $E$
est un sous-ensemble de~$\SO(X)$, l'ensemble des zéros communs
dans~$\Spec\SO(X)$ des éléments de~$E$ est noté~$\V(E)$, et son
complémentaire est~$\U(E)$. Plus précisément,
$$ \V(E)=\{p\in\Spec\SO(X)\suchthat E\subseteq p\}
$$
et
$$
\U(E)=\{p\in\Spec\SR^k(F)\suchthat E\not\subseteq p\}.
$$ Comme précédemment, les sous-ensembles de la forme~$\U(E)$ constituent une
topologie sur~$\Spec\SO(X)$~\cite[Chapitre~1]{EGA1}. L'espace
topologique~$\Spec\SO(X)$ est quasi-compact. Plus généralement, les
ouverts de la forme~$\U(f_1,\ldots,f_m)$ sont quasi-compacts.

Soit $x\in X$. On note~$\M_x$ l'idéal maximal de~$\SO(X)$ des
fonctions $k$-régulues sur~$X$ s'annulant en~$x$.
La Proposition~\ref{maximal} implique alors l'énoncé suivant:

\begin{prop}
Soit $k$ un entier surnaturel.  Soit $X$ une variété $k$-régulue
affine.  Soit $\M$ un idéal maximal de $\SO(X)$. Il existe un et un
seul~$x\in X$ tel que~$\M_x=\M$.\qed
\end{prop}

Comme dans le cas~$\R^n$, on a une injection
$$
\iota\colon X\lra\Spec\SO(X).
$$ 
Le Théorème~\ref{thbijferirrprem} se généralise alors ainsi:

\begin{thm}
Soit $k$ un entier naturel.  Soit $X$ une variété $k$-régulue affine.
L'application~$\iota$ ci-dessus induit une bijection entre l'ensemble
des sous-ensembles fermés de~$\Spec\SO(X)$ et l'ensemble des
sous-ensembles fermés de~$X$. Plus précisément, pour tout
sous-ensemble fermé~$F$ de~$X$ il existe un et un seul ensemble
fermé~$G$ de~$\Spec\SO(X)$ tel que~$F=\iota^{-1}(G)$.\qed
\end{thm}

Pour les ouverts cela donne l'énoncé suivant:

\begin{cor}
Soit $k$ un entier naturel.  Soit $X$ une variété $k$-régulue affine.
Soit $U\subseteq X$ un ouvert. Il existe un et un seul ouvert~$V$
de~$\Spec\SO(X)$ tel que $\iota^{-1}(V)=U$.\qed
\end{cor}

On a encore:

\begin{cor}
Soit $k$ un entier naturel.  Soit $X$ une variété $k$-régulue
affine. Tout ouvert de~$\Spec\SO(X)$ est de la forme~$\U(f)$, pour une
certaine fonction $k$-régulue sur~$X$.\qed
\end{cor}

Comme auparavant, on identifiera~$X$ avec son image
dans~$\Spec\SO(X)$ via l'application~$\iota$.

\begin{cor}
Soit $k$ un entier surnaturel. Soit $X$ une variété $k$-régulue
affine.  La restriction du faisceau structural sur~$\Spec\SO(X)$ à~$X$
s'identifie avec le faisceau~$\SO$.\qed
\end{cor}

Soit $M$ un $\SO(X)$-module. Le faisceau~$\tilde{M}$ sur le
schéma~$\Spec\SO(X)$ induit encore par restriction un
faisceau~$\tilde{M}$ sur~$X$. Pour un ouvert~$U$ de~$X$, on a
$$
\tilde{M}(U)=M_f
$$ lorsque~$k$ est fini, où~$f$ est une fonction $k$-régulue
sur~$X$ avec~$\D(f)=U$. 

Un faisceau \emph{$k$-régulu} sur~$X$ est un faisceau en
$\SO$-modules sur~$X$.  Soit~$\SF$ un faisceau $k$-régulu sur~$X$.  On
dira que le faisceau~$\SF$ est \emph{quasi-cohérent} s'il existe un
recouvrement de~$X$ par des ouverts $k$-régulus~$U_i$, $i\in I$, tels
que, pour tout~$i$, il existe un $\SO(U_i)$-module $M_i$ avec~$\tilde
M_i$ isomorphe à~$\SF_{|U_i}$.

La version $k$-régulue du Théorème~A de Cartan implique alors l'énoncé
suivant:

\begin{thm}\label{thAX}
Soit $k$ un entier naturel.  Soit $X$ une variété $k$-régulue affine.
Le foncteur qui associe à un $\SO(X)$-module $M$ le faisceau~$\tilde
M$ est une équivalence de catégories sur la catégorie des faisceaux
$k$-régulus quasi-cohérents sur~$X$. Le foncteur réciproque est le
foncteur ``sections globales'' $H^0$. En particulier, le
faisceau~$\tilde{M}$ est engendré par ses sections globales.\qed
\end{thm}

De même, on a une version $k$-régulue du Théorème~B de Cartan pour les
variétés $k$-régulues affines:

\begin{thm}\label{thBX}
Soit $k$ un entier naturel.  Soit $X$ une variété $k$-régulue affine.
Soit~$\SF$ un faisceau $k$-régulu quasi-cohérent sur~$X$. Alors,
$$
H^i(X,\SF)=0
$$
pour tout~$i\neq0$.\qed
\end{thm}

En particulier les théorèmes \ref{thAX} et \ref{thBX} s'appliquent pour tout $k\in \N$ lorsque $X$ est une variété réelle algébrique affine singulière. L'énoncé est alors en tout point identique au cas algébriquement clos ! Cf.~\cite[Chap. III]{Ha77}.

\section{Ensembles r\'egulument ferm\'es}

Le but de cette section est de déterminer les sous-ensembles
$k$-régulument fermés de $\R^n$ pour $k$ et $n$ des entiers naturels.
Ceci fait, si $F$ est un sous-ensemble $k$-régulument fermé de $\R^n$
(par exemple si $F$ est une sous-variété réelle algébrique de $\R^n$),
on saura caractériser les sous-ensembles $k$-régulument fermés de $F$
via la topologie induite.

\subsection*{Ensembles alg\'ebriquement constructibles}

On rappelle qu'un sous-ensemble semi-alg\'ebrique de $\R^n$ est 
alg\'ebriquement constructible s'il appartient \`a l'alg\`ebre 
bool\'eenne engendr\'e par les sous-ensembles Zariski ferm\'es de
$\R^n$. Ces ensembles forment une cat\'egorie constructible, 
de la m\^eme mani\`ere que les ensembles sym\'etriques par arcs. 
On renvoie \`a \cite{Paru-const,KurPar} pour une 
introduction aux cat\'egories constructibles.

Le but de cette section est de montrer que les sous-ensembles
$k$-r\'egulument fermés de $\R^n$ 
coïncident avec les sous-ensembles alg\'ebriquement constructibles
fermés de $\R^n$. On commence par deux r\'esultats pr\'ecisant la
place 
des points r\'eguliers ainsi que la dimension d'un ensemble 
$k$-r\'egulument ferm\'e vis-\`a-vis de son adh\'erence de Zariski.

Si $E$ est un sous-ensemble de $\R^n$, on notera dans la suite 
$\overline{E}^{\mathrm{eucl}}$ l'adhérence de $E$ dans $\R^n$ 
pour la topologie euclidienne.
\begin{thm}
\label{hypersurface} Soit $I$ un idéal premier de $\SRC(\R^n)$. On note
$J=I\cap \RR[x_1,\ldots ,x_n]$ et on suppose que $J$ est principal,
i.e. 
$J=(s)$ avec
$s\in\RR[x_1,\ldots ,x_n]$ irréductible. On note $V$ l'ensemble $\Z (J)$ et
$V_{\mathrm{reg}}$ l'ensemble des points lisses de $V$. Alors
$$\overline{V_{\mathrm{reg}}}^{\mathrm{eucl}}\subseteq \Z (I)
=\{ x\in\RR^n|\,\, I\subseteq \M_x\}=\{x\in\RR^n
|\,\,f(x)=0\, , \,\forall f\in I\}.$$
\end{thm}

\begin{proof}
Les fonctions de $\SRC (\R^n)$ étant continues, il est suffisant de montrer que 
$$
V_{\mathrm{reg}}\subseteq \Z (I)=\{ x\in\RR^n|\,\, I\subseteq \M_x\}\;.
$$
Pour $x\in V_{\mathrm{reg}}$, on doit montrer que $I\subseteq \M_x$.
Supposons l'existence de $f\in I$ telle que $f\not\in\M_x$, c'est-\`a-dire 
$f(x)\not= 0$. On écrit
$f=\frac{p}{q}$ avec $p,q\in\RR[x_1,\ldots,x_n]$ premiers entre
eux. On obtient alors
$$p=q.f\in I\cap \RR[x_1,\ldots ,x_n]=J=(s)$$ et donc
$p=r.s$ avec
$r\in\RR[x_1,\ldots ,x_n]$. Ceci implique que $p$ s'annule
identiquement sur $V$, i.e. $V\subseteq  \Z(p)$. 
Le Lemme \ref{reelintersection} nous montre que
$J=(s)$ est un idéal réel et par conséquent 
$\dim \,(V=\Z(s))=n-1$ \cite[Thm. 4.5.1]{BCR}.
On en déduit que $\dim\,\Z(p)=n-1$.
Comme $\Z(q)\subseteq \Z(p)$ puisque $f\in \SRC (\R^n)$, on a aussi $q(x)=0$ car $f(x)\not=0$. 
On sait que $\dim\,
\Z(q)\leq n-2$ par la
 Proposition
\ref{codim}. Comme $x\in V_{\mathrm{reg}}$, il existe un voisinage ouvert
semi-algébrique de $x$ dans $V\subseteq \Z(p)$ de dimension $n-1$
\cite[Prop. 3.3.10]{BCR}.
On note $A$ le semi-algébrique $\Z(p)\setminus\Z
(q)$. D'après ce qu'on a dit précédemment, le 
point $x$ est un point adhérent à $A$.
Le lemme de sélection des courbes
\cite[Thm. 2.5.5]{BCR} pour le semi-algébrique $A$ 
nous fournit une fonction semi-algébrique continue $h\colon
[0,1]\rightarrow \RR^n$ telle que $h(0)=x$ et $h(]0,1])\subseteq A$.
Par conséquent $f\circ h\colon [0,1]\rightarrow \RR$ est une fonction
semi-algébrique continue telle que $(f\circ h)(]0,1])=0$ et $(f\circ
h)(0)\not=0$, ce qui contredit la continuité de $f$.
\end{proof}

On évalue maintenant la dimension d'un fermé
régulu vu comme un ensemble semi-algébrique, c'est-\`a-dire la dimension de son
adhérence de Zariski.

\begin{prop}
\label{dimension}
Soit $I$ un idéal radical de $\SRC(\R^n)$ et
$J=I\cap \RR[x_1,\ldots ,x_n]$ sa trace sur les polynômes.
On note $V$ l'ensemble $\Z (J)$. Alors 
$$\overline{\Z (I)}^{\mathrm{Zar}}=V$$
où $\overline{\Z (I)}^{\mathrm{Zar}}$ est l'adhérence pour la topologie de
Zariski de $\Z (I)$. En particulier $\dim\, \Z(I)=\dim \, V$.
\end{prop}

\begin{proof} Posons $V'=\overline{\Z (I)}^{\mathrm{Zar}}$. Alors $V'$ est inclus dans $V$
  car $V$ est un fermé de Zariski contenant $\Z(I)$.

Supposons maintenant que $V'$ soit strictement inclus dans $V$. Alors
$\I (V')=\I (\Z(I))$ contient strictement $\I (V)$, sinon
$\Z(\I(V'))=V'=\Z(\I(V))=V$. Par conséquent il existe $p\in \I
(V')$ tel que $p\not\in \I (V)$. Mais $p$ est une fonction régulue qui
s'annule identiquement sur $\Z(I)$ et donc, par le Nullstellensatz régulu,
$$p\in
\I_{\SRC}(\Z(I))\cap\RR[x_1,\ldots ,x_n] =I\cap \RR[x_1,\ldots
,x_n]=J.$$ 
Par conséquent $p(V)=0$, et donc $p$ appartient \`a $\I (V)$, en contradiction avec le choix de $p$.
\end{proof}

Dans la cas du plan, on peut construire {\it \`a la main} une fonction r\'egulue qui s\'epare les points isol\'es de l'adh\'erence euclidienne des points r\'eguliers d'une courbe, d\'emontrant au passage que l'adh\'erence au sens de la topologie euclidienne des points r\'eguliers de la courbe est un ensemble r\'egulument ferm\'e.

\begin{prop}
\label{revhypersurface} Soit $I$ un idéal premier de $\SRC(\R^2)$. Soit
$J=I\cap \RR[x_1,x_2]$ et notons $V$ l'ensemble $\Z (J)$. On note aussi $V_{\mathrm{reg}}$
l'ensemble des points lisses de $V$ et 
$W=\overline{V_{\mathrm{reg}}}^{\mathrm{eucl}}$. Alors
$$W= \Z (I)=\{ x\in\RR^2|\,\, I\subseteq \M_x\}=\{x\in\RR^2
|\,\,f(x)=0\,\,\forall f\in I\}.$$
\end{prop}

\begin{proof}
Si $J=(0)$ alors $I=(0)$ et le résultat est évident.
Si $J$ est maximal, le Corollaire \ref{maximal} donne le
résultat voulu.

On suppose donc que $J$ est de hauteur $1$, c'est-\`a-dire
$J=(s)$ avec $s\in \RR[x_1,x_2]$ irréductible. Ainsi $V=\Z (s)$ est une
courbe alg\'ebrique affine irréductible. L'idéal $J$ étant réel, on a $\I
(\Z(J))=J$. L'inclusion $$W\subseteq \Z
(I)$$ provient du Théorème \ref{hypersurface}.
Comme $J\subseteq I$, on obtient par ailleurs $\Z(I)\subseteq\Z (J)=V$. En résumé, on a
les inclusions suivantes $$W\subseteq \Z
(I)\subseteq V.$$
Pour terminer la preuve, pour tout $x\in V\setminus W$ on doit trouver $f\in
I$ telle que $f(W)=0$ et vérifiant $f(x)\not=0$.

Soit $x\in V\setminus W$. Supposons $x$ \'egal \`a l'origine $O$ pour simplifier les notations. Alors $O$ est un
zéro isolé de $s$. Il existe donc un disque fermé $B$ centré en
$O$ tel que $O$ soit l'unique zéro de $s$ dans $B$. En utilisant l'inégalité
de {\Lbarre}ojasiewicz pour $x_1^2+x_2^2$ et $\frac{1}{s^2}$ dans $B$, il existe un
entier strictement positif $N$ tel que $\frac{( x_1^2+x_2^2)^N}{s^2}$
s'étende de façon continue par $0$ en $O$ dans $B$. 
On pose $f=\frac{s^2}{s^2+( x_1^2+x_2^2)^{2N}}$. La fonction $f$ est
clairement continue en dehors de $O$ et peut être étendue de manière
continue par $1$
en $O$ car 
$$\lim_{(x_1,x_2)\rightarrow O}\frac{( x_1^2+x_2^2)^{2N}}{s^2}=0.$$
Par conséquent $f\in \SR^0(\R^2)$ et la restriction de $f$ à $V_{\mathrm{reg}}$ est
identiquement nulle puisque la restriction de $f$ à $V\setminus
\{ O\}$ est identiquement nulle. Comme $f$ est continue
on a aussi $f(W)=0$. Il reste donc à montrer que
$f$ appartient \`a $I$. D\'ej\`a
$$(s^2+(( x_1^2+x_2^2)^N)^2).f\in I,$$ 
et puisque $I$ est
un idéal premier, au moins un des deux termes du produit appartient à
$I$. 
Supposons qu'il s'agisse du premier, c'est-\`a-dire
$$(s^2+((
x_1^2+x_2^2)^N)^2)\in I\cap \RR[x_1,x_2]=J.$$ 
Alors $( x_1^2+x_2^2)^N$ appartient \`a $J$ car $J$ est un idéal réel. L'idéal $J$ étant de plus premier, on en déduit
que
$x_1^2+x_2^2\in J$. Mais alors $x_1\in J$ et $x_2\in J$ en invoquant une
nouvelle fois la r\'ealit\'e de $J$. On en déduit que $J$ est
maximal, en contradiction avec notre hypoth\`ese. Par conséquent $f$ appartient \`a $I$, ce qui termine la preuve de la proposition.
\end{proof}

En particulier, on constate que les courbes $0$-r\'egulument ferm\'ees du plan sont des ensembles Zariski constructibles ferm\'es. On montre ci-dessous que c'est en fait le cas pour tout ensemble $k$-r\'egulu.

\begin{thm}\label{const} Soit $k$ un entier naturel.
Les sous-ensembles $k$-r\'egulument ferm\'es de $\R^n$ coincident 
avec les sous-ensembles alg\'ebriquement constructibles ferm\'es de $\R^n$.
\end{thm}

\begin{proof} On sait d\'ej\`a que tout ensemble $k$-r\'egulument
  ferm\'e de $\R^n$ est alg\'ebriquement constructible 
d'apr\`es le Corollaire \ref{coregfermeestconstructible}. 
On montre la r\'eciproque par r\'ecurrence sur la dimension. 

En dimension nulle, le r\'esultat est clair. 
Soit $C$ un sous-ensemble alg\'ebriquement constructible ferm\'e de
$\R^n$. On peut supposer que $C$ est irr\'eductible, autrement dit que son adh\'erence de Zariski $V$ est une sous-vari\'et\'e r\'eelle alg\'ebrique irr\'eductible de $\RR^n$, quitte \`a raisonner composante par composante. On note $Z$ la r\'eunion des composantes alg\'ebriquement constructibles de $V$ de dimension strictement plus petite que $V$. Ainsi $V=C\cup Z$. Il existe une fonction $k$-r\'egulue $g$ d\'efinie sur $\RR^n$ telle que $\Z(g)=Z$ par hypoth\`ese de r\'ecurrence. Par ailleurs, soit $f$ une fonction r\'eguli\`ere sur $\RR^n$ telle que $\Z(f)=V$. D'apr\`es le Lemme \ref{LojaregulU}, il existe un entier naturel $N$ tel que la fonction $\frac{g^N}{f^2}$, $k$-r\'egulue sur $\D(f)$, se prolonge par $0$ sur $\D(f) \cup Z\setminus C$ en une fonction $k$-r\'egulue. Posons
$$h=\frac{f^2}{f^2+g^{2N}}.$$
Alors la fonction $h$ est $k$-r\'egulue sur $\D(g)$, et s'annule sur $C \cap \D(g)$. On la prolonge par $1$ sur $Z$. Elle est alors $k$-r\'egulue en dehors de $Z \cap C$ par choix de $N$. L'ensemble $Z\cap C$ est alg\'ebriquement constructible de dimension strictement plus petite que $C$, il existe donc une fonction $k$-r\'egulue $l$ telle que $Z\cap C=\Z(l)$ par hypoth\`ese de r\'ecurrence. D'apr\`es le Lemme \ref{LojaregulRn}, il existe alors un entier $N'$ tel que la fonction $l^{N'}h$ soit $k$-r\'egulue sur $\RR^n$. Elle satisfait $\Z(l^{N'}h)=C$, ce qui fait de $C$ un ensemble $k$-r\'egulument ferm\'e.
\end{proof}

\begin{cor}\label{cor-const} Pour $k$ et $k'$ des entiers naturels, les topologies k-r\'egulue et $k'$-r\'egulue sont \'equivalentes.
\end{cor}

\subsection*{D\'ecomposition en ensembles sym\'etriques par arcs}

On sait qu'un sous-ensemble $k$-r\'egulument ferm\'e de $\R^n$ est un sous-ensemble alg\'ebriquement constructible ferm\'e d'apr\`es le Th\'eor\`eme \ref{const}. En particulier, un ensemble $k$-r\'egulument ferm\'e irr\'eductible est \'egal \`a la composante alg\'ebriquement constructible de dimension maximale de son adh\'erence de Zariski. 
Malheureusement, il est difficile de d\'ecrire g\'eom\'etriquement cette composante.
Le but de cette section est de d\'ecrire cette composante alg\'ebriquement constructible de dimension maximale en terme d'ensembles sym\'etriques par arcs. 

\begin{prop}
\label{annulregulue} Soit $V$ une sous-variété alg\'ebrique irréductible de $\RR^n$
de dimension $d$. Soit $S$ un sous-ensemble semi-alg\'ebrique de $V$ de dimension $d$ et soit $f\in\SRC(\RR^n)$
s'annulant identiquement sur $S$. Alors $f$ s'annule identiquement sur $V_{\mathrm{reg}}$ et donc sur la réunion $W$ des
composantes $\AR$-irréductibles de dimension maximale de $V$.
\end{prop}

Pour autant, il se peut que $f$ s'annule sur strictement plus que $W$ comme on le verra dans l'exemple \ref{c-ex}.

\begin{proof}[D\'emonstration de la Proposition \ref{annulregulue}]
La fonction $f$ reste r\'egulue en restriction \`a $V$ (cf. Remarque
\ref{rem.here}). Il existe donc une composition d'\'eclatements \`a
centres lisses $\phi\colon\tilde V\rightarrow V$, avec $\tilde V$ une
vari\'et\'e r\'eelle alg\'ebrique irr\'eductible non singuli\`ere,
telle que $f\circ\phi$ soit régulière sur $\tilde V$ d'apr\`es le Th\'eor\`eme \ref{threguliereaeclt}.

Comme $S$ est de dimension $d$, la fonction régulière $f\circ\phi$ s'annule sur
un sous-ensemble Zariski dense de $\tilde V$, donc sur $\tilde V$ qui est irr\'eductible. Alors $f$ s'annule sur les points r\'eguliers de $V$, donc aussi sur $W$ car $Z(f)$ est un ensemble sym\'etrique par arcs.
\end{proof}

Soit $E$ un ensemble semi-algébrique de $\RR^n$. D'apr\`es \cite[Cor. 2.15]{KurPar}, l'intersection des ensembles
symétriques par arcs contenant $E$ est un ensemble symétrique par arcs de m\^eme dimension que $E$, et noté
$\overline{E}^{\AR}$. De même, on note $\overline{E}^{C}$ l'adhérence
de $E$
pour la topologie algébriquement constructible (identique à la topologie régulue comme on vient de le voir).

On généralise maintenant le Théorème \ref{hypersurface} au cas non principal.
 
\begin{thm}
\label{hypersurfacegen}
Soit $I$ un idéal premier de $\SRC (\R^n)$. On note
$J=I\cap \RR[x_1,\ldots ,x_n]$ et $V$ l'ensemble $\Z (J)$. On note aussi
$V_{\mathrm{reg}}$ l'ensemble des points lisses de $V$. Alors 
$$
\Z (I)=\overline{V_{\mathrm{reg}}}^{C}
$$
est l'unique composante $C$-irréductible de dimension maximale $d$ de $V$. De plus, si $W$ désigne la réunion
des composantes $\AR$-irréductibles de $V$ de dimension $d$, alors
$$
\overline{V_{\mathrm{reg}}}^{\AR}=W\subseteq \Z (I)\;.
$$
\end{thm}

\begin{proof}
D'apr\`es le Th\'eor\`eme \ref{const}, on sait déjà que $\Z (I)$ est un ensemble $C$-irréductible fermé, c'est-à-dire un fermé régulu irréductible. On décompose $V=\cup_{i=1}^{s}W_i \cup\cup_{j=1}^{t}Z_j$ en composantes $C$-irréductibles fermées où les $W_i$ sont de dimension maximale $d$ et $\dim Z_j< d$ pour $j\in \{1,\ldots,t\}$. D'après la Proposition~\ref{dimension}, l'ensemble $\Z(I)$ est $C$-irréductible de dimension $d$. Par conséquent, on peut supposer que $\Z(I)=W_1$ et il existe $f_1\in \SRC (\R^n)$ telle que $W_1=\Z(f_1)$. La fonction $f_1$ s'annule en fait identiquement sur $V_{\mathrm{reg}}$ (Proposition~\ref{annulregulue}). On a donc 
$$
V_{\mathrm{reg}}\subset \Z(f_1)=\Z(I)\;.
$$

Si $s\geq 2$, on écrit $W_2=\Z(f_2)$ avec $f_2\in \SRC (\R^n)$ et de la même manière on obtient $V_{\mathrm{reg}}\subset \Z(f_2)=W_2$. Donc $V_{\mathrm{reg}}\subset W_1\cap W_2$ ce qui est impossible car $V_{\mathrm{reg}}$ est de dimension $d$. On en déduit que $\Z(I)$ est l'unique composante $C$-irréductible de dimension maximale de $V$ et finalement que $\Z (I)=\overline{V_{\mathrm{reg}}}^{C}$.

Remarquons que l'adhérence symétrique par arcs
$\overline{V_{\mathrm{reg}}}^{\AR}$ est égale à $W$ d'après
\cite[Thm. 2.21, Lem. 2.22]{KurPar}. Comme $V_{\mathrm{reg}}\subset \Z(I)$, on en tire $
\overline{V_{\mathrm{reg}}}^{\AR}=W\subseteq \Z (I)
$.

\end{proof}

On poursuit l'\'etude des relations entre les fonctions r\'egulues sur une vari\'et\'e r\'eelle alg\'ebrique et celles sur la r\'eunion de ses composantes $\AR$-irréductibles de dimension maximale.

\begin{prop}
\label{hypersurface2}  
Soit $V$ une sous-variété alg\'ebrique irréductible de $\RR^n$. 
On note $W=\overline{V_{\mathrm{reg}}}^{\AR}$ la réunion des
composantes $\AR$-irréductibles de dimension maximale de $V$. Alors
$$\I_{\SRC} (W)=\I_{\SRC}(V_{\mathrm{reg}})$$
est un idéal premier de $\SRC (\R^n)$.
\end{prop}

\begin{proof}
Si $f\in \I_{\SRC}(V_{\mathrm{reg}})$ alors $f\in\I_{\SRC} (W)$ car $\Z(f)$ est un
ensemble symétrique par arcs. On a donc $\I_{\SRC} (W)=\I_{\SRC}(V_{\mathrm{reg}})$.

Soient $f$ et $g$ deux fonctions régulues telles que le produit $f.g$
s'annule identiquement sur $V_{\mathrm{reg}}$. Les fonctions $f$ et $g$ restent r\'egulues en restriction \`a $V$ d'apr\`es la Remarque \ref{rem.here}, il existe donc une composition d'\'eclatements \`a centres lisses $\phi\colon\tilde V\rightarrow V$ 
telle que $f\circ\phi$ et $g\circ\phi$ soient régulières sur $\tilde V$, avec $\tilde V$ une vari\'et\'e r\'eelle alg\'ebrique irr\'eductible non singuli\`ere, d'apr\`es le Th\'eor\`eme \ref{threguliereaeclt}. Le
produit de fonctions régulières $(f\circ\phi).(g\circ\phi)$ 
s'annule identiquement sur la variété irréductible $\tilde V$. 
On peut donc conclure que, par exemple, $f\circ \phi$ s'annule identiquement sur 
$\tilde V$ et donc $f$ s'annule identiquement sur $V_{\mathrm{reg}}$
i.e. $f\in \I_{\SRC}(V_{\mathrm{reg}})$, ce qui termine la preuve.
\end{proof}

\begin{prop}
\label{irred}
Soit $V$ une sous-variété alg\'ebrique irréductible de $\RR^n$. On note
$W=\overline{V_{\mathrm{reg}}}^{\AR}$ 
la réunion des
composantes $\AR$-irréductibles de dimension maximale de $V$.
Si $V=W$ (en particulier si $V$ est lisse), alors $V$ est un ensemble
régulument fermé et régulument irréductible.
\end{prop}

\begin{proof}
Comme les fonctions régulières sont régulues, $V$ est un fermé régulu.
De plus $V=W$ donc l'idéal $\I_{\SRC} (W)=\I_{\SRC}(V)$ est un idéal premier par 
la Proposition \ref{hypersurface2}. 

Supposons que $V$ soit la réunion $V_1\cup
V_2$ de deux
fermés régulus non-vides. D'apr\`es le Théorème \ref{thmradprinc}, il
existe des fonctions régulues $f_1,f_2$ telles que $V_1=\Z(f_1)$ et
$V_2=\Z (f_2)$. Alors le produit $f_1.f_2$ appartient \`a l'id\'eal premier $\I_{\SRC}(V)$ et on en
déduit que $V\subseteq V_1$ ou $V\subseteq V_2$.
\end{proof}

 Soit $V\subset \RR^n$ une sous-variété alg\'ebrique irréductible. On note
$W=\overline{V_{\mathrm{reg}}}^{\AR}$ 
la réunion des
composantes $\AR$-irréductibles de dimension maximale de $V$. La même
preuve que celle de la proposition précédente montre que $W$ est un ensemble
régulument irréductible. Mais ce n'est pas en g\'en\'eral un ensemble
régulument fermé.

\begin{ex}\label{c-ex} 
On considère la sous-variété alg\'ebrique irréductible $V$ de
$\RR^4=\RR^2\times\RR^2$ 
de dimension~$2$ donnée par les équations 
$$
\begin{cases}
(x+2)(x+1)(x-1)(x-2)+y^2=0\;;\\
u^2=xv^2\;.
\end{cases}
$$
Cette vari\'et\'e est un ensemble r\'egulument irr\'eductible.

\begin{figure}[ht]
\centering
\includegraphics[height =3cm]{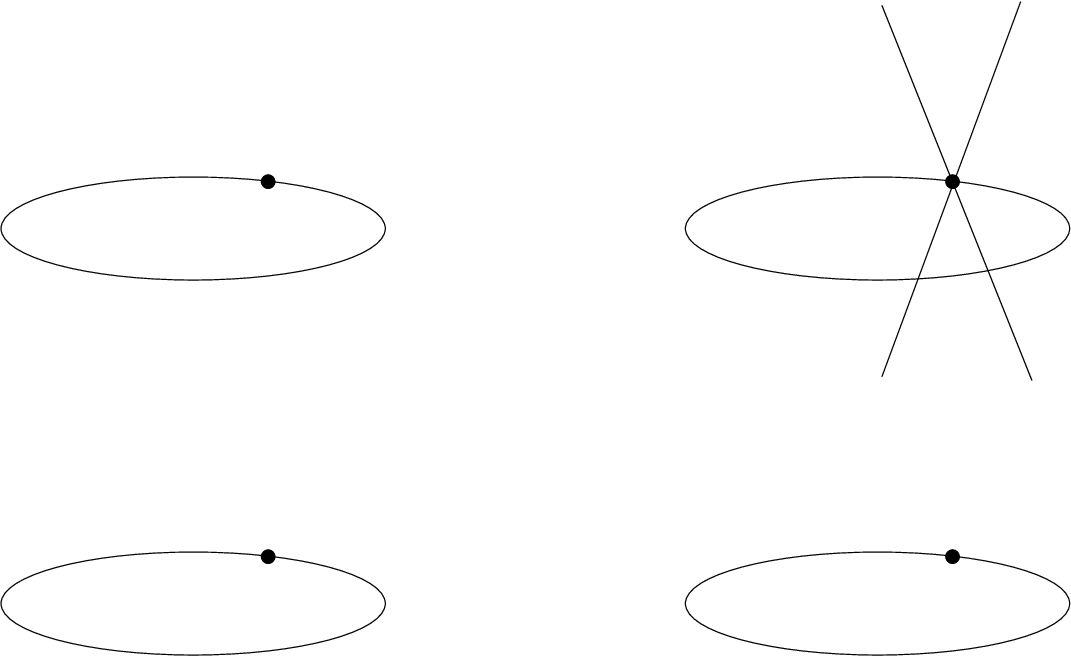}
\caption{Un ferm\'e r\'egulument irr\'eductible.}
        \label{fig.contre-exemple}
\end{figure}

Sur la figure, on a représenté la fibre de $(x,y,u,v)\mapsto (x,y)$ au
dessus de chaque point des deux composantes connexes de  la courbe
hyperelliptique d'équation $(x+2)(x+1)(x-1)(x-2)+y^2=0$. Au dessus de
l'un des ovales la fibre est un point, alors qu'au dessus de l'autre
ovale, la fibre est form\'ee de deux droites s\'ecantes.

La variété $V$ possède deux composantes connexes $W$ et $Z$, avec $W$ l'unique composante symétrique par arcs irréductible de dimension $2$ de
$V$. Le lieu singulier de $V$ est la courbe hyperelliptique $C$
d'équations $(x+2)(x+1)(x-1)(x-2)+y^2=0$ et $u=v=0$, possédant deux
composantes connexes $Z$ et $Y$ avec $Y\subset W$. On suppose que
$W$ est un fermé régulu. Il existe donc $f\in\SRC (\R^4)$
telle que $\Z (f)=W$ et par conséquent $f$ s'annule identiquement sur
$Y$ mais pas sur $Z$. Une telle fonction ne peut exister par la
Proposition \ref{annulregulue} en considérant sa restriction à $C=Z\cup Y$.
\end{ex}

Les r\'esultats pr\'ec\'edents permettent de d\'eterminer la dimension de Krull de
l'anneau $\SR^k(\R^n)$ pour $k\in\N$. 

\begin{prop}
\label{dimensionregulue} Soit $k\in\N$. L'espace topologique $\R^n$ est de dimension
$n$ pour la topologie régulue. L'anneau $\SR^k(\R^n)$ est par conséquent
de dimension de Krull $n$.
\end{prop}

\begin{proof}
Soit $$W_0\subsetneq W_1\subsetneq\ldots\subsetneq W_m\subseteq \R^n$$
une suite d'inclusions strictes de fermés régulus irréductibles avec
$m>n$. Pour $i=0,\ldots, m$, on note $I_i=\I_{\SRC}(W_i)$,
$J_i=I_i\cap \R [x_1,\ldots,x_n]$ et
$V_i=\overline{W_i}^{Zar}$. On a $\dim\, W_i=\dim\, V_i$ et de plus 
$V_i$ est
irréductible pour la topologie de Zariski car $V_i=\Z (J_i)$ d'après
la Proposition \ref{dimension} (l'idéal $J_i$ est bien un
idéal premier). On a aussi
$W_i=\overline{(V_i)_{reg}}^{C}$ d'après le Théorème \ref{hypersurfacegen}.
On obtient par conséquent une suite d'inclusions strictes de fermés de
Zariski irréductibles 
$$V_0\subsetneq V_1\subsetneq\ldots\subsetneq V_m\subsetneq \R^n$$
avec $m>n$, ce qui est impossible.
La dimension régulue de $\R^n$ est donc
inférieure à $n$.
Il est clairement possible de construire une suite d'inclusions strictes
$$V_0\subsetneq V_1\subsetneq\ldots\subsetneq V_n= \R^n$$ de fermés de
Zariski irréductibles lisses. D'après la Proposition \ref{irred}, les
$V_i$ sont des fermés régulus irréductibles ce qui prouve
que la dimension régulue de $\R^n$ est bien $n$. 
Soit $k\in\NN$. On rappelle que les topologies $k$-régulue et
$0$-régulues coincident.
En utilisant le Nullstellensatz $k$-régulu, on en déduit que la dimension
de Krull de $\SR^k (\R^n)$ est $n$.
\end{proof}

\begin{exem}\label{exem.umbrellas}
On considère trois exemples  (les deux premiers sont classiques).

\begin{itemize}

 \item[$\bullet$] Le parapluie de Whitney~:\par
Considérons la sous-variété $V$ d'équation $zx^2=y^2$ dans
$\R^3$. Le parapluie $V$ est un ensemble régulument fermé
et vérifie la relation
$V=\overline{V_{\mathrm{reg}}}^{\AR}$. Par
la Proposition \ref{irred}, c'est un ensemble régulument fermé et irréductible.
\begin{figure}[ht]
\centering
\includegraphics[height =6cm]{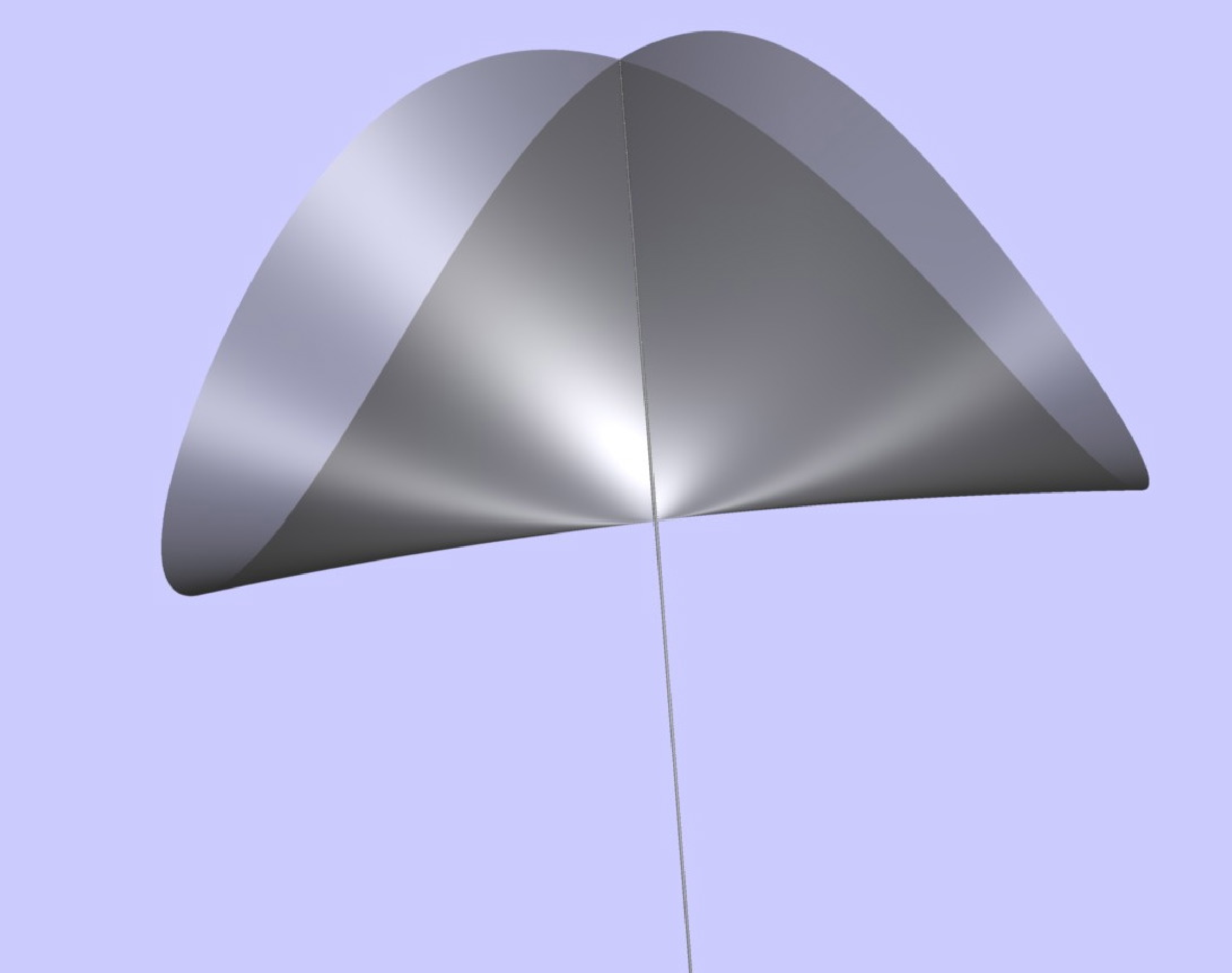}
\caption{Parapluie de Whitney.}
        \label{fig.whitney}
\end{figure}

\item[$\bullet$] Le parapluie de Cartan~:\par
Considérons la sous-variété $V$ d'équation $z(x^2+y^2)=x^3$ de $\R^3$. La
décomposition de $V$ en ensembles $\AR$-irréductibles est $W\cup Z$ où
$W=\overline{V_{\mathrm{reg}}}^{\AR}$ est la toile du parapluie et $Z$ est le
manche. L'ensemble $W$ est régulument fermé car $W=\Z (f)$ avec
$f=z-\frac{x^3}{x^2+y^2}$.
Par la Proposition \ref{hypersurface2}, $W$
est régulument irréductible. L'ensemble $Z$ est un ensemble
régulument fermé car $Z=\Z(x^2+y^2)$. L'ensemble $V$ est donc un
ensemble régulument fermé et non-régulument irréductible.
\begin{figure}[ht]
\centering
\includegraphics[height =6cm]{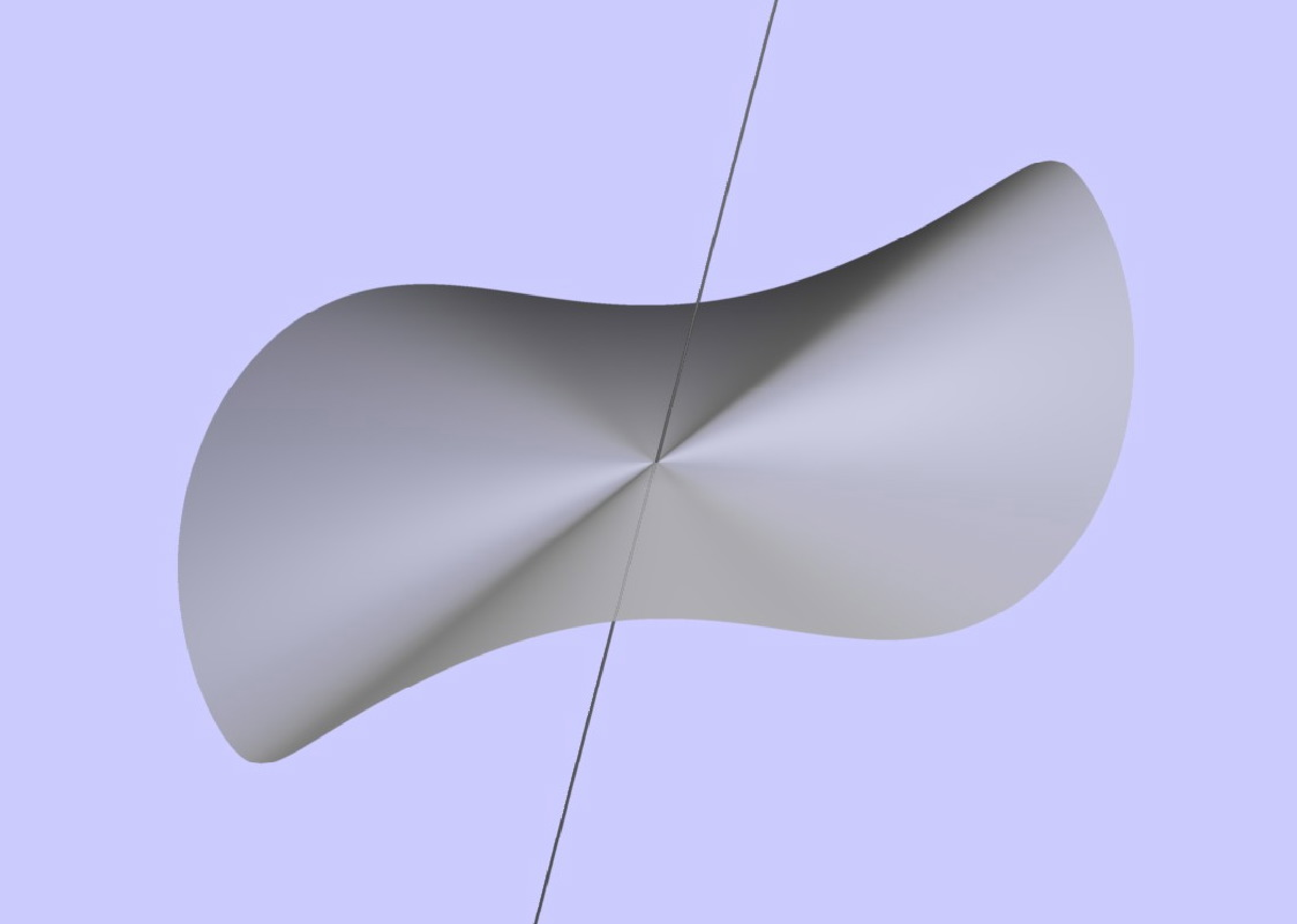}
\caption{Parapluie de Cartan.}
        \label{fig.cartan}
\end{figure}

\item[$\bullet$] Un parapluie cornu~:\par
Considérons la sous-variété $V$ d'équation 
$s(x,y,z)=x^2+y^2((y-z^2)^2+yz^3)=x^2+y^4+y^2z^4+y^3z^3-2y^3z^2=0$ 
de $\R^3$. La
décomposition de $V$ en ensembles $\AR$-irréductibles est $W\cup Z$ où
$W=\overline{V_{\mathrm{reg}}}^{\AR}$ est la corne du parapluie et $Z$ est le
manche. L'ensemble $W$ est régulument fermé car $W=\Z (f)$ avec
$f=z^2\frac{s(x,y,z)}{x^2+y^4+y^2z^4}$. 
Par la Proposition \ref{hypersurface2}, $W$
est régulument irréductible. Le manche $Z$ est un ensemble
régulument fermé car $Z=\Z(x^2+y^2)$. Comme pour le parapluie de
Cartan, l'ensemble $V$ est donc un
ensemble régulument fermé et non-régulument irréductible.
\begin{figure}[ht]
\centering
\includegraphics[height =6cm]{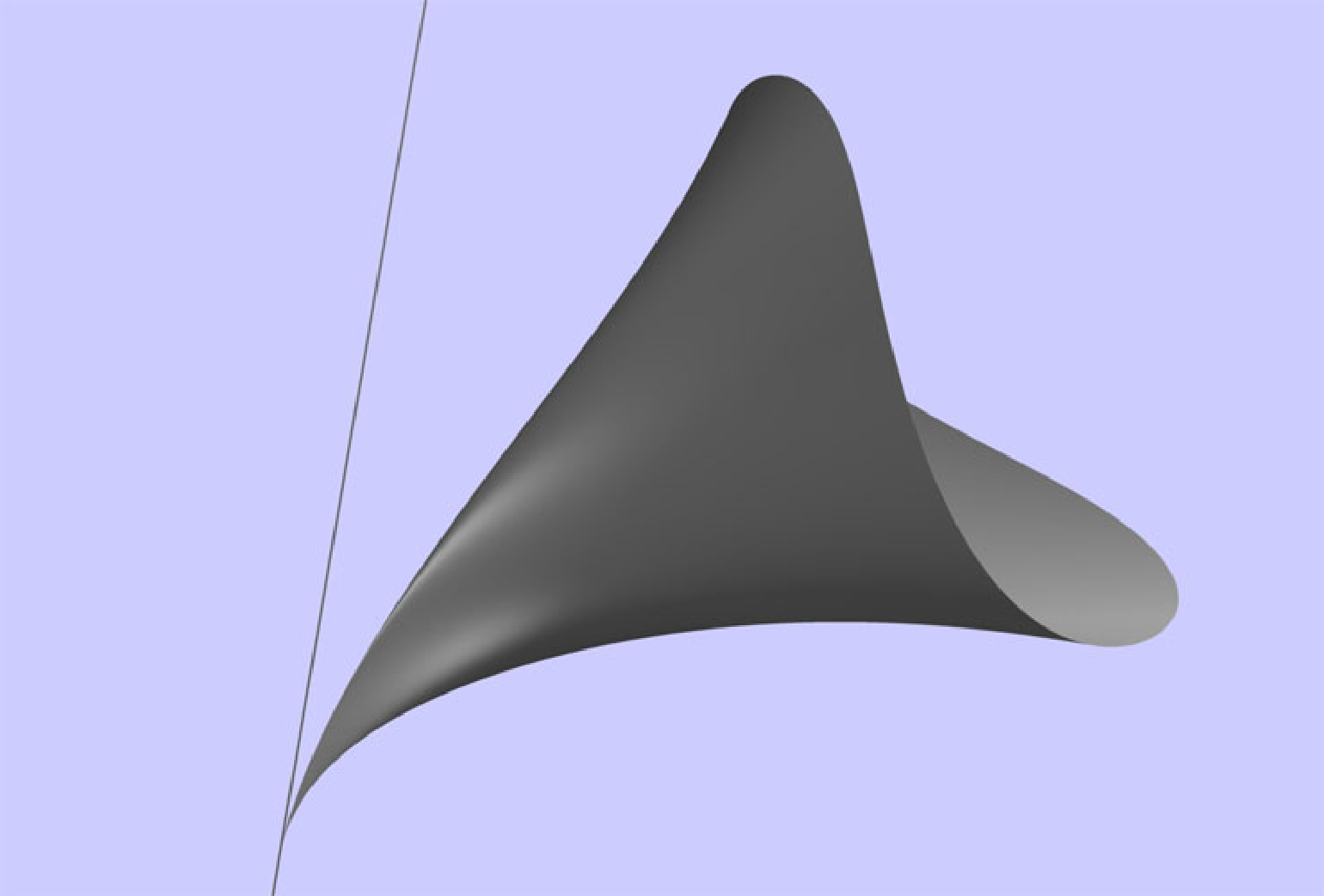}
\caption{Un parapluie cornu.}
        \label{fig.horned}
\end{figure}

\end{itemize}
\end{exem}

On énonce maintenant un théorème qui résume ce qui précède et qui via le Nullstellensatz, est équivalent au Théorème~\ref{hypersurfacegen}.
\begin{thm}
\label{thmcaracterisation}
Les fermés irréductibles régulus de $\R^n$ sont les ensembles du type
$$\overline{V_{\mathrm{reg}}}^{C}$$
où $V_{\mathrm{reg}}$ désigne l'ensemble des points lisses d'une
variété réelle algébrique irréductible $V\subset \R^n$.
\end{thm}

Le r\'esultat suivant permet de construire 
de fa\c con algorithmique le plus petit ferm\'e alg\'ebriquement
constructible contenant les points r\'eguliers d'une vari\'et\'e
r\'eelle alg\'ebrique affine et donc de caractériser les fermés irréductibles
régulus de $\R^n$.

\begin{prop}\label{sym-necessaire}
Soit $V\subset \RR^n$ une sous-vari\'et\'e r\'eelle alg\'ebrique irr\'eductible. Soit $W$ la r\'eunion des composantes sym\'etriques par arcs de dimension maximale de $V$. Soit $Z$ une composante sym\'etrique par arc irr\'eductible de dimension strictement plus petite que $V$. Si
$$\dim \overline Z^{Zar} \cap W =\dim Z,$$
alors $Z$ est dans $\overline W^{C}$.
\end{prop}

\begin{rem} La r\'eciproque de la Proposition \ref{sym-necessaire} est fausse, voir l'Exemple \ref{ex-algo}.
\end{rem}

\begin{proof}[D\'emonstration de la Proposition \ref{sym-necessaire}]
Notons que l'adh\'erence de Zariski de $Z$ forme une 
vari\'et\'e r\'eelle alg\'ebrique irr\'eductible de m\^eme dimension que $Z$.
Soit $f$ une fonction r\'egulue d\'efinie sur $\RR^n$ 
et s'annulant sur $W$. Par hypoth\`ese, la fonction $f$ 
s'annule sur un ensemble semi-alg\'ebrique $S=\overline Z^{Zar} \cap
W$ de m\^eme dimension que $\overline Z^{Zar}$, donc sur $Z$ aussi
d'apr\`es la Proposition \ref{annulregulue}. 
Par cons\'equent $Z$ appartient \`a l'adh\'erence 
pour la topologie alg\'ebriquement constructible de $W$.
\end{proof}

On d\'ecrit maintenant de mani\`ere algorithmique l'unique composante alg\'ebriquement constructible irr\'eductible de dimension maximale d'une vari\'et\'e r\'eelle alg\'ebrique affine en termes d'ensembles sym\'etriques par arcs irr\'eductibles. Soit $V\subset \RR^n$ une sous-vari\'et\'e r\'eelle alg\'ebrique affine irr\'eductible de dimension $d$. On note
$$V=\cup_{i\in I_d}W_i \cup_{i=0}^{d-1} \cup_{j\in I_i} Z_j^i$$
la d\'ecomposition de $V$ en sous-ensembles sym\'etriques par arcs
irréductibles, avec $\dim Z_j^i=i$ pour $j\in I_i$.
On note $W$ la r\'eunion des composantes sym\'etriques par arcs de dimension maximale de $V$.

 Soit $j\in I_{d-1}$.
\begin{enumerate}
\item Si $Z_j^{d-1}$ satisfait la condition de la Proposition \ref{sym-necessaire}, alors $Z_j^{d-1} \subset \overline W^{C}$. 
\item Sinon, $Z_j^{d-1}$ n'est pas tout entier contenu dans $\overline
  W^{C}$. N\'eanmoins, un sous-ensemble sym\'etrique par arcs de
  $Z_j^{d-1}$ peut \^etre contenu dans $\overline
  W^{C}$. L'intersection $Z_j^{d-1} \cap W$ est un ensemble
  sym\'etrique par arcs de dimension au plus $d-2$, et on rajoute les
  composantes sym\'etriques par arcs irr\'eductibles de son
  adh\'erence de Zariski $\overline {Z_j^{d-1} \cap W}^{Zar}$ \`a
  celles de dimension au plus $d-2$ de $V$. 
Proc\'edant de m\^eme pour tout $j\in I_{d-1}$, on obtient ainsi une 
r\'eunion d'ensembles sym\'etriques par arcs irr\'eductibles de 
dimension au plus $d-2$
$$\cup_{i=0}^{d-2} \cup_{j\in I_i^2} Z_j^i,$$
o\`u $I_i^2$ est un ensemble d'indices fini contenant $I_i$.
\end{enumerate}

Notons $W^2$ la r\'eunion de $W$ et des composantes $Z_j^{d-1}$, 
pour $j \in I_{d-1}$,  satisfaisant la condition de la 
Proposition \ref{sym-necessaire}.
On renouvelle maintenant les op\'erations $(1)$ et $(2)$ avec $W^2$ 
\`a la place de $W$ et en consid\'erant les indices $j$ appartenant
\`a $I_{d-2}^2$. Ainsi, on construit pas \`a pas, en au plus $d-1$
\'etapes, 
l'adh\'erence pour la topologie alg\'ebriquement constructible de $W$.

\begin{ex}\label{ex-algo} On modifie l\'eg\`erement l'exemple
  \ref{c-ex} 
de la fa\c con suivante. On considère la sous-variété alg\'ebrique 
irréductible $V'$ de
$\RR^6$ 
de dimension~$3$ donnée par les équations 
$$
\begin{cases}
(x+2)(x+1)(x-1)(x-2)(x-4)(x-5)+y^2=0\;;\\
w^2+u^2=xv^2\;;\\
t^2=(x-3)w^2\;.
\end{cases}
$$
Notons $C$ la courbe donn\'ee par les \'equations de $V'$ en faisant $t=u=v=w=0$.
La d\'ecomposition en sous-ensembles sym\'etriques par arcs irr\'eductibles de $V'$ est de la forme $V'=Z'_1\cup Z'_2 \cup W'$ avec $Z_1'$ l'ovale de $C$ contenant le point de coordonn\'ees $(-2,0,0,0,0,0)$, $Z_2'$ la surface contenant l'ovale de $C$ contenant le point de coordonn\'ees $(1,0,0,0,0,0)$, et $W'$ la partie de dimension trois contenant le dernier ovale de $C$.

\begin{figure}[ht]
\centering
\includegraphics[height =3cm]{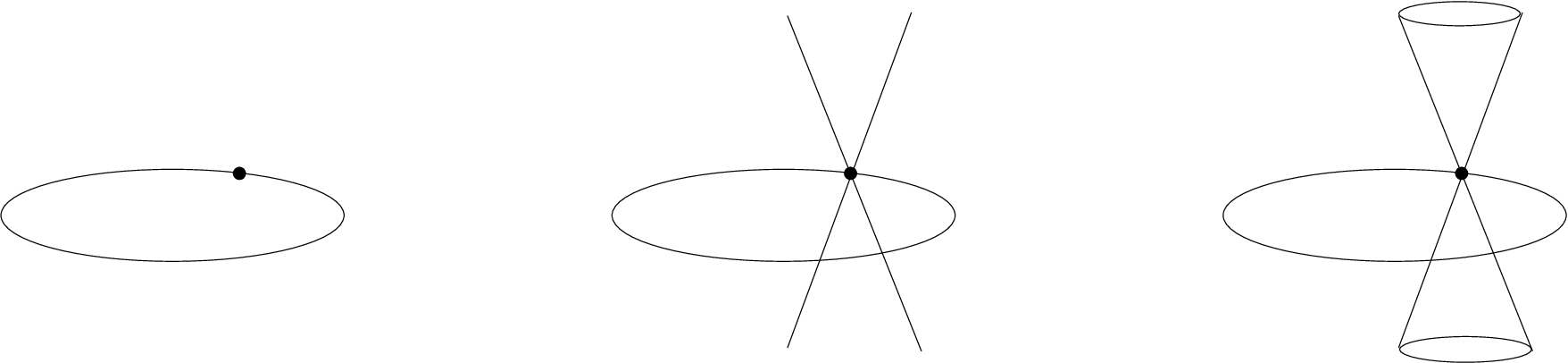}
\caption{Une chaine d'ensembles sym\'etriques par arcs.}
        \label{exem-algo}
\end{figure}

Soit $T\in \RR[x,y,t,u,v,w]$ l'\'equation d'un tore contenant en son int\'erieur l'ovale $C\cap W'$  (c'est-\`a-dire l'ovale de droite sur le dessin). On suppose de plus que $T$ est n\'egatif \`a l'int\'erieur du tore et positif \`a l'ext\'erieur (cf. Example 3.2.8 dans \cite{BCR}). On d\'efinit alors $V$ dans $\RR^7$, en ajoutant une variable $s$, par les \'equations de $V'$ auxquelles on rajoute l'\'equation $s^2=T$.

La vari\'et\'e $V$ ainsi obtenue forme un rev\^etement double de $V'$ \`a l'ext\'erieur du tore, et efface la partie \`a l'int\'erieur du tore. Notons $W$ la partie sym\'etrique par arcs de dimension maximale de $V$. 
Soit $Z_1$ une composante sym\'etrique par arcs de dimension un de $V$ s'envoyant sur l'ovale $Z_1'$ (de gauche sur le dessin) de la courbe $C$ par le rev\^etement. L'adh\'erence de Zariski de $Z_1$, form\'ee de quatre ovales vivant au dessus des ovales $Z_1'$ et $Z_2'\cap C$ de la courbe $C$, ne rencontre pas $W$ par construction de $V$. Ainsi $Z_1$ ne satisfait pas l'hypoth\`ese de la Proposition \ref{sym-necessaire}. Pour autant, $Z_1$ est bien dans l'adh\'erence pour la topologie alg\'ebriquement constructible de $W$. En effet, les deux composantes sym\'etriques par arcs de dimension deux de $V$ satisfont les hypoth\`eses de la Proposition \ref{sym-necessaire} puisque $\overline {Z_2'}^{Zar}=V'\cap \{t=0\}$, et donc font partie de $\overline W^{C}$. Notons $Z_2$ leur r\'eunion. Alors l'adh\'erence de Zariski de $Z_1$ rencontre $Z_2$ le long d'une courbe, donc d'apr\`es la Proposition \ref{sym-necessaire}, $Z_1$ est dans l'adh\'erence pour la topologie alg\'ebriquement constructible de $Z_2$, donc dans celle de $W$.
\end{ex}

\end{document}